\documentclass{amsart}
\usepackage{hyperref}

\usepackage{amsmath,amssymb,amsthm,amsfonts,amsrefs}

\usepackage{color}

\numberwithin{equation}{section}

\newtheorem{theorem}{Theorem}[section]
\newtheorem{lemma}[theorem]{Lemma}

\newtheorem{proposition}[theorem]{Proposition}
\theoremstyle{definition}

\theoremstyle{remark}
\newtheorem{remark}[theorem]{Remark}

\newcommand{\dis}{\displaystyle}
\renewcommand{\dim}{3}
\newcommand{\rmi}{{\rm i}}
\newcommand{\rmre}{{\rm Re}}

\newcommand{\R}{\mathbb{R}}

\newcommand{\threed}{{\mathbb R}^{\dim}}
\newcommand{\semiG}{\mathbb{A}}
\newcommand{\highG}{\Theta}
\newcommand{\highB}{\Lambda}

\newcommand{\sourceG}{g}
\newcommand{\Jcont}{j}
\newcommand{\wN}{\ell}
\newcommand{\wE}{j}
\newcommand{\wB}{b}
\newcommand{\solU}{f}
\newcommand{\CHd}{\dot{\mathcal{H}}}

\newcommand{\FM}{\mu}
\newcommand{\vel}{v}
\newcommand{\nsm}{|}

\newcommand{\FP}{\mathbf{P}}

\newcommand{\FL}{L}

\newcommand{\FI}{\mathbf{I}}

\newcommand{\CE}{\mathcal{E}}
\newcommand{\CF}{\mathcal{F}}
\newcommand{\CH}{\mathcal{H}}

\newcommand{\CL}{\mathcal{L}}
\newcommand{\CN}{\mathcal{N}}
\newcommand{\CZ}{\mathcal{Z}}

\newcommand{\na}{\nabla}

\newcommand{\al}{\alpha}
\newcommand{\be}{\beta}

\newcommand{\la}{\lambda}
\newcommand{\de}{\delta}
\newcommand{\si}{\sigma}
\newcommand{\pa}{\partial}

\newcommand{\Ga}{\Gamma}

\newcommand{\eqdef}{\overset{\mbox{\tiny{def}}}{=}}

\newcommand{\ind}{ {\mathbf 1}}

\newcommand{\testF}{C_{c}^{\infty}(\R^3_x \times \R^3_\vel)}
\newcommand{\simnew}{\approx}

\begin{document}                        

\title{The Vlasov-Poisson-Landau System in $\R^3_x$}

\author[R. M. Strain]{Robert M. Strain}
\thanks{R.M.S. was partially supported by the NSF grant DMS-0901463, and an Alfred P. Sloan Foundation Research Fellowship.}

\author[K. Zhu]{Keya Zhu}
\address{University of Pennsylvania, Department of Mathematics, David Rittenhouse Lab, 209 South 33rd Street, Philadelphia, PA 19104-6395, USA} 
\email{strain at math.upenn.edu  \& zhuk at math.upenn.edu}
\urladdr{http://www.math.upenn.edu/~strain/  \&   http://www.math.upenn.edu/~zhuk/}

\begin{abstract}
For the Landau-Poisson system with Coulomb interaction in $\R^3_x$, we prove the global existence, uniqueness, and large time convergence rates to the Maxwellian equilibrium for solutions which start out sufficiently close.
\end{abstract}


\setcounter{tocdepth}{1}

\maketitle
\tableofcontents

\thispagestyle{empty}




%
%
\section{Introduction}\label{intro}
In this paper we consider the important Landau-Poisson system:
\begin{eqnarray}
&&\partial _{t}F_{+}+v\cdot \nabla _{x}F_{+}+\frac{e_{+}}{m_{+}}E\cdot
\nabla _{v}F_{+}=Q(F_{+},F_{+})+Q(F_{-},F_{+}),  \notag  \label{boltzmann} \\
&&\partial _{t}F_{-}+v\cdot \nabla _{x}F_{-}-\frac{e_{-}}{m_{-}}E\cdot
\nabla _{v}F_{-}=Q(F_{+},F_{-})+Q(F_{-},F_{-}),  \label{vlasov} \\
&&F_{\pm }(0,x,v)=F_{0,\pm }(x,v).  \notag  \label{initial}
\end{eqnarray}%
Here $F_{\pm }(t,x,v)\geq 0$ are the spatially periodic number density
functions for the ions (+) and electrons (-) respectively, at time $t\geq 0$%
, position $x=(x_{1},x_{2},x_{3})\in \mathbb{R}^{3}$, velocity $v=(v_{1},v_{2},v_{3})\in \mathbb{R}^{3}$, and $e_{\pm },$ $m_{\pm }$ the magnitude of their charges and masses.  The collision between charged particles is given by 
\begin{multline}  \label{q}
Q(G_{\pm},G_{\mp})(v)
=
\frac{c_{\pm \mp}}{m_{\pm}}\nabla _{v}\cdot \int_{\mathbb{R}^{3}}\Phi (v-v_*)\frac{G_{\pm}(v_*)\nabla _{v}G_{\mp}(v)%
}{m_{\mp}} dv_*
\\
-\frac{c_{\pm \mp}}{m_{\pm}}
\nabla _{v}\cdot \int_{\mathbb{R}^{3}}\Phi (v-v_*)\frac{G_{\mp}(v)\nabla _{v_*}G_{\pm}(v_*)}{m_{\pm}} dv_*
\end{multline}
where $\Phi $ is the fundamental Landau (or Fokker-Planck) kernel \cite{hinton}: 
\begin{equation}
\Phi (v)=\frac{1}{|v|}\left( I-\frac{v\otimes v}{|v|^{2}}\right).    \label{landau}
\end{equation}%
Here $c_{\pm \mp}=2\pi e_{\pm}^{2}e_{\mp}^{2}\ln \Lambda ,$ $\ln \Lambda =\ln (\frac{%
\lambda _{D}}{b_{0}}),$ $\lambda _{D}=(\frac{T_0}{4\pi n_{e}e^{2}})^{1/2}$
is of course the Debye shielding distance and $b_{0}=\frac{e^{2}}{3T_0}$ is the
typical `distance of closest approach' for a thermal particle \cite{hinton}.   Note that the $\pm$ and the $\mp$ signify the possibility of either the $+$ or the $-$ in the sign configuration.

The self-consistent electrostatic field $E(t,x)=-\nabla \phi $, and the electric
potential $\phi $ will then satisfy:
\begin{equation}   \label{poisson}
\nabla \cdot E
=
-\Delta \phi =4\pi \rho 
=
4\pi \int_{\mathbb{R}^{3}}\{e_{+}F_{+}-e_{-}F_{-}\}dv. 
\end{equation}
We consider the global Maxwellian equilibrium: 
\begin{equation*}
\mu _{+}(v)=\frac{n_{0}}{e_{+}}(\frac{m_{+}}{2\pi \kappa T_{0}}%
)^{3/2}e^{-m_{+}|v|^{2}/2\kappa T_{0}},\;\;\;\mu _{-}(v)=\frac{n_{0}}{e_{-}}(%
\frac{m_{-}}{2\pi \kappa T_{0}})^{3/2}e^{-m_{-}|v|^{2}/2\kappa T_{0}}.
\end{equation*}
Our main goal will be to prove global in time existence of solutions to the Cauchy problem \eqref{vlasov}, and the convergence to these equilibria in large time for perturbative initial conditions.

For notational simplicity and without loss of generality, we normalize all constants in the
Vlasov-Poisson-Landau system to be one. Accordingly, we normalize the
Maxwellian as 
\begin{equation}
\mu (v)\equiv \mu _{+}(v)=\mu _{-}(v)=(2\pi)^{-3/2}e^{-|v|^{2}/2}.  \label{mu}
\end{equation}%
We define the standard perturbation $f_{\pm }(t,x,v)$ to $\mu $ as 
\begin{equation}
F_{\pm }=\mu +\sqrt{\mu }f_{\pm }.  \label{f}
\end{equation}%
Consider the vector $f(t,x,v)= [{f_{+}(t,x,v)}, {f_{-}(t,x,v)}]$, 
the Vlasov-Poisson-Landau
system for the perturbation now takes the form (two equations):
\begin{eqnarray}
\{\partial _{t}+v\cdot \nabla _{x}\pm E\cdot \nabla _{v}\}f_{\pm }
\mp
2\{E\cdot v\}\sqrt{\mu }
+
L_{\pm }f &=&\pm \{E\cdot v\}f_{\pm }+\Gamma _{\pm
}(f,f),  \label{vl} \\
&&  \notag \\
-\Delta \phi &=&\int \sqrt{\mu}[f_{+}-f_{-}]dv.  \label{p}
\end{eqnarray}%
For any $g=[{g_{+}}, {g_{-}}],$ the
linearized collision operator $Lg$ in (\ref{vl}) is given 
\begin{equation}
Lg\equiv \binom{L_{+}g\;}{L_{-}g}\equiv -\frac{1}{\sqrt{\mu }}\binom{2Q(\mu ,%
\sqrt{\mu }g_{+})+Q(\sqrt{\mu }\{g_{+}+g_{-}\},\mu )\;}{2Q(\mu ,\sqrt{\mu }%
g_{-})+Q(\sqrt{\mu }\{g_{+}+g_{-}\},\mu )}.  \label{L}
\end{equation}%
For $g=[g_{+},g_{-}]$ and $h=[h_{+},h_{-}],$ the nonlinear collision
operator $\Gamma (g,h)$ is 
\begin{equation}
\Gamma (g,h)\equiv \binom{\Gamma _{+}(g,h)}{\Gamma _{-}(g,h)}\equiv \frac{1}{%
\sqrt{\mu }}\binom{Q(\sqrt{\mu }g_{+},\sqrt{\mu }h_{+})+Q(\sqrt{\mu }g_{-},%
\sqrt{\mu }h_{+})}{Q(\sqrt{\mu }g_{+},\sqrt{\mu }h_{-})+Q(\sqrt{\mu }g_{-},%
\sqrt{\mu }h_{-})}.  \label{gamma}
\end{equation}
This is the model that we will study in the rest of this paper.

\subsection{Previous Results and Our Approach}
We recall that there have been many constructions of global in time perturbative solutions to kinetic equations in the last decade, see for instance  \cite{G1,G2,G4,G5,GS,gsNonCutJAMS,GrS2,SG1,SG2,MR2100057}.  

In particular we now discuss some specific results in the perturbative context.  We point out the global existence proof for the full relativistic Landau-Maxwell system from \cite{MR2100057} in 2004, and very recently the global result for the relativistic Vlasov-Maxwell-Boltzmann system in \cite{GS}.  In 2011, Guo constructed global smooth solution near Maxwellian for the (non-relativistic) Vlasov-Poisson-Landau system, as in \eqref{vl}, in \cite{G0} in the torus.  Notice that this Vlasov-Poisson-Landau system is physically motivated from the relativistic case because it arises as the formal Newtonian limit of the relativistic Landau-Maxwell system (\cite{MR2100057})

For other work related to the Landau equation from
different approaches, see \cite{AB,AV,HY,L,V,Z1,Z2}
and the references therein. 

Unfortunately, despite these advances, the stability of the Maxwellian equilibrium for the non-relativistic Vlasov-Poisson-Landau system in the whole space, as in \eqref{vl}, has remained out of reach.  Since the Poincar{\'e} inequality fails in $\R^3_x$, the strategy in \cite{G0} can not be directly applied.  But we can develop new ideas, building upon the methods in \cite{DS}, \cite{sNonCutOp} and \cite{S2} to overcome the main new difficulties. 

We first prove optimal large time decay rates for the linearized system by paying the cost of a gain in a velocity weight.
Then we build upon strategies from \cite{sNonCutOp}, \cite{G0}, and \cite{S2} to develop the following energy inequality for the solution of \eqref{p}:
$$\frac{d}{dt}\mathcal{\tilde{E}}(f)+\mathcal{\tilde{D}}\lesssim (||\pa_t\phi||_\infty+||\na_x\phi||_\infty)\mathcal{\tilde{E}}(f),
$$
where $\mathcal{\tilde{E}}(f)$ denotes some energy functional and  $\mathcal{\tilde{D}}(f)$ denotes some dissipation functional.  To use this energy inequality we notice, as usual, that $\mathcal{\tilde{D}}$ controls $\mathcal{\tilde{E}}$ when the weight $|v|$ is small compared to time variable $t^{p'}$ except for the term $\FP f$. When the weight $|v|$ is big compared to time variable $t^{p'}$, we  have control of the energy growth by putting more weights on the initial data. The gain in the  weight  function in our linear theory will not pose a problem for $\FP f$ in the Duhamel formula, because $\FP f$ has exponential decay in $v$. So by Gronwall's inequality, we  have nice control, \eqref{polydecay1}, as long as $\int (||\pa_t\phi||_\infty+||\na_x\phi||_\infty)(s)ds<\infty$. 

Unfortunately, the decay of $\mathcal{\tilde{E}}$ itself does not guarantee that  $$\int (||\pa_t\phi||_\infty+||\na_x\phi||_\infty)(s)ds<\infty.$$  To overcome this difficulty we develop another energy  inequality for the purpose of controlling  $\int (||\pa_t\phi||_\infty+||\na_x\phi||_\infty)(s)ds$ as follows:
$$\frac{d}{dt}\mathcal{\tilde{E}}^{(2)}(f)+\mathcal{\tilde{D}}^{(2)}\lesssim (||\pa_t\phi||_\infty+||\na_x\phi||_\infty)\mathcal{\tilde{E}}^{(2)}(f)+||\na_x \FP f||_2^2,$$
where $\mathcal{\tilde{E}}^{(2)}(f)$ no longer contains $\FP f$.
Then this new energy inequality gives us enough decay of $\mathcal{\tilde{E}}^{(2)}(f)$, which guarantees  $\int (||\pa_t\phi||_\infty+||\na_x\phi||_\infty)(s)ds<\infty$. 

We also mention the recent works of \cite{DYZ-VPBsoft,DYZ-VPL} for another approach to this problem.

In Section \ref{sec:not} we will introduce the notations used in this paper.  Then in Section \ref{sec:thm} we will state our main theorem.

\subsection{Notation} \label{sec:not}
For notational simplicity, we use $||\cdot ||_{p}$ to denote $L^{p}$ norms
in $\mathbb{R}^{3}_x\times \mathbb{R}^{3}_v$ or in $\R^3_x$, and $||\cdot ||_{p,w}$ for $L^{p}$ norms with weight $w(v)$ in  $\mathbb{R}^{3}_x\times \mathbb{R}^{3}_v$ or in $\R^3_x$.  We use $|\cdot |_{p}$ to denote $L^{p}$ norms
in $\mathbb{R}^{3}_v$, and $|\cdot |_{p,w}$ for $L^{p}$ norms with weight $w(v)$ in $\mathbb{R}^{3}_v$.
Let the multi-indices $\alpha $ and $\beta $ be $\alpha =[\alpha _{1},\alpha
_{2},\alpha _{3}],\;\;\beta =[\beta _{1},\beta _{2},\beta _{3}],$ and we
define $\partial _{\beta }^{\alpha }\equiv \partial _{x_{1}}^{\alpha
_{1}}\partial _{x_{2}}^{\alpha _{2}}\partial _{x_{3}}^{\alpha _{3}}\partial
_{v_{1}}^{\beta _{1}}\partial _{v_{2}}^{\beta _{2}}\partial _{v_{3}}^{\beta
_{3}}.$  For such a multi-index $\theta = [\theta_1, \theta_2, \theta_3]$, If each component of $\theta$ is not greater than that of $\bar{\theta}$'s, we denote by $\theta \leq \bar{\theta};$ $\theta <\bar{\theta}$
means $\theta \leq \bar{\theta},$ and $|\theta |<|\bar{\theta}|$ where $|\theta | = \theta_1 +  \theta_2 +  \theta_3$. 
Let $w(v)\geq 1$ be a weight function and $||\cdot ||_{2,w}$
to denote the weighted $L^{2}$ norm. We define 
\begin{equation}
||f||_{\sigma ,w}\equiv 
||\langle v\rangle ^{-%
\frac{1}{2}}f||_{2,w}+||\langle v\rangle ^{-\frac{3}{2}}\nabla _{v}f\cdot 
\frac{v}{|v|}||_{2,w}+||\langle v\rangle ^{-\frac{1}{2}}\nabla _{v}f\times 
\frac{v}{|v|}||_{2,w},  \label{sigma}
\end{equation}
with\ $\langle v\rangle =\sqrt{1+|v|^{2}}.$ Also $||f||_{\sigma}=||f||_{\sigma ,1}$. Similarly 
\begin{equation}
|f|_{\sigma ,w} 
\equiv |\langle v\rangle ^{-%
\frac{1}{2}}f|_{2,w}+|\langle v\rangle ^{-\frac{3}{2}}\nabla _{v}f\cdot 
\frac{v}{|v|}|_{2,w}+|\langle v\rangle ^{-\frac{1}{2}}\nabla _{v}f\times 
\frac{v}{|v|}|_{2,w},  \label{sigma'}
\end{equation}
and $|f|_{\sigma}=|f|_{\sigma ,1}$.

In this paper, the norm of a vector means the sum of the norms of all components of the vector. Also the norm of $\na_x^k f$ means the sum of the norms of functions $\partial^\al f$ where $|\al|=k$.

We define
the velocity weight 
\begin{equation}  \label{weight}
w(\alpha ,\beta )(v)\equiv 
\langle v\rangle^{2(l-|\alpha |-|\beta |)},\text{ \ \ \ }l\geq |\alpha |+|\beta |.  
\end{equation}%
We also define 
\begin{equation}
w_k\equiv \langle v\rangle^k.\label{weight'}
\end{equation}
We use $||\cdot||_{\si,k}$ and $|\cdot|_{\si,k}$ to denote $||\cdot||_{\si, w_k}$ and $|\cdot|_{\si, w_k}$ respectively. We use $||\cdot||_{p,k}$ and $|\cdot|_{p,k}$ to denote $||\cdot||_{p, w_k}$ and $|\cdot|_{p, w_k}$ respectively.

Recall (\ref{sigma}). We define the following continuous functionals, which are called the instant energy and the dissipation rate, by
\begin{eqnarray}
\mathcal{\tilde{E}}_{m;l}(f)(t) \label{penergy2}
&\equiv &\sum_{|\alpha |+|\beta |\leq m}\sum_{\pm
}||\partial _{\beta }^{\alpha }f_{\pm }(t)||_{2,w(\alpha ,\beta )}^{2}+||E(t)||_2^2,\text{
\ }   \notag\\
\mathcal{\tilde{D}}_{m;l}(f)(t) \label{pdiss2}
&\equiv &\sum_{|\alpha |+|\beta |\leq m}\sum_{\pm
}||\partial _{\beta }^{\alpha }\{\FI-\FP\}f_{\pm }(t)||_{\sigma ,w(\alpha ,\beta )}^{2}\notag\\
&  &+\sum_{1\leq |\alpha |\leq m}\sum_{\pm
}||\partial^{\alpha}\FP f_{\pm }(t)||_{2}^{2}+||E(t)||_2^2,%
\text{ \ } \notag \\
\mathcal{\tilde{\tilde{D}}}_{m;l}(f)(t) \label{pdiss3}
&\equiv &\sum_{|\alpha |\leq m}\sum_{\pm
}||\partial^{\alpha }\{\FI-\FP\}f_{\pm }(t)||_{\sigma ,w(\alpha ,0 )}^{2}\notag\\
&  &+\sum_{1\leq |\alpha |\leq m}\sum_{\pm
}||\partial^{\alpha}\FP f_{\pm }(t)||_{2}^{2}+||E(t)||_2^2,%
\text{ \ } \notag 
\end{eqnarray}

For an integrable function $g: \threed\to\R$, its Fourier transform is defined by
\begin{equation*}
  \widehat{g}(k)= \CF g(k)\eqdef \int_{\threed} e^{-2\pi \rmi x\cdot k} g(x)dx, \quad
  x\cdot
   k\eqdef\sum_{j=1}^\dim x_jk_j,
   \quad
   k\in\threed,
\end{equation*}
where $\rmi =\sqrt{-1}\in \mathbb{C}$. For two
complex vectors $a,b\in\mathbb{C}^\dim$, $(a\mid b)=a\cdot
\overline{b}$ denotes the dot product over the complex field, where
$\overline{b}$ is the ordinary complex conjugate of $b$.

We use $\langle\cdot,\cdot\rangle$ to denote the inner product over the Hilbert space $L^2_\vel$, i.e.
\begin{equation*}
    \langle g,h\rangle=\int_{\threed} g(\vel)\cdot \overline{h(\vel)} ~ d\vel,\ \ g=g(\vel), ~h=h(\vel)\in
    L^2_\vel.
\end{equation*}
Analogously $\left(\cdot,\cdot\right)$ denotes the inner product over $L^2(\threed_x\times\threed_\vel)$.
For $r\geq 1$, we define the mixed Lebesgue
space $Z_r=L^2_\vel(L^r_x)=L^2(\threed_\vel;L^r(\threed_x))$ with the norm
\begin{equation*}
\|g\|_{Z_r}\eqdef \left(\int_{\threed}\left(\int_{\threed}
    |g(x,\vel)|^rdx\right)^{2/r}d\vel\right)^{1/2}.
\end{equation*}
We introduce the norms
$\|\cdot\|_{\CHd^m}$ and $\|\cdot\|_{\CH^m}$ with $m\geq 0$ and $r\geq
1$ given by
\begin{equation}\label{brief.norm0}
    \|\solU\|_{\CHd^m}^2\eqdef \|\solU\|_{L^2_\vel(\dot{H}^m_x)}^2,
    \quad
    \|\solU\|_{\CH^m}^2\eqdef \|\solU\|_{L^2_\vel(H^m_x)}^2,
    \quad
    \|\solU\|_{\CL^2}^2 \eqdef\|\solU\|_{\CH^0}^2.
\end{equation}
Here $\dot{H}^m_x=\dot{H}^m(\threed_x)$ is the standard homogeneous $L^2_x$ based Sobolev space:
$$
\| g \|_{\dot{H}^m(\threed_x)}^2 \eqdef \int_{\threed} dk ~ |k|^{2m} | \hat{g}(k)|^2.
$$
In the next section we will discuss our main results.

Throughout this paper  we let $C$  denote
some positive (generally large) inessential constant and $\la$ denotes some positive (generally small) inessential constant, where both $C$ and
$\la$ may change values from line to line. 
Furthermore $A \lesssim B$ means $A \le C B$, and 
$A \gtrsim B$ means $B \lesssim A$.
In addition,
$A\approx B$ means $A \lesssim B$ and $B \lesssim A$.

\subsection{Main Theorem}  \label{sec:thm}
Here is our main theorem of this paper:

\begin{theorem}
\label{maintheorem}  Consider the initial data $f_{0}$ with $F_{0,\pm }(x,v)=\mu +\sqrt{\mu }f_{0,\pm
}(x,v)\geq 0.$  
Fix  $\ell > 5$, then there exists a small $M=M(\ell)>0$ such that if 
\begin{equation*}
\mathcal{\tilde{E}}_{3;\ell}({f}_{0})+||f_0||_{Z_1}^2\leq M,
\end{equation*}%
then there exists a unique global solution $f(t,x,v)$ to the
Vlasov-Poisson-Laudau system (\ref{vl}) and (\ref{p}) with $F_{\pm
}(t,x,v)=\mu +\sqrt{\mu }f_{\pm }(t,x,v)\geq 0.$

(1) Furthermore, we have 
\begin{eqnarray}
\mathcal{\tilde{E}}_{3;3}(f(t))
&\leq &
C_{\ell}(1+t)^{-\frac{3}{2}}\left(\mathcal{\tilde{E}}_{3;3+\frac{3}{5}\left(\ell - \frac{5}{2}\right)}(f_{0})+||f_0||_{Z_1}^2\right),
\label{polydecay1} \end{eqnarray}
and moreover
\begin{equation}
||\partial _{t}\phi (t)||_{\infty }+||\nabla _{x}\phi (t)||_{\infty}
 \leq 
 C_{\ell}(1+t)^{-\frac{5}{4}+\varepsilon}\left(\sqrt{\mathcal{\tilde{E}}_{3;\ell}(f_{0})}+||f_0||_{Z_1}\right),
\label{polydecay2} 
\end{equation}
where $\varepsilon \equiv \left( \ell - \frac{5}{2}\right)^{-1} \frac{5}{8} \in ( 0, \frac{1}{4})$.

(2) In addition, if $\mathcal{\tilde{E}}_{m;l}(f_{0})<\infty $ for any $l\geq \ell,$ $%
l\geq m$ $\geq 3,$  there exists an increasing continuous function 
$P_{m,l}(\cdot )$\ with $P_{m,l}(0 )=0$ such that the unique solution
satisfies 
\begin{equation}
\sup_{0\leq t\leq \infty }\mathcal{\tilde{E}}_{m;l}(f(t))+\int_{0}^{\infty }%
\mathcal{D}_{m;l}(f(s))ds\leq P_{m,l}(\mathcal{\tilde{E}}_{m;l}(f_{0}) ).
\label{highnorm}
\end{equation}
\end{theorem}

The rest of this paper is organized as follows.  In Section \ref{local} we study the local existence theory for solutions to the Vlasov-Poisson-Landau system.  In Section \ref{sec:LD} we study time decay properties of solutions to the linear model.  Then in Section \ref{energy inequality} we prove a collection of non-linear energy inequalities for a local solution.
After that in Section \ref{proofmain} we use the time decay estimates and the energy inequalities to deduce that the local solution is in fact global.
Lastly in Appendix \ref{timeDecayInt} we estimate some special time integrals.

\section{Local in time solutions}\label{local}

In this section, we construct a unique local-in time solution to the
Vlasov-Poisson-Landau system (\ref{vl}) and (\ref{p}) if $\mathcal{E}%
_{2;2}(f_{0})$ is sufficiently small.

We first note that the linearized collision operator $\FL$ is non-negative and is further known to be locally coercive in
the sense that there is a constant $\la_0>0$ such that:
\begin{equation}\label{coercLANDAU}
\langle f, \FL f  \rangle \geq \la_0 |\{\FI-\FP\} f |^2_\sigma,
\end{equation}
where, for fixed $(t,x)$, $\FP$ denotes the orthogonal projection from
$L^2_\vel$ to $\CN$, which is the null space of $L:$ 
\begin{equation}
\CN
\equiv
\mathrm{span}\left\{ \sqrt{\mu }\binom{1}{0},\;\sqrt{\mu }\binom{0}{1}%
,\;v_{i}\sqrt{\mu }\binom{1}{1},\;|v|^{2}\sqrt{\mu }\binom{1}{1}\right\} ,
\label{nulll}
\end{equation}%
with $1\leq i\leq 3.$
 
Here are some Sobolev inequalities frequently used throughout this paper.
\begin{lemma}
For $f:\R^3_x\to\R$, we have
\begin{eqnarray}
||f||_{\infty} &\lesssim & 
||\nabla_x f||_2^{1/2}||\nabla_x^2 f||_2^{1/2}
\lesssim 
||\nabla_x f||_2 + ||\nabla_x^2 f||_2;\label{inftySob}
\\
||f||_{6} &\lesssim & ||\nabla_x f||_2;\label{6Sob}\\
||f||_{4} &\lesssim & 
|| f||_2^{1/2}||\nabla_x f||_2^{1/2}
\lesssim
|| f||_2+||\nabla_x f||_2;\label{4Sob}
\\
||\nabla_x f||_{2} &\lesssim & 
|| f||_2^{1/2}||\nabla_x^2 f||_2^{1/2}
\lesssim
|| f||_2+||\nabla_x^2 f||_2.\label{2Sob}
\end{eqnarray}
\end{lemma}

In this section we will use some basic estimates from Section 2 of \cite{G0}. 
We would like to improve the result of Proposition 6 in \cite{G0}. 
For any vector-valued function ${\bf g(}v{\bf )=}[g_i],$ we define the
projection to the vector $[v_i]$ as 
\begin{equation}
P_vg_i\equiv \sum \{g_jv_j\}\frac{v_i}{|v|^2},\text{ }1\le i\le 3.
\label{pv}
\end{equation}
We will also use the norm estimate from Corollary 1 in \cite{G1}
\begin{multline}
|g|_{\sigma ,w_\wN }^2 \gtrsim
 \left|  [1+|v|]^{-\frac 3 2}\{P_v\partial _ig\}\right|_{2,w_\wN}^2+
 \left|  [1+|v|]^{-\frac{1} 2}\{[I-P_v]\partial _ig\}\right|_{2,w_\wN}^2
\\
+\left|  [1+|v|]^{- \frac{1}2}g\right|_{2,w_\wN}^2,
\quad \forall \wN \in \R.
\label{lbound}
\end{multline}
We establish the estimate for the nonlinear term $\Gamma [g_{1} , g_2]$.

\begin{proposition}\label{nonlinearTHM}
Fix 
$\wN\in \mathbb{R}$.
Then
\begin{equation}\label{langle} 
\langle w_{2\wN }\partial _\beta ^\alpha \Gamma [g_1,g_2],\;\partial
_\beta ^\alpha g_3\rangle  
\lesssim
\sum_{\alpha_1 \le \alpha}
\sum_{\bar{\beta}\le\beta_1 \le \beta}
|\partial _{\bar{\beta}}^{\alpha _1}g_1|_{2,-\wB}
|\partial _{\beta -\beta _1}^{\alpha -\alpha _1}g_2|_{\sigma ,\wN}
|\partial _\beta^\alpha g_3|_{\sigma ,\wN}. 
\end{equation}
Here $\wB>0$ can be arbitrarily large. 
\end{proposition}

The key new point in this estimate is to use directly Proposition 3.5 in \cite{gsNonCutJAMS} on the worst term \eqref{gamma3}.

\begin{proof} 
Recall $\Gamma [g_1,g_2]$ in (\ref{gamma}). By the product
rule, we expand 
\[
\langle w_{2\wN }\partial _\beta ^\alpha \Gamma [g_{1,}g_2],\partial
_\beta ^\alpha g_2\rangle =\sum C_\alpha ^{\alpha _1}C_\beta ^{\beta
_1}\times G_{\alpha _1\beta _1}, 
\]
where $G_{\alpha _1\beta _1}$takes the form: 
\begin{eqnarray}
&&+\langle w_{2\wN }\{\phi ^{ij}*\partial _{\beta _1}[v_i\mu
^{1/2}\partial _j\partial ^{\alpha _1}g_1]\}\partial _{\beta -\beta
_1}^{\alpha -\alpha _1}g_2,\partial _\beta ^\alpha g_3\rangle  \label{gamma4}
\\
&&-\langle w_{2\wN }\{\phi ^{ij}*\partial _{\beta _1}[v_i\mu
^{1/2}\partial ^{\alpha _1}g_1]\}\partial _j\partial _{\beta -\beta
_1}^{\alpha -\alpha _1}g_2,\partial _\beta ^\alpha g_3\rangle  \label{gamma2}
\\
&&-\langle w_{2\wN }\{\phi ^{ij}*\partial _{\beta _1}[\mu ^{1/2}\partial
^{\alpha _1}g_1]\}\partial _j\partial _{\beta -\beta _1}^{\alpha -\alpha
_1}g_2,\partial _i\partial _\beta ^\alpha g_3\rangle  \label{gamma1} \\
&&+\langle w_{2\wN }\ \{\phi ^{ij}*\partial _{\beta _1}[\mu
^{1/2}\partial _j\partial ^{\alpha _1}g_1]\}\partial _{\beta -\beta
_1}^{\alpha -\alpha _1}g_2,\partial _i\partial _\beta ^\alpha g_3\rangle
\label{gamma3} \\
&&-\langle \partial _i[w_{2\wN }]\{\phi ^{ij}*\partial _{\beta _1}[\mu
^{1/2}\partial ^{\alpha _1}g_1]\}\partial _j\partial _{\beta -\beta
_1}^{\alpha -\alpha _1}g_2,\partial _\beta ^\alpha g_3\rangle  \label{part1}
\\
&&+\langle \partial _i[w_{2\wN }]\{\phi ^{ij}*\partial _{\beta _1}[\mu
^{1/2}\partial _j\partial ^{\alpha _1}g_1]\}\partial _{\beta -\beta
_1}^{\alpha -\alpha _1}g_2,\partial _\beta ^\alpha g_3\rangle .
\label{part2}
\end{eqnarray}
The last two terms appear when we integrate by parts over $v_i$ variable.

We establish the upper bound in \eqref{langle} for each of the terms above. For the last two terms (\ref{part1}) and (%
\ref{part2}), we have 
\[
\phi ^{ij}(v)=O(|v|^{-1})\in L_{loc}^2({\R}^3). 
\]
From
\[
|\partial _{\beta _1-\bar{\beta}}\{\mu ^{1/2}\}| \lesssim
\mu ^{1/4}, 
\]
the Cauchy-Schwartz inequality implies  for any large $\wB > 0$ that
\begin{eqnarray} \notag
&&\{\phi ^{ij}*\partial _{\beta _1}[\mu ^{1/2}\partial ^{\alpha _1}g_1]\}\\
 &
\lesssim
&\sum_{\bar{\beta}\le \beta _1}|\phi ^{ij}*\{\partial _{\beta _1-\bar{\beta}%
}\mu ^{1/2}\partial _{\bar{\beta}}^{\alpha _1}g_1\}| \notag
\\
&\lesssim &
\{|\phi ^{ij}|^2*\mu ^{1/4}\}^{1/2}(v)\left\{ \sum_{\bar{\beta}\le
\beta _1}\int \mu ^{1/4}(v_*)
|\partial _{\bar{\beta}}^{\alpha _1}g_1(v_*)|^{2} dv_*\right\} ^{1/2}  \nonumber \\
&\lesssim &
[1+|v|]^{-1}\sum_{\bar{\beta}\le \beta _1}\left| 
\partial _{\bar{\beta}}^{\alpha _1}g_1\right| _{2, - \wB}.  
 \label{convolution} 
\end{eqnarray}
Since $|\partial _i[w_{2\wN }]|\le C[1+|v|]^{-1}w_{2\wN },$ we
estimate (\ref{part1}) from (\ref{convolution}) as follows
\begin{eqnarray*}
&&
\sum_{\bar{\beta}\le \beta _1}
\left|  \partial _{\bar{\beta}}^{\alpha _1}g_1\right|_{2,-\wB }
\int w_{2\wN }[1+|v|]^{-2}|\partial
_j\partial _{\beta -\beta _1}^{\alpha -\alpha _1}g_2\partial _\beta ^\alpha
g_3|dv \\
&\lesssim&
\sum_{\bar{\beta}\le \beta _1}
\left|  \partial _{\bar{\beta}}^{\alpha _1}g_1\right|_{2,-\wB }
\left| w_\wN [1+|v|]^{-\frac {3}{ 2}}\partial_j\partial _{\beta -\beta _1}^{\alpha -\alpha _1}g_2\right| _2\left|
w_\wN [1+|v|]^{-\frac {1}{ 2}}\partial _\beta ^\alpha g_3\right| _2 \\
&\lesssim&
\sum_{\bar{\beta}\le \beta _1}|\partial _{\bar{\beta}}^{\alpha
_1}g_1|_{2,-\wB }
|\partial _{\beta -\beta _1}^{\alpha -\alpha_1}g_2|_{\sigma ,\wN }
|\partial _\beta ^\alpha g_3|_{\sigma ,\wN },
\end{eqnarray*}
by (\ref{lbound}).

To estimate \eqref{part2}, we integrate by parts on the $\partial_j$ derivative in the $v_*$ integration (inside the convolution).  This splits
\begin{equation}
\phi ^{ij}*\partial _{\beta _1}[\mu ^{1/2}\partial _j\partial ^{\alpha
_1}g_1]=\partial _j\phi ^{ij}*\partial _{\beta _1}[\mu ^{1/2}\partial
^{\alpha _1}g_1]-\phi ^{ij}*\partial _{\beta _1}[\partial _j\mu
^{1/2}\;\partial ^{\alpha _1}g_1].  \label{parts}
\end{equation}
When the derivative hits $\mu^{1/2}$ we proceed exactly as in the estimate for \eqref{part1}.  When the derivative hits $\phi^{ij}(v-v_*)$ we use the symmetry in $v$ and $v_*$ to convert the $\partial_j$ to be a derivative in $v$ (instead of one in $v_*$) and then we integrate by parts in $\partial_j$ again (now within the $v$ integration).  All of the remaining terms resulting from this integration by parts can be estimated as in \eqref{part1}.

We now estimate \eqref{gamma4}.  Again recall that $\phi ^{ij}(v)=O(|v|^{-1})\in L_{loc}^2.$ By (\ref{convolution}), we have for any large $\wB>0$ that
\begin{equation}
|\phi ^{ij}*\partial _{\beta _1}[v_i\mu ^{1/2}\partial _j\partial ^{\alpha
_1}g_1]|
\lesssim
[1+|v|]^{-1}\sum_{\bar{\beta}\le \beta }|\partial _{\bar{\beta}%
}^{\alpha _1}g_1|_{\sigma ,-\wB }.  \nonumber
\end{equation}
Hence the corresponding integrand \eqref{gamma4} is bounded by 
\begin{equation*}
C\sum \left| \partial _{\bar{\beta}}^{\alpha _1}g_1\right| _{2,-\wB}
[1+|v|]^{-1}
\left|\partial _{\beta -\beta_1}^{\alpha -\alpha _1}g_2 \right| 
\left| \partial _\beta ^\alpha g_3\right|.
\end{equation*}
Again using Cauchy-Schwartz and \eqref{lbound}, the $v$ integral of the above is clearly bounded by right hand side
of \eqref{langle}.

Next we consider the estimates for \eqref{gamma2} and \eqref{gamma1}.  We decompose their double
integration region $[v,v_{*}]\in $ ${\R}^3\times {\R}^3$ into
three parts: 
\[
\{|v|\le 1\},\;\;\;\{2|v_{*}|\ge |v|,\;|v|\ge 1\}\text{ \ \ and \ \ }%
\{2|v_{*}|\le |v|,\;|v|\ge 1\}. 
\]
For the first part $\{|v|\le 1\},$ recall $\phi ^{ij}(v)=O(|v|^{-1})\in L_{loc}^2.$ By (\ref{convolution}), we have 
\begin{eqnarray}
&&|\phi ^{ij}*\partial _{\beta _1}[\mu ^{1/2}\partial ^{\alpha
_1}g_1]|+|\phi ^{ij}*\partial _{\beta _1}[v_i\mu ^{1/2}\partial ^{\alpha
_1}g_1]|  \nonumber \\
&\le &C[1+|v|]^{-1}\sum_{\bar{\beta}\le \beta }|\partial _{\bar{\beta}%
}^{\alpha _1}g_1|_{2,-\wB }.  \label{l2} 
\end{eqnarray}
Hence their corresponding integrands over the region $\{|v|\le 1\}$ are
bounded by 
\begin{equation*}
C\sum \left| \partial _{\bar{\beta}}^{\alpha _1}g_1\right| _{2,-\wB}w^{2\wN }[1+|v|]^{-1}|\partial _j\partial _{\beta -\beta
_1}^{\alpha -\alpha _1}g_2|[|\partial _i\partial _\beta ^\alpha
g_3|+|\partial _\beta ^\alpha g_3|],
\end{equation*}
whose $v-$integral over $\{|v|\le 1\}$ is clearly bounded by right hand side
of (\ref{langle}). We thus conclude the first part of $\{|v|\le 1\}$ for 
(\ref{gamma2}) and \eqref{gamma1}.

For the second part $\{2|v_{*}|\ge |v|,|v|\ge 1\},$ we have 
\[
|\partial _{\beta _1}\{\mu ^{1/2}(v_{*})\}|
+
|\partial _{\beta_1}\{v_{*j}\mu ^{1/2}(v_{*})\}|
\lesssim
\mu^{1/8} (v_{*}) \mu^{1/8} (v). 
\]
By the same type of estimates as in (\ref{convolution}), the $v-$ integrands
in (\ref{gamma2}) and (\ref{gamma1}) are bounded by: 
\begin{eqnarray*}
&&
\mu^{1/8} (v)
w_{2\wN }|\partial _j\partial _{\beta -\beta _1}^{\alpha
-\alpha _1}g_2|[|\partial _i\partial _\beta ^\alpha g_3|+|\partial _\beta
^\alpha g_3|]
\int  \left| \phi ^{ij}(v-v_{*}) \right|
\mu^{1/8} (v_{*}) 
|\partial _{\bar{\beta}}^{\alpha _1}g_1(v_{*})|dv_{*}
\\
&\lesssim &
\left| \partial _{\bar{\beta}}^\alpha g_1\right| _{2,-\wB}
[1+|v|]^{-1} \mu^{1/8} (v)
w_{2\wN }|\partial _j\partial _{\beta
-\beta _1}^{\alpha -\alpha _1}g_2|
~ [|\partial _i\partial _\beta ^\alpha
g_3|+|\partial _\beta ^\alpha g_3|].
\end{eqnarray*}
By (\ref{lbound}), its $v-$integral is bounded by the right hand side of (%
\ref{langle}) because of the fast decaying factor $\mu^{1/8} (v)$.   
We thus conclude the second part of $\{2|v_{*}|\ge |v|,|v|\ge 1$ $\}$ for the terms (\ref
{gamma2}) and (\ref{gamma1}).

We finally consider the third part of \{$2|v_{*}|\le |v|,|v|\ge 1\},$
for which we shall estimate the terms in (\ref{gamma2}) and (\ref{gamma1}).
The key is to expand $\phi ^{ij}(v-v_{*}).$

To estimate (\ref{gamma1}) over the region $|v|\ge 1$ and $2|v_{*}|\le
|v|,$ we expand $\phi ^{ij}(v-v_{*})$ to get 
\begin{equation}
\phi ^{ij}(v-v_{*})=\phi ^{ij}(v)-\sum_k\partial _k\phi
^{ij}(v)v_{*k}+\frac 12\sum_{k,l}\partial _{kl}\phi ^{ij}(\bar{v}%
)v_{*k}v_{*l}.  \label{taylor2}
\end{equation}
where $\bar{v}$ is between $v$ and $v-v_{*}.$ We plug (\ref{taylor2})
into the integrand of (\ref{gamma1}). Notice that for either fixed $i$ or $%
j, $%
\begin{equation}
\sum_i\phi ^{ij}(v)v_i=\sum_j\phi ^{ij}(v)v_j=0.  \label{phiv}
\end{equation}
>From (\ref{pv}), Lemma 3 in \cite{G1} and (\ref{phiv}), we can decompose $\partial
_j\partial _{\beta -\beta _1}^{\alpha -\alpha _1}g_2$ and $\partial
_i\partial _\beta ^\alpha g_3$ into their $P_v$ parts as well as $I-P_v$
parts. For the first term in the expansion (\ref{taylor2}) 
\begin{eqnarray}
&&|\sum_{ij}\phi ^{ij}(v)\partial _j\partial _{\beta -\beta _1}^{\alpha
-\alpha _1}g_2(v)\partial _i\partial _\beta ^\alpha g_3(v)|  \nonumber \\
&=&|\sum_{ij}\phi ^{ij}(v)\{[I-P_v]\partial _j\partial _{\beta -\beta
_1}^{\alpha -\alpha _1}g_2(v)\}\{[I-P_v]\partial _i\partial _\beta ^\alpha
g_3(v)\}|  \label{expand1} \\
&\le &C[1+|v|]^{-1}|[I-P_v]\partial _j\partial _{\beta -\beta
_1}^{\alpha -\alpha _1}g_2(v)|\times |[I-P_v]\partial _i\partial _\beta
^\alpha g_3(v)|.  \nonumber
\end{eqnarray}
Here we have used (\ref{phiv}) so that sum of terms with either 
$$
P_v\partial_j\partial _{\beta -\beta _1}^{\alpha -\alpha _1}g_2,
\quad
\text{or}
\quad
P_v\partial_i\partial _{\beta -\beta _1}^{\alpha -\alpha _1}g_3,
$$ 
vanishes. For the
second term in the expansion (\ref{taylor2}), by taking a $\partial_k$ derivative of 
\[
\sum_{i,j}\phi ^{ij}(v)v_iv_j=0, 
\]
we have 
$
\sum_{i,j}\partial _k\phi ^{ij}(v)v_iv_j=-2\sum_j\phi ^{kj}(v)v_j=0. 
$
Therefore, expanding $\partial _j\partial _{\beta -\beta _1}^{\alpha -\alpha
_1}g_2$ and $\partial _i\partial _\beta ^\alpha g_3$ into their $P_v$ and $%
I-P_v$ parts yields 
\begin{eqnarray*}
&&\sum_{i,j}\partial _k\phi ^{ij}(v)\partial _j\partial _{\beta -\beta
_1}^{\alpha -\alpha _1}g_2(v)\partial _i\partial _\beta ^\alpha
g_3(v)=\sum_{i,j}\partial _k\phi ^{ij}(v)\times \\
\times \{ &&[I-P_v]\partial _j\partial _{\beta -\beta _1}^{\alpha -\alpha
_1}g_2[I-P_v]\partial _i\partial _\beta ^\alpha g_3+[I-P_v]\partial
_j\partial _{\beta -\beta _1}^{\alpha -\alpha _1}g_2[P_v\partial _i\partial
_\beta ^\alpha g_3] \\
&&+[P_v\partial _j\partial _{\beta -\beta _1}^{\alpha -\alpha
_1}g_2][I-P_v]\partial _i\partial _\beta ^\alpha g_3\},
\end{eqnarray*}
where 
\[
\sum_{i,j}\partial _k\phi ^{ij}(v)P_v\partial _j\partial _{\beta -\beta
_1}^{\alpha -\alpha _1}g_2\times P_v\partial _i\partial _\beta ^\alpha
g_3=0. 
\]
Notice that $|\partial _k\phi ^{ij}(v)|\le C[1+|v|]^{-2},$ for $|v|\ge 1,$ we majorize the above by 
\begin{eqnarray}
C &&[1+|v|]^{-3 /2}\{|P_v\partial _j\partial _{\beta -\beta _1}^{\alpha
-\alpha _1}g_2|+|P_v\partial _i\partial _\beta ^\alpha g_3|\}\times 
\nonumber \\
&&\times [1+|v|]^{-1/2}\{|[I-P_v]\partial _j\partial _{\beta -\beta
_1}^{\alpha -\alpha _1}g_2|+|[I-P_v]\partial _i\partial _\beta ^\alpha g_3|\}
\label{expand2} \\
&&+C[1+|v|]^{-1/2}|[I-P_v]\partial _j\partial _{\beta -\beta
_1}^{\alpha -\alpha _1}g_2||[I-P_v]\partial _i\partial _\beta ^\alpha g_3|. 
\nonumber
\end{eqnarray}
The third term in (\ref{taylor2}) now can be estimated as follows. Since 
\begin{equation}
\frac 12|v|\le |v|-|v_{*}|\le |\bar{v}|\le |v_{*}|+|v|\le \frac
32|v|,  \label{vbound}
\end{equation}
thus 
$
|\partial _{kl}\phi ^{ij}(\bar{v})|\le C[1+|v|]^{-3} , 
$
and we have 
\begin{eqnarray}
&&|\sum_{k,l}\partial _{kl}\phi ^{ij}(\bar{v})\partial _j\partial _{\beta
-\beta _1}^{\alpha -\alpha _1}g_2(v)\partial _i\partial _\beta ^\alpha
g_3(v)|  \nonumber \\
&\le &C[1+|v|]^{-3}|\partial _j\partial _{\beta -\beta _1}^{\alpha
-\alpha _1}g_2(v)\partial _i\partial _\beta ^\alpha g_3(v)|.  \label{expand3}
\end{eqnarray}
Combining (\ref{taylor2}), (\ref{expand1}), (\ref{expand2}) and (\ref
{expand3}), we have 
\begin{eqnarray*}
&&|\sum_{i,j}\phi ^{ij}(v-v_{*})\partial _i\partial _{\beta -\beta
_1}^{\alpha -\alpha _1}g_2\partial _i\partial _\beta ^\alpha g_3| \\
&\le &C[1+|v_{*}|]^2\{|\sum_{i,j}\phi ^{ij}(v)\partial _j\partial
_{\beta -\beta _1}^{\alpha -\alpha _1}g_2(v)\partial _i\partial _\beta
^\alpha g_3(v)| \\
&&+|\sum_{i,j}\partial _k\phi ^{ij}(v)\partial _j\partial _{\beta -\beta
_1}^{\alpha -\alpha _1}g_2(v)\partial _i\partial _\beta ^\alpha g_3(v)| \\
&&+|\sum_{i,j}\partial _{kl}\phi ^{ij}(\bar{v})\partial _j\partial _{\beta
-\beta _1}^{\alpha -\alpha _1}g_2(v)\partial _i\partial _\beta ^\alpha
g_3(v)|\} \\
&\le &C[1+|v_{*}|]^2\{\sigma ^{ij}\partial _i\partial _{\beta -\beta
_1}^{\alpha -\alpha _1}g_2\partial _j\partial _{\beta -\beta _1}^{\alpha
-\alpha _1}g_2\}^{1/2}\{\sigma ^{ij}\partial _i\partial _\beta ^\alpha
g_3\partial _j\partial _\beta ^\alpha g_3\}^{1/2},
\end{eqnarray*}
where we have used (\ref{lbound}). The $v-$ integrand over $\{2|v^{\prime
}|\le |v|,$ $|v|\ge 1\}$ in (\ref{gamma1}) is thus bounded by 
\begin{eqnarray*}
&&w_{2\wN }\int [1+|v_{*}|^2]\mu ^{1/4}(v_{*})|\partial _{%
\bar{\beta}}^{\alpha _1}g_1(v_{*})|dv_{*}\times \\
&&\times \{\sigma ^{ij}\partial _i\partial _{\beta -\beta _1}^{\alpha
-\alpha _1}g_2\partial _j\partial _{\beta -\beta _1}^{\alpha -\alpha
_1}g_2\}^{1/2}\{\sigma ^{ij}\partial _i\partial _\beta ^\alpha g_3\partial
_j\partial _\beta ^\alpha g_3\}^{1/2} \\
&\le &C|\partial _{\bar{\beta}}^\alpha g_1|_{2,-\wB }\{w_{2\wN }\sigma
^{ij}\partial _i\partial _{\beta -\beta _1}^{\alpha -\alpha _1}g_2\partial
_j\partial _{\beta -\beta _1}^{\alpha -\alpha _1}g_2\}^{1/2}\{w_{2\wN
}\sigma ^{ij}\partial _i\partial _\beta ^\alpha g_3\partial _j\partial
_\beta ^\alpha g_3\}^{1/2}.
\end{eqnarray*}
Its further integration over $v$ is bounded by the right hand side of (\ref
{langle}).

We now consider the term (\ref{gamma2}). We again expand $\phi
^{ij}(v-v_{*})$ as 
\begin{equation}
\phi ^{ij}(v-v_{*})=\phi ^{ij}(v)-\sum_k\partial _k\phi ^{ij}(\bar{v}%
)v_{*k},  \label{phi1}
\end{equation}
with $\bar{v}$ between $v$ and $v-v_{*}.$ Since $\sum_j\phi
^{ij}(v)v_j=0,$ we obtain as before 
\begin{eqnarray}
&&\sum_{i,j}\phi ^{ij}(v)\partial _j\partial _{\beta -\beta _1}^{\alpha
-\alpha _1}g_2(v)\partial _\beta ^\alpha g_3(v)  \nonumber \\
&=&\sum_{i,j}\phi ^{ij}(v)\{I-P_v\}\partial _j\partial _{\beta -\beta
_1}^{\alpha -\alpha _1}g_2(v)\times \partial _\beta ^\alpha g_3(v)
\label{name1} \\
&\le &C|[1+|v|]^{-1/2}\{I-P_v\}\partial _j\partial _{\beta -\beta
_1}^{\alpha -\alpha _1}g_2(v)||[1+|v|]^{-1/2}\partial _\beta ^\alpha
g_3(v)|.  \nonumber
\end{eqnarray}
Notice that from (\ref{vbound}), $|\partial _k\phi ^{ij}(\bar{v})|\le
C[1+|v|]^{-2}.$ Hence 
\begin{eqnarray}
&&|\partial _k\phi ^{ij}(\bar{v})\partial _j\partial _{\beta -\beta
_1}^{\alpha -\alpha _1}g_2(v)\partial _\beta ^\alpha g_3(v)|  \label{name2}
\\
&\le &\{[1+|v|]^{-3 /2}|\partial _j\partial _{\beta -\beta _1}^{\alpha
-\alpha _1}g_2(v)|\}\{[1+|v|]^{-1/2}|\partial _\beta ^\alpha
g_3(v)|\}.  \nonumber
\end{eqnarray}
By (\ref{lbound}), we conclude that the integrand in (\ref{gamma2}) can be
majorized as 
\begin{eqnarray*}
&\int &|w_{2\wN }\phi ^{ij}(v-v_{*})\partial _{\beta
_1}[v_i^{\prime }\mu ^{1/2}(v_{*})\partial ^{\alpha _1}g_1(v^{\prime
})]\partial _j\partial _{\beta -\beta _1}^{\alpha -\alpha _1}g_2(v)\partial
_\beta ^\alpha g_3(v)|dv_{*} \\
&=&\int |w_{2\wN }[\phi ^{ij}(v)-\partial _k\phi ^{ij}(\bar{v}%
)v_{*k}]\partial _{\beta _1}[v_i^{\prime }\mu ^{1/2}(v^{\prime
})\partial ^{\alpha _1}g_1(v_{*})]\partial _j\partial _{\beta -\beta
_1}^{\alpha -\alpha _1}g_2(v)\partial _\beta ^\alpha g_3(v)|dv_{*} \\
&\le &C\int [1+|v_{*}|]\mu ^{1/4}(v_{*})|\partial _{\bar{\beta}%
}^{\alpha _1}g_1(v_{*})|dv_{*}\times \\
&&\times \{w_{2\wN }\sigma ^{ij}\partial _i\partial _{\beta -\beta
_1}^{\alpha -\alpha _1}g_2\partial _j\partial _{\beta -\beta _1}^{\alpha
-\alpha _1}g_2\}^{1/2}\{w_{2\wN }\sigma ^{ij}\partial _i\partial _\beta
^\alpha g_3\partial _j\partial _\beta ^\alpha g_3\}^{1/2},
\end{eqnarray*}
where the summation is over $1\le i,j\le 3$ and its further integration over 
$v$ is bounded by the right hand side of (\ref{langle}).

We are now ready to prove the desired estimate for the most singular term \eqref{gamma3}.   Now to handle \eqref{gamma3}  we integrate by parts on the $\partial_j$ derivative in the $v_*$ integration (inside the convolution).  This splits
\begin{equation}
\phi ^{ij}*\partial _{\beta _1}[\mu ^{1/2}\partial _j\partial ^{\alpha
_1}g_1]=
\partial _j\phi ^{ij}*\partial _{\beta _1}[\mu ^{1/2}\partial^{\alpha _1}g_1]
-
\phi ^{ij}*\partial _{\beta _1}[\partial _j\mu^{1/2}
\;\partial ^{\alpha _1}g_1].  \label{partsSECONDtime}
\end{equation}
For  \eqref{gamma3}, the term with the integrand in \eqref{partsSECONDtime} involving $\partial _j\mu^{1/2}$ can now be estimated exactly as in the previous estimate of \eqref{gamma2}.  Thus in the remainder of this estimate we focus or attention on the term in \eqref{partsSECONDtime} which involves $\partial _j\phi ^{ij}$.

Notice $\partial _j\phi ^{ij}(v) = O \left( |v|^{-2}\right)$.  Therefore, in the convolution $\partial _j\phi ^{ij}*\partial _{\beta _1}[\mu ^{1/2}\partial^{\alpha _1}g_1]$, 
we now split into the two regions $|v-v_*| \ge 1$ and $|v-v_*| \le 1$.  When we restrict the convolution to the region where $|v-v_*| \ge 1$, similar to \eqref{convolution}, we have
\begin{equation}\notag
\left| \partial _j\phi ^{ij}*\partial _{\beta _1}[\mu ^{1/2}\partial^{\alpha _1}g_1] \right|
\lesssim
[1+|v|]^{-2}\sum_{\bar{\beta}\le \beta _1}\left| 
\partial _{\bar{\beta}}^{\alpha _1}g_1\right| _{2, - \wB},
\end{equation}
which will hold for any $\wB >0$.  Plugging this upper bound into the integrand of \eqref{gamma3} and applying the approach which was used to estimate \eqref{part1} yields the desired estimate for this term.

It now only remains to estimate  \eqref{gamma3} with the integrand $\partial _j\phi ^{ij}*\partial _{\beta _1}[\mu ^{1/2}\partial^{\alpha _1}g_1]$ on the region $|v-v_*| \le 1$.  However from the previous estimates, on this region, we can observe that \eqref{gamma3} with the integrand $\partial _j\phi ^{ij}*\partial _{\beta _1}[\mu ^{1/2}\partial^{\alpha _1}g_1]$ is bounded above in absolute value by a constant multiple of 
$$
\int_{\mathbb{R}^3\times \mathbb{R}^3} dv dv_* ~ w_{2\wN}(v)  |v-v_*|^{-2} ~ \mu^{1/4}(v) \left|\partial^{\alpha _1}_{\bar{\beta}}g_1(v_*) 
\partial _{\beta -\beta_1}^{\alpha -\alpha _1}g_2(v) \partial _i\partial _\beta ^\alpha g_3(v)
\right| 
{\bf 1}_{|v-v_*| \le 1}.
$$
Here ${\bf 1}_{|v-v_*| \le 1}$ is the usual indicator function of the indicated region: $|v-v_*| \le 1$.  In the above expression we are implicitly summing in particular over $i$ and $\bar{\beta}$.
Now we can apply the HLS estimate from Proposition 3.5 in  \cite{gsNonCutJAMS}.  Using this estimate the expression in the previous display is bounded above by a constant multiple of 
\begin{multline*}
\left| \mu^\delta\partial _{\bar{\beta}}^{\alpha _1}g_1\right| _{L^2_v}
\left| \mu^\delta  \partial _{\beta -\beta_1}^{\alpha -\alpha _1}g_2\right| _{H^1_v}
\left|  \mu^\delta  \partial _i\partial _\beta ^\alpha g_3\right| _{L^2_v}
\\
\lesssim
\left| 
\partial _{\bar{\beta}}^{\alpha _1}g_1\right| _{2, - \wB}
\left|   \partial _{\beta -\beta_1}^{\alpha -\alpha _1}g_2\right| _{\sigma, \ell}
\left|  \partial _i\partial _\beta ^\alpha g_3\right| _{\sigma, \ell}.
\end{multline*}
As in Proposition 3.5 in \cite{gsNonCutJAMS}, $\delta>0$ is some small number. 
In Proposition 3.5 of \cite{gsNonCutJAMS}, we have chosen $b_1 =1$ and $b_2 = 0$.  Of course, again, above $\wB>0$ can be taken arbitrarily large.  We thus conclude the proof of Theorem \ref{nonlinearTHM}. 
\end{proof}

Now we start our construction of the local in time solution. We start with 
\begin{equation}
F^{0}(t,x,v)=\mu \text{ \ \ or \ \ }f^{0}\equiv 0.  \label{start}
\end{equation}%
To preserve the positivity for $F^{n+1}$, we design the following iterating
sequence of $F_{\pm }^{n+1}$ as in \cite{MR2100057}: 
\begin{eqnarray}
\lbrack \partial _{t}+v\cdot \nabla _{x}\mp \nabla _{x}\phi ^{n}\cdot \nabla
_{v}]F_{\pm }^{n+1} &=&Q(F_{\pm }^{n},F_{\pm }^{n+1})-8\pi F_{\pm
}^{n}(F_{\pm }^{n+1}-F_{\pm }^{n})  \notag \\
&&+Q(F_{\mp }^{n},F_{\pm }^{n+1})-8\pi F_{\mp }^{n}(F_{\pm }^{n+1}-F_{\pm
}^{n})  \notag \\
&=&\Phi ^{ij}\ast \lbrack F_{\pm }^{n}+F_{\mp }^{n}]\partial _{ij}F_{\pm
}^{n+1}+8\pi \{F_{\pm }^{n}\}^{2}+8\pi F_{\mp }^{n}F_{\pm }^{n},  \notag \\
\Delta \phi ^{n+1} &=&-\int (F_{+}^{n+1}-F_{-}^{n+1})dv.  \label{vlpFn}
\end{eqnarray}
We have used the formulation for $Q$ which can be found, for instance, in \cite[Lemma 8]{G1}.
We note that $F_{\pm }^{n}\geq 0$
implies $F_{\pm }^{n+1}\geq 0$ from (\ref{vlpFn}) and the maximum principle.  We now rewrite the
above iteration in the perturbation form of $F^{n+1}=\mu +\sqrt{\mu }%
f^{n+1}: $%
\begin{eqnarray}
&&[\partial _{t}+v\cdot \nabla _{x}\mp \nabla _{x}\phi ^{n}\cdot \nabla
_{v}]f_{\pm }^{n+1}-Af_{\pm }^{n+1}\pm \nabla _{x}\phi ^{n}\cdot vf_{\pm
}^{n+1}  \notag \\
&=&\mp 2\nabla _{x}\phi ^{n}\cdot v\sqrt{\mu }+K_{\pm }f^{n}+\Gamma _{\pm
}(f^{n},f^{n+1})  \notag \\
&&-8\pi (f_{\pm }^{n}+f_{\mp }^{n})\sqrt{\mu }(f_{\pm }^{n+1}-f_{\pm
}^{n})-16\pi \mu (f_{\pm }^{n+1}-f_{\pm }^{n})  \notag \\
&&-\Delta \phi ^{n+1}=\int \sqrt{\mu} (f_{+}^{n+1}-f_{-}^{n+1})dv.  \label{vlpfn}
\end{eqnarray}%
with $f^{n+1}|_{t=0}=f_{0}.$ Here for $g=[g_{+}, g_{-}]$, we denote as
in [G1] [SG1-2]: 
\begin{equation*}
Ag_{\pm } =\frac{2}{\sqrt{\mu }}Q(\mu ,\sqrt{\mu }g_{\pm }), \quad
K_{\pm }g =\frac{1}{\sqrt{\mu }}Q(\sqrt{\mu }[g_{\pm }+g_{\mp }],\mu ).
\end{equation*}%
 
We have the following local existence theorem:

\begin{lemma}
\label{mixnenergy}Assume for $M>0$ sufficiently small, we have 
$
\mathcal{\tilde{E}}_{2;2}(f^{n})\leq M.
$
Suppose that $l \ge m \ge 2$.

(1) We have 
{\small
\begin{eqnarray}
&&\mathcal{\tilde{E}}_{2;l}(f^{n+1})+\int_{0}^{t}\mathcal{D}_{2;l}(f^{n+1})ds 
\notag \\
&\leq &\frac{1}{8}\int_{0}^{t}\mathcal{D}_{2;l}(f^{n})ds+C_{l}[\mathcal{\tilde{E}}%
_{2;l}(f_{0})\notag\\
&&+\int_{0}^{t}\{||\nabla _{x}\phi ^{n}||_2^2+||\nabla _{x}\phi ^{n+1}||_2^2+\sum_{|\alpha |\leq 2}\{||\mu ^{\delta%
}\partial ^{\alpha }f^{n}||_{2}^{2}+||\mu ^{\delta}\partial ^{\alpha
}f^{n+1}||_{2}^{2}\} \}] \notag \\
&&+C_{l}\int_{0}^{t}[\sum_{|\alpha ^{\prime }|+|\beta ^{\prime }|\leq
1}\left\Vert |\partial _{\beta ^{\prime }}^{\alpha ^{\prime }}f^{n}|_{\sigma
}\right\Vert _{H^{\frac{3}{4}}}^{2}+||\nabla _{x}\phi ^{n}||_{\infty
}+||\partial _{t}\phi ^{n}||_{\infty }]\mathcal{\tilde{E}}_{2;l}(f^{n+1})  \notag
\\
&&+C_{l}\int_{0}^{t}\sum_{|\alpha ^{\prime }|+|\beta ^{\prime }|\leq
1}\left\Vert |\partial _{\beta ^{\prime }}^{\alpha ^{\prime
}}f^{n+1}|_{\sigma ,\frac{w(\alpha ^{\prime },\beta ^{\prime })}{\langle
v\rangle ^{2}}}\right\Vert _{H^{\frac{3}{4}}}\sqrt{\mathcal{D}_{2;l}(f^{n})%
}\sqrt{\mathcal{\tilde{E}}_{2;l}(f^{n+1})}  \notag \\
&&+C_{l}\int_{0}^{t}\sum_{|\alpha ^{\prime }|+|\beta ^{\prime }|\leq
1}\left\Vert |\partial _{\beta ^{\prime }}^{\alpha ^{\prime }}f^{n+1}|_{2,%
\frac{w(\alpha ^{\prime },\beta ^{\prime })}{\langle v\rangle ^{2}}%
}\right\Vert _{H^{\frac{3}{4}}}\sqrt{\mathcal{D}_{2;l}(f^{n})}\sqrt{%
\mathcal{D}_{2;l}(f^{n+1})}  \notag \\
&&+C_{l}\int_{0}^{t}\sum_{|\alpha ^{\prime }|+|\beta ^{\prime }|\leq 1}\sqrt{%
\mathcal{D}_{2;l}(f^{n})}\left\Vert |\partial _{\beta ^{\prime }}^{\alpha
^{\prime }}(f_{\pm }^{n+1}-f_{\pm }^{n})|_{\sigma }\right\Vert _{H^{\frac{3}{%
4}}}\sqrt{\mathcal{\tilde{E}}_{2;l}(f^{n+1})}  \notag \\
&&+C_{l}\int_{0}^{t}\sum_{|\alpha ^{\prime }|+|\beta ^{\prime }|\leq 1}\sqrt{%
\mathcal{D}_{2;l}(f^{n+1}-f^{n})}\left\Vert |\partial _{\beta ^{\prime
}}^{\alpha ^{\prime }}f^{n}|_{\sigma }\right\Vert _{H^{\frac{3}{4}}}\sqrt{%
\mathcal{\tilde{E}}_{2;l}(f^{n+1})}  \label{e2n} \\
&&+C_{l}\int_{0}^{t}\sqrt{\mathcal{\tilde{E}}_{2;l}(\mu \{f^{n+1}-f^{n}\})}\sqrt{%
\mathcal{\tilde{E}}_{2;l}(f^{n+1})}.  \notag
\end{eqnarray}%
}

(2) For $m\geq 3,$ we have 
{\small
\begin{eqnarray}
&&\mathcal{\tilde{E}}_{m;l}(f^{n+1},\phi ^{n})+\int_{0}^{t}\mathcal{D}%
_{m;l}(f^{n+1})ds  \notag \\
&\leq &\frac{1}{8}\int_{0}^{t}\mathcal{D}_{m;l}(f^{n})ds+C_{l}\mathcal{\tilde{E}}%
_{m;l}(f_{0})\notag\\
&&+C_{l,m}\int_{0}^{t}\{||\nabla _{x}\phi ^{n}||_2^2+||\nabla _{x}\phi ^{n+1}||_2^2+\sum_{|\alpha |=m}\{||\mu ^{\delta%
}\partial ^{\alpha }f^{n}||_{2}^{2}+||\mu ^{\delta}\partial ^{\alpha
}f^{n+1}||_{2}^{2}\}\}  \notag \\
&&+C_{l,m}\int_{0}^{t}[\sum_{|\alpha ^{\prime }|+|\beta ^{\prime }|\leq
\lbrack \frac{m}{2}]}\left\{ \left\Vert |\partial _{\beta ^{\prime
}}^{\alpha ^{\prime }}f^{n}|_{\sigma ,\frac{w(\alpha ^{\prime },\beta
^{\prime })}{\langle v\rangle ^{2}}}\right\Vert _{H^{\frac{3}{4}%
}}^{2}+\left\Vert |\partial _{\beta ^{\prime }}^{\alpha ^{\prime
}}f^{n+1}|_{\sigma ,\frac{w(\alpha ^{\prime },\beta ^{\prime })}{\langle
v\rangle ^{2}}}\right\Vert _{H^{\frac{3}{4}}}^{2}\right\}\notag\\
&&\left. +||\nabla \phi
^{n}||_{\infty }+||\partial _{t}\phi ^{n}||\right]  \times \lbrack \mathcal{\tilde{E}}_{m;l}(f^{n})+\mathcal{\tilde{E}}_{m;l}(f^{n+1})] 
\notag \\
&&+C_{l,m}\int_{0}^{t}\sum_{|\alpha ^{\prime }|+|\beta ^{\prime }|\leq
\lbrack \frac{m}{2}]}\left\Vert |\partial _{\beta ^{\prime }}^{\alpha
^{\prime }}f^{n}|_{\sigma }\right\Vert _{H^{\frac{3}{4}}}\sqrt{\mathcal{D}%
_{m;l}(f^{n+1}-f^{n})}\sqrt{\mathcal{\tilde{E}}_{m;l}(f^{n+1})}  \notag \\
&&+C_{l,m}\int_{0}^{t}\sum_{|\alpha ^{\prime }|+|\beta ^{\prime }|\leq
\lbrack \frac{m}{2}]}\left\Vert |\partial _{\beta ^{\prime }}^{\alpha
^{\prime }}(f^{n+1}-f^{n})|_{\sigma }\right\Vert _{H^{\frac{3}{4}}}\sqrt{%
\mathcal{D}_{m;l}(f^{n})}\sqrt{\mathcal{\tilde{E}}_{m;l}(f^{n+1})}  \notag \\
&&+C_{l,m}\int_{0}^{t}\sqrt{\mathcal{\tilde{E}}_{m;l}(\mu \{f^{n+1}-f^{n}\})}\sqrt{%
\mathcal{\tilde{E}}_{m;l}(f^{n+1})}  \notag \\
&&+C_{l,m}\int_{0}^{t}[\mathcal{D}_{m-1;l}(f^{n})+\mathcal{D}%
_{m-1;l}(f^{n+1})][\mathcal{\tilde{E}}_{m-1;l}(f^{n})+\mathcal{\tilde{E}}%
_{m-1;l}(f^{n+1})+1].  \label{emn}
\end{eqnarray}  }
Above $\delta$ is some small positive number.
\end{lemma}

\begin{proof}
By the computation in Lemma 10 of \cite{G0}, we have
\begin{eqnarray}
&&\frac{d}{dt}\left\{ \int \frac{e^{\pm 2\phi ^{n}}w^{2}(\partial
_{\beta }^{\alpha }f_{\pm }^{n+1})^{2}}{2}\right\} -\int \langle
w^{2}\partial _{\beta }^{\alpha }Af^{n+1},\partial _{\beta }^{\alpha
}f^{n+1}\rangle  \label{local1} \\
&=&-\int e^{\pm 2\phi ^{n}}w^{2}\delta _{\beta }^{\mathbf{e}%
_{i}}\partial _{\beta -\mathbf{e}_{i}}^{\alpha +\mathbf{e}_{i}}f_{\pm
}^{n+1}\partial _{\beta }^{\alpha }f_{\pm }^{n+1} +\tilde{A}+\tilde{B},  \label{local2} 
\label{local11}
\end{eqnarray}
where
\begin{eqnarray}
\tilde{A}&=&\pm \sum_{\alpha _{1}<\alpha }C_{\alpha }^{\alpha _{1}}\int e^{\pm
2\phi ^{n}}w^{2}\partial _{\beta }^{\alpha }f_{\pm }^{n+1}\partial
^{\alpha -\alpha _{1}}\nabla _{x}\phi ^{n}\cdot \nabla _{v}\partial _{\beta
}^{\alpha _{1}}f_{\pm }^{n+1}  \label{local3} \\
&&\mp \sum_{\alpha _{1}<\alpha }C_{\alpha }^{\alpha _{1}}\int e^{\pm
2\phi ^{n}}w^{2}\partial _{\beta }^{\alpha }f_{+}^{n+1}\partial
^{\alpha -\alpha _{1}}\nabla _{x}\phi ^{n}\cdot \partial _{\beta }[v\partial
^{\alpha _{1}}f_{\pm }^{n+1}]  \label{local4} \\
&&\mp \int [\frac{2(l-|\alpha |-|\beta |)}{1+|v|^{2}}\nabla _{x}\phi
^{n}\cdot v-\phi _{t}^{n}]e^{\pm 2\phi ^{n}}w^{2}(\partial _{\beta
}^{\alpha }f_{\pm }^{n+1})^{2}  \label{local5} \\
&&+\int w^{2}(e^{\pm 2\phi ^{n}}-1)\partial _{\beta }^{\alpha }f_{\pm
}^{n+1}\partial _{\beta }^{\alpha }Af_{\pm }^{n+1}.  \label{local6} 
\end{eqnarray}
and
\begin{eqnarray}
\tilde{B}&=&\mp 2\int e^{\pm 2\phi ^{n}}w^{2}\nabla _{x}\partial ^{\alpha }\phi
^{n}\cdot \partial _{\beta }[v\sqrt{\mu }]\partial _{\beta }^{\alpha }f_{\pm
}^{n+1}  \label{local7} \\
&&+\int w^{2}e^{\pm 2\phi _{n}}\partial _{\beta }^{\alpha }K_{\pm
}f^{n}\partial _{\beta }^{\alpha }f_{\pm }^{n+1}  \label{local8} \\
&&+\int w^{2}e^{\pm 2\phi ^{n}}\partial _{\beta }^{\alpha }\Gamma _{\pm
}(f^{n},f^{n+1})\partial _{\beta }^{\alpha }f_{\pm }^{n+1}  \label{local9} \\
&&-8\pi \int w^{2}e^{\pm 2\phi ^{n}}\partial _{\beta }^{\alpha }[\sqrt{%
\mu }(f_{\pm }^{n}+f_{\mp }^{n})(f_{\pm }^{n+1}-f_{\pm }^{n})]\partial
_{\beta }^{\alpha }f_{\pm }^{n+1}  \label{local10} \\
&&-16\pi \int w^{2}e^{\pm 2\phi ^{n}}\partial _{\beta }^{\alpha }[\mu
(f_{\pm }^{n+1}-f_{\pm }^{n})]\partial _{\beta }^{\alpha }f_{\pm }^{n+1}.
\label{local11}
\end{eqnarray}
Above 
\begin{equation}
\delta _{\beta }^{\mathbf{e}_{i}}=1\text{ \ if \ \ \ }\mathbf{e}_{i}\leq
\beta ;\text{ \ }\ \text{or \ \ }\delta _{\beta }^{\mathbf{e}_{i}}=0,\text{
\ otherwise}.  \label{de}
\end{equation}

In our case, we have to estimate $\frac{d}{dt}||\nabla _{x}\phi ^{n+1}||_2^2$. Also we use a new estimate of \eqref{local7}.  For the estimate used in \cite{G0} of this term does not hold in $\R^3_x$.

For  $\frac{d}{dt}||\nabla _{x}\phi ^{n+1}||_2^2$, we use the continuity equation. Recall $\rho ^{n}=\int \sqrt{\mu 
}[f_{+}^{n}-f_{-}^{n}]dv$ and $j^{n}=\int v\sqrt{\mu }%
[f_{+}^{n}-f_{-}^{n}]dv $. We now note that from the continuity equation of 
\begin{equation*}
\rho _{t}^{n}+\nabla _{x}\cdot j^{n}=0,
\end{equation*}%
we have 
\begin{equation}
-\Delta \nabla _{x}\phi ^{n}=\nabla _{x}\rho ^{n},\text{ \ \ }\Delta
\partial _{t}\phi ^{n}=\nabla _{x}\cdot j^{n}.  \label{phit}
\end{equation}%
Thus
\begin{equation}\label{ineqE}
|\frac{d}{dt}||\nabla _{x}\phi ^{n+1}||_2^2|\leq ||\nabla _{x}\phi ^{n+1}||_2||j^{n+1}||_2\lesssim ||\nabla _{x}\phi ^{n+1}||_2^2+||\mu ^\delta
f^{n+1}||_{2}^{2}.
\end{equation}
For (\ref{local7}), via repeated integration by parts in $v$, we can move all the $v$ derivative $\partial_\beta$ out of $\partial_\beta^\alpha f^{n+1}$ to the factor $\partial_\beta[v\sqrt{\mu}]w^2$, so we have
\begin{eqnarray}
&&|\mp 2\int e^{\pm 2\phi ^{n}}w^{2}\nabla _{x}\partial ^{\alpha }\phi
^{n}\cdot \partial _{\beta }[v\sqrt{\mu }]\partial _{\beta }^{\alpha }f_{\pm
}^{n+1}|\notag\\
&\lesssim& ||\nabla _{x}\partial ^{\alpha }\phi
^{n}||_2^2+ ||\mu^\delta\partial ^{\alpha }f_{\pm
}^{n+1}||_2^2\notag\\
&\lesssim& ||\nabla _{x}\phi
^{n}||_2^2+ ||\mu^\delta\partial ^{\alpha }f_{\pm
}^{n}||_2^2+||\mu^\delta\partial ^{\alpha }f_{\pm
}^{n+1}||_2^2\label{ineqlocal7},
\end{eqnarray}
where the last line comes from interpolation.

We then collect estimates for all other terms in the proof of Lemma 10 in \cite{G0}.
For $\beta =0,$%
\begin{equation}
-\int \langle w^{2}(\alpha ,0)\partial ^{\alpha }Af_{\pm }^{n+1},\partial
^{\alpha }f_{\pm }^{n+1}\rangle dx\gtrsim ||\partial ^{\alpha
}f^{n+1}||_{\sigma ,w(\alpha ,0)}^{2}-C_{m}||\chi \partial ^{\alpha
}f^{n+1}||_{2}^{2};  \label{ax}
\end{equation}%
$\chi (v)$ being a general cutoff function in $v$. For $\beta \neq 0,$ for
any $\eta >0,$ we have 
\begin{eqnarray}
&&-\int \langle w^{2}(\alpha ,\beta )\partial _{\beta }^{\alpha }Af_{\pm
}^{n+1},\partial _{\beta }^{\alpha }f_{\pm }^{n+1}\rangle dx  \notag \\
&\gtrsim &||\partial _{\beta }^{\alpha }f^{n+1}||_{\sigma ,w(\alpha ,\beta
)}^{2}-\eta \mathcal{D}_{m;l}(f^{n+1})-C_{l,\eta }\sum_{\beta ^{\prime }<\beta
}||\partial _{\beta ^{\prime }}^{\alpha }f^{n+1}||_{\sigma ,w(\alpha ,\beta
^{\prime })}^{2}.  \label{av}
\end{eqnarray}
For any $\eta >0$ and $\beta \geq \mathbf{e}%
_{i}$%
\begin{eqnarray}
(\ref{local2}) &\lesssim &||\delta _{\beta }^{\mathbf{e}_{i}}\partial
_{\beta -\mathbf{e}_{i}}^{\alpha }f^{n+1}||_{\sigma ,w(\alpha ,\beta -%
\mathbf{e}_{i})}||\partial _{\beta -\mathbf{e}_{i}}^{\alpha +\mathbf{e}%
_{i}}f_{\pm }^{n+1}||_{\sigma ,w(\alpha +\mathbf{e}_{i},\beta -\mathbf{e}%
_{i})}  \notag \\
&\leq &\eta \mathcal{D}_{m;l}(f^{n+1})+C_{\eta }||\delta _{\beta }^{%
\mathbf{e}_{i}}\partial _{\beta -\mathbf{e}_{i}}^{\alpha }f_{\pm
}^{n+1}||_{\sigma ,w(\alpha ,\beta -\mathbf{e}_{i})}.  \label{streamlocal}
\end{eqnarray}
For (\ref{local3}) and (\ref{local4}), we have
\begin{eqnarray}
&&(\ref{local3})+(\ref{local4})  \label{fieldnonlinear} \\
&\lesssim &\eta ||\partial _{\beta }^{\alpha }f^{n+1}||_{\sigma
,w}^{2}+C_{\eta }\left\Vert |f_{\pm }^{n+1}|_{\sigma ,\frac{w(0,0)}{\langle
v\rangle ^{2}}}\right\Vert _{H^{\frac{3}{4}}}^{2}\mathcal{E}%
_{m;m}(f^{n})\notag\\
&&+C_{\eta }[\mathcal{E}_{m-1;m-1}(f^{n})+M]\mathcal{D}%
_{m-1;l}(f^{n+1}).  \notag
\end{eqnarray}%
For (\ref{local5}), we have
\begin{equation}
(\ref{local5})\leq C_{l,m}\int_{0}^{t}\{||\phi _{t}^{n}||_{\infty }+||\nabla
_{x}\phi ^{n}||_{\infty }\}w^{2}(\partial _{\beta }^{\alpha }f_{\pm
}^{n+1})^{2}.  \label{dfield}
\end{equation}
For (\ref{local6}), we have for any $\eta >0,$%
\begin{eqnarray}
(\ref{local6}) &\lesssim &\sqrt{M}||\partial _{\beta }^{\alpha
}f^{n+1}||_{\sigma ,w(\alpha ,\beta )}^{2}+C_{l,m}\sqrt{M}\sum_{|\alpha
|\leq m}||\partial ^{\alpha }f^{n+1}\chi ||_{2}^{2}  \label{smalll} \\
&&+\eta \mathcal{D}_{m;l}(f^{n+1})+C_{l,m,\eta }\sum_{\beta ^{\prime }<\beta
}||\partial _{\beta ^{\prime }}^{\alpha }f^{n+1}||_{\sigma ,w(\alpha ,\beta
^{\prime })}^{2}.  \notag
\end{eqnarray}
For (\ref{local8}), we have
\begin{eqnarray}
(\ref{local8}) &\lesssim &\{\eta \sqrt{\mathcal{D}_{m;l}(f^{n})}%
+C_{l,m,\eta }\sum_{|\alpha |\leq m}||\chi \partial ^{\alpha }f^{n}||_{2}\}%
\sqrt{\mathcal{D}_{m;l}(f^{n+1})}  \label{kbound} \\
&\lesssim &\eta \lbrack \mathcal{D}_{m;l}(f^{n})+\mathcal{D}%
_{m;l}(f^{n+1})]+C_{l,m,\eta }\sum_{|\alpha |\leq m}||\chi \partial
^{\alpha }f^{n}||_{2}^{2}.  \notag
\end{eqnarray}
For (\ref{local10}), we have
\begin{eqnarray*}
&&|w\partial _{\beta }^{\alpha }\{\mu ^{\frac{1}{4}}(f_{\pm }^{n}+f_{\mp
}^{n})\}\{\mu ^{\frac{1}{4}}(f_{\pm }^{n+1}-f_{\pm }^{n})\}|_{2} \\
&\lesssim &\sum_{|\alpha ^{\prime }|+|\beta ^{\prime }|\leq m}|w\partial
_{\beta -\beta ^{\prime }}^{\alpha -\alpha ^{\prime }}\{\mu ^{\frac{1}{4}%
}(f_{\pm }^{n}+f_{\mp }^{n})\}\partial _{\beta ^{\prime }}^{\alpha ^{\prime
}}\{\mu ^{\frac{1}{4}}(f_{\pm }^{n+1}-f_{\pm }^{n})\}|_{2} \\
&\lesssim &\sum_{|\alpha ^{\prime }|+|\beta ^{\prime }|\leq m}|w\partial
_{\beta -\beta ^{\prime }}^{\alpha -\alpha ^{\prime }}\{\mu ^{\frac{1}{4}%
}(f_{\pm }^{n}+f_{\mp }^{n})\}|_{L_{v}^{4}}|\partial _{\beta ^{\prime
}}^{\alpha ^{\prime }}\{\mu ^{\frac{1}{4}}(f_{\pm }^{n+1}-f_{\pm
}^{n})\}|_{L_{v}^{4}}\notag\\
&\lesssim &C_{l}\sum_{|\alpha ^{\prime }|+|\beta ^{\prime }|\leq \lbrack 
\frac{m}{2}]}\sqrt{\mathcal{D}_{m;l}(f^{n})}\left\Vert |\partial _{\beta
^{\prime }}^{\alpha ^{\prime }}(f_{\pm }^{n+1}-f_{\pm }^{n})|_{\sigma
}\right\Vert _{H^{\frac{3}{4}}}\notag\\
&&+C_{l}\sum_{|\alpha ^{\prime }|+|\beta ^{\prime }|\leq \lbrack \frac{m}{2}]}%
\sqrt{\mathcal{D}_{m;l}(f^{n+1}-f^{n})}\left\Vert |\partial _{\beta
^{\prime }}^{\alpha ^{\prime }}f^{n}|_{\sigma }\right\Vert _{H^{\frac{3}{4}%
}}.
\end{eqnarray*}%
For (\ref{local11}), we have
$$
(\ref{local11})\lesssim \sqrt{\mathcal{E}_{2;l}(\mu
\{f^{n+1}-f^{n}\})}\sqrt{\mathcal{E}_{2;l}(f^{n+1})}.$$
For (\ref{local9}), when $|\al|+|\be|\leq 2$, we have
$$(\ref{local9})\lesssim \sqrt{\mathcal{E}_{2;2}(f^{n})}\mathcal{D}%
_{2;l}(f^{n+1}).$$
For (\ref{local9}), when $|\alpha |+|\beta |=m\geq 3,$  for any $\eta >0,$ there exists $C_{l,m,\eta }>0$ with
\begin{eqnarray}
(\ref{local9})
&\lesssim & \sqrt{\mathcal{E}_{2;2}(f^{n})}||\partial _{\beta }^{\alpha }f^{n+1}||_{\sigma ,w}^{2}  \notag \\
&&+\eta \{\sqrt{\mathcal{E}_{2;2}(f^{n})}
+1\}\sum_{\substack{ |\alpha ^{\prime }|+|\beta ^{\prime }|=m  \\ \beta
^{\prime }\leq \beta }}||\partial _{\beta ^{\prime }}^{\alpha ^{\prime
}}f^{n+1}||_{\sigma ,w}^{2}  \notag \\
&&+C_{l,m,\eta }\sum_{|\alpha ^{\prime }|+|\beta ^{\prime }|\leq \lbrack 
\frac{m}{2}]} \left\Vert
|\partial _{\beta ^{\prime }}^{\alpha ^{\prime }}f^{n+1}|_{\sigma ,\frac{%
w(\alpha ^{\prime },\beta ^{\prime })}{\langle v\rangle ^{2}}}\right\Vert
_{H^{\frac{3}{4}}}^{2} \mathcal{E}_{m;l}(f^{n}) \notag \\
&&+C_{l,m,\eta }\{\mathcal{E}_{m-1;l}(f^{n})+1\}\{\mathcal{D}%
_{m-1;l}(f^{n+1})\}.  \label{gamma3x}
\end{eqnarray}
We note that our estimate of (\ref{local9}) is much simpler than that in Lemma 10 of \cite{G0}, because our result in Proposition \ref{nonlinearTHM} is much simpler than that of Proposition 6 in \cite{G0}. Combining the estimates for \eqref{local1} to \eqref{local11} completes the proof of this lemma.
\end{proof}

We now have the following uniform bound on the iterating sequence:

\begin{lemma}
\label{uniformbound}Assume $f_{0}\in \testF$ such that $F_{0}=\mu +%
\sqrt{\mu }f_{0}\ge 0$ with (\ref{start}).

(1) There exist small constants $0<T\leq 1$ and $M>0,$ such that if $%
\mathcal{\tilde{E}}_{2;2}(f_{0},\phi _{0})$ is sufficiently small, then
\begin{equation}
\mathcal{\tilde{E}}_{2;2}(f^{n+1})+\int_{0}^{t}\mathcal{D}_{2;2}(f^{n+1})(s)ds%
\leq M.  \label{2uniform}
\end{equation}

(2) $\{f^{n}\}$ is Cauchy in $L^{\infty }([0,T],L_{x,v}^{2})$ and $\{\nabla_x \phi^n\}$ is Cauchy in $L^{\infty }([0,T],L_{x}^{2}).$

(3) There exists $C_{l}>0$ such that for $0\leq t\leq T:$ 
\begin{equation}
\mathcal{\tilde{E}}_{2;l}(f^{n+1})(t)+\int_{0}^{t}\mathcal{D}%
_{2;l}(f^{n+1})(s)ds\leq C_{l}\mathcal{\tilde{E}}_{2;l}(0).  \label{2luniform}
\end{equation}

(4) Assume (2) is valid. For $m\geq 3,$ there exists an increasing
continuous function $P_{m,l}$ with $P_{m,l}(0)=0$ such that for $0\leq t\leq
T:$ 
\begin{equation}
\mathcal{\tilde{E}}_{m;l}(f^{n+1})+\int_{0}^{t}\mathcal{D}_{m;l}(f^{n+1})ds\leq
P_{m,l}(\mathcal{\tilde{E}}_{m;l}(f_{0})).  \label{mluniform}
\end{equation}

(5) $\{F^{n}\geq 0\}.$
\end{lemma}
\begin{proof}
For (\ref{2uniform}), with Lemma \ref{mixnenergy}, the proof of part (1) of Lemma 11 in \cite{G0} is still valid in our setting.

Now we turn to part (2); with the continuity equation (\ref{phit}), we have 
\begin{equation}\label{phicauchy}
\frac{d}{dt}||\nabla_x (\phi^{n+1}-\phi^n)||_2^2 \lesssim ||\nabla_x (\phi^{n+1}-\phi^n)||_2^2+\sum_{\pm}||f_\pm^{n+1}-f_\pm^n||_2^2.
\end{equation}
Combinig (\ref{phicauchy}) with the estimates of the difference equation of part (2) of Lemma 11 in \cite{G0} gives
\begin{eqnarray*}
&&||f^{n+1}-f^{n}||_{2}^{2}(t)+||\nabla_x (\phi^{n+1}-\phi^n)||_2^2+\int_{0}^{t}||f^{n+1}-f^{n}||_{\sigma }^{2}\\
&\leq& 
\frac{1}{8}\int_{0}^{t}||f^{n}-f^{n-1}||_{\sigma
}^{2}(s)\\
&&+C\int_{0}^{t}\{||\nabla_x (\phi^{n+1}-\phi^n)||_2^2+||f^{n+1}-f^{n}||_{2}^{2}+||f^{n}-f^{n-1}||_{2}^{2}%
\}.
\end{eqnarray*}%
with $\sqrt{\mathcal{E}_{2;2}(f^{n})}$ sufficiently small and 
$e^{\pm 2\phi ^{n}} \approx 1$, 
By taking $T$ sufficiently small, we obtain%
\begin{eqnarray*}
&&\max \left\{ \sup_{0\leq t\leq T}\{||f^{n+1}-f^{n}||_{2}^{2}(t)+||\nabla_x (\phi^{n+1}-\phi^n)||_2^2\},\text{ }%
\int_{0}^{T}||f^{n+1}-f^{n}||_{\sigma }^{2}\right\} \\
&\leq &\frac{1}{4}\left\{ \sup_{0\leq t\leq
T}\{||f^{n+1}-f^{n}||_{2}^{2}(t)+||\nabla_x (\phi^{n+1}-\phi^n)||_2^2\}+\int_{0}^{T}||f^{n}-f^{n-1}||_{\sigma
}^{2}\right\} .
\end{eqnarray*}%
By iteration of $n,$ 
$$
\sup_{0\leq t\leq T}\{||f^{n+1}-f^{n}||_{2}^{2}(t)+||\nabla_x (\phi^{n+1}-\phi^n)||_2^2\}+\int_{0}^{T}||f^{n+1}-f^{n}||_{\sigma
}^{2}\leq \frac{1}{2^{n}}.
$$  
So $\{f^{n}\}$ and $\nabla_x \phi^n$ are Cauchy sequences in $%
L^{\infty }(L_{x,v}^{2})$.

The proof of part (3), (4) and (5) is the same as the corresponding proof in Lemma 11 in \cite{G0}.
\end{proof}

Similar to \cite{G0}, we summarize the local well-posedness as $n\rightarrow \infty .$

\begin{theorem}
\label{localsolution}Assume that $\mathcal{\tilde{E}}_{2;2}(f_{0})$ is
sufficiently small. Then there exist $0<T\leq 1$ and $M>0$ small such that
there is a unique solution $F=\mu +\sqrt{\mu }f \ge 0$ with 
\begin{equation*}
\mathcal{\tilde{E}}_{2;2}(f)(t)+\int_{0}^{t}\mathcal{D}_{2;2}(f)(s)ds\lesssim 
\mathcal{\tilde{E}}_{2;2}(0 \lesssim M.
\end{equation*}%
In general$,$ if $0\leq t\leq T,$ there exists an increasing continuous
function $P_{m,l}(\cdot )$ with $P_{m,l}(0)=0$ such that 
\begin{equation*}
\mathcal{\tilde{E}}_{m;l}(f)(t)+\int_{0}^{t}\mathcal{D}_{m;l}(f)(s)ds
\lesssim 
P_{m,l}(\mathcal{\tilde{E}}_{m;l}(f_{0})).
\end{equation*}
\end{theorem}

\section{Linear decay theory}\label{sec:LD}
The linearized Vlasov-Poisson-Landau
system for the perturbation  takes the form 
\begin{eqnarray}\label{VPLlinear} 
\{\partial _{t}+v\cdot \nabla _{x}\}f_{\pm }\mp
2\{E\cdot v\}\sqrt{\mu }+L_{\pm }f 
&=& \sourceG_\pm, \\
&&  \notag \\
-\Delta \phi &=&\int \sqrt{\mu}[f_{+}-f_{-}]dv,  
\label{poissonLINEAR}
\end{eqnarray}%
where in the case of \eqref{vl} the source term is of the form 
$$
 \sourceG_\pm
 \eqdef
 \mp E\cdot \nabla _{v} f_{\pm }
\pm \{E\cdot v\}f_{\pm }+\Gamma _{\pm}(f,f). 
$$
Again we recall the definition  $E \eqdef -\nabla_x \phi$.

\subsection{Moment equations}\label{sec.me}

In this section, we begin with the representation of the macroscopic
projector $\FP$ and then derive from the perturbed system some
macroscopic balance laws and high-order moment equations.

Given any $f(t,x,v) = [f_+, f_-]$, one can write $\FP f$ in \eqref{coercLANDAU}  and \eqref{nulll} using the expansion $\FP f=[\FP_+ f, \FP_- f]$ as
\begin{equation} \label{form.p}
    \FP_\pm f= \{a_\pm(t,x)+b(t,x)\cdot \vel+c(t,x)(|\vel|^2-3)\}\FM^{1/2}. 
\end{equation}
Of course $\FP$ is a projection from $L^2_\vel\times L^2_\vel$ to $\CN$, where the coefficient functions $a_\pm(t,x)$, $b(t,x)\equiv [b_1(t,x),b_2(t,x),b_3(t,x)]$ and $c(t,x)$ depend on $f(t,x,\vel)$. 

Since the projection $\FP$ is orthogonal we have
\begin{equation*}
     \int_{\R^3} \psi(\vel)\cdot \{\FI-\FP\} f ~ d\vel=0,\ \ \forall\,\psi = [\psi_+, \psi_-]\in
    \CN.
\end{equation*}
This together with the 
form \eqref{form.p} of $\FP$ imply
\begin{equation}
\notag
\begin{split}
  \dis   & a_\pm= \langle \FM^{1/2}, f_\pm\rangle= \langle \FM^{1/2}, \FP_\pm f\rangle,
  \\
  \dis  & b_i=\frac{1}{2}\langle \vel_i \FM^{1/2}, f_++f_-\rangle
=\langle \vel_i \FM^{1/2},\FP_\pm f\rangle,
\\
  \dis & c= \frac{1}{12}\langle (|\vel|^2-3) \FM^{1/2}, f_+ + f_-\rangle
= \frac{1}{6}\langle (|\vel|^2-3) \FM^{1/2}, \FP_\pm f\rangle.
    \end{split}
\end{equation}
In the rest of this section we will derive the equations for these macroscopic variables and also the high-order moments as follows.

First consider the linearized system with a non-homogeneous source $\sourceG(t,x,\vel)=[\sourceG_+,\sourceG_-]$ as in \eqref{VPLlinear} and \eqref{poissonLINEAR}.  Taking velocity integrations of \eqref{VPLlinear} with respect to the velocity moments
\begin{equation}
\notag
    \FM^{1/2},\ \  \vel_i  \FM^{1/2}, i=1,2,3,\ \  \frac{1}{6}(|\vel|^2-3)\FM^{1/2},
\end{equation}
one has
\begin{eqnarray}
&&\quad \pa_t a_\pm +\na_x\cdot b 
+
\na_x \cdot \langle \vel \FM^{1/2},\{\FI_\pm-\FP_\pm\} f\rangle
=
\langle\FM^{1/2},\sourceG_\pm\rangle,
\label{m0}
\\
&& 
\pa_t [b_i+  \langle \vel_i \FM^{1/2},\{\FI_\pm-\FP_\pm\} f\rangle
]+\pa_i (a_\pm+2c)\mp 2E_i
\notag
\\
&&
\qquad 
+\na_x\cdot \langle \vel\vel_i
\FM^{1/2},\{\FI_\pm-\FP_\pm\} f\rangle
=
\langle\vel_i  \FM^{1/2},\sourceG_\pm {-\FL_\pm f}\rangle,
\label{m1}\\
&&
\pa_t \left[c+ \frac{1}{6}\langle (|\vel|^2-3)\FM^{1/2},\{\FI_\pm-\FP_\pm\} f\rangle \right]
+  \frac{1}{3} \na_x\cdot b
\notag
\\
&&\qquad 
+
\frac{1}{6}\na_x\cdot \langle (|\vel|^2-3)\vel\FM^{1/2},\{\FI_\pm-\FP_\pm\} f\rangle
 =\frac{1}{6}\langle (|\vel|^2-3)\FM^{1/2},\sourceG_\pm  {- \FL_\pm f}\rangle,
 \label{m2}
\end{eqnarray}
where we have set $\FI=[\FI_+,\FI_-]$ with $\FI_\pm f=f_\pm$. 

We next define the following high-order moment functions 
$\highG(f_\pm)=(\highG_{ij}(f_\pm))_{3\times 3}$ 
 and
$\highB(f_\pm)=(\highB_1(f_\pm),\highB_2(f_\pm),\highB_3(f_\pm))$ 
by
\begin{equation}
  \highG_{ij}(f_\pm) = \langle(\vel_i\vel_j-1)\FM^{1/2}, f_\pm\rangle,\ \
  \highB_i(f_\pm)=\frac{1}{10}\langle(|\vel|^2-5)\vel_i\FM^{1/2}, f_\pm\rangle.\label{def.gala}
\end{equation}
Further taking velocity integrations of \eqref{VPLlinear} with respect to the above high-order
moments one has
\begin{eqnarray}
&&\pa_t [\highG_{ii}( \{\FI_\pm-\FP_\pm\}f)+2c]+2\pa_i b_i
=\highG_{ii}(l_{\pm}+\sourceG_\pm),\label{m2ii}\\
&&\pa_t \highG_{ij}( \{\FI_\pm-\FP_\pm\}f) +\pa_j b_i+\pa_i b_j +\na_x\cdot
\langle \vel\FM^{1/2},\{\FI_\pm-\FP_\pm\} f\rangle
\notag
\\
&&\hspace{4cm}
=
\highG_{ij}(l_{\pm}+\sourceG_\pm)
+
\langle\FM^{1/2},\sourceG_\pm\rangle,\ \ i\neq j,\label{m2ij}
\\
&& 
\pa_t \highB_i( \{\FI_\pm-\FP_\pm\}f)+\pa_i c=\highB_i(l_{\pm}+\sourceG_\pm),
\label{m3}
\end{eqnarray}
where
\begin{equation}\label{def.l}
l_\pm =-\vel \cdot \na_x \{\FI_\pm-\FP_\pm\}f-\FL_\pm f.
\end{equation}
Here we  used the moment values of the normalized
global Maxwellian $\FM$:
\begin{eqnarray*}
&&\langle 1, \FM\rangle=1, \ \ 
\langle |\vel_j|^2, \FM\rangle=1,\ \ 
\langle |\vel|^2, \FM\rangle=3,\\
&&\langle |\vel_j|^2|\vel_m|^2, \FM\rangle=1, \ \ j\neq m,\\
&&\langle |\vel_j|^4, \FM\rangle=3,\ \ \langle |\vel|^2|\vel_j|^2,
\FM\rangle=5.
\end{eqnarray*}
Additionally to derive \eqref{m2ij} we have used \eqref{m0}.

\subsection{The linearized system with micro sources}
\label{sec.decayl}

In this section, we are concerned with time-decay properties of
solutions to the Cauchy problem on the linearized
Vlasov-Poisson-Landau system with microscopic sources.

Consider the Cauchy problem on the linearized system with a
microscopic source $\sourceG=\sourceG(t,x,\vel) = [\sourceG_+,
\sourceG_-]$ as in \eqref{VPLlinear} and \eqref{poissonLINEAR}, where $\sourceG=\{\FI-\FP\} \sourceG$ and $[f_0,\phi_0]$ satisfies the
equation \eqref{poissonLINEAR} initially.

For simplicity, we write
\begin{equation}
\notag
    U=[f,\phi],\ \ U_0=[f_0,\phi_0].
\end{equation}
Formally, the solution to the Cauchy problem \eqref{VPLlinear} and \eqref{poissonLINEAR} is denoted by
\begin{eqnarray}
&& U(t)=U^{I}(t)+U^{II}(t),\label{def.lu}\\
&& U^{I}(t)=\semiG(t)U_0,\ \ U^{I}=[u^I,\phi^I], \label{def.lu1}\\
&& U^{II}(t)=\int_0^t\semiG(t-s)[\sourceG(s),0,0]ds,\ \ U^{II}=[u^{II},\phi^{II}],\label{def.lu2}
\end{eqnarray}
where $\semiG(t)$ is the linear solution operator for the Cauchy problem
on the linearized homogeneous system corresponding to  \eqref{VPLlinear} and \eqref{poissonLINEAR} with $\sourceG=0$.

Notice that $U^{II}(t)$ is well-defined because $[\sourceG(s),0,0]$ for
any $0\leq s\leq t$ satisfies the compatibility condition \eqref{compatibilityC}
due to the fact that $\FP \sourceG(s)=0$ and hence
\begin{equation}  \label{compatibilityC}
    \int_{\R^3}\FM^{1/2}[\sourceG_+(s)-\sourceG_-(s)]d\vel=0.
\end{equation}
We introduce the norms
$\|\cdot\|_{\dot{\CH}^m_\wN}$, $\|\cdot\|_{\CZ_r^\wN}$ with $m\geq 0$, $\wN \in \R$, and $r\geq
1$ given by
\begin{equation}\label{brief.norm}
    \|U\|_{\dot{\CH}^m_\wN}^2\eqdef \|w_\wN f\|_{L^2_\vel(\dot{H}^m_x)}^2+\|\nabla_x \phi\|_{\dot{H}^m_x}^2,
    \quad
    \|U\|_{\CZ_r^\wN} \eqdef \|w_\wN f\|_{Z_r}+\| \nabla_x \phi \|_{L^r_x},
\end{equation}
for $U=[\solU,\phi]$, we set $\CL^2_\wN=\CH^0_\wN$ and $\CL^2=\CH^0_0$  as usual. 

The main result of this section is stated as follows.

\begin{theorem}\label{thm.ls}
Let $1\leq r\leq 2$,  $\ell \in\R$, and let $m\geq 0$.
$U$ is defined in \eqref{def.lu}, \eqref{def.lu1} and
\eqref{def.lu2} as the solution to the Cauchy problem \eqref{VPLlinear}.  
Let $\sigma_{r,m}$ be given by
\begin{equation}
\label{rateLIN}
    \sigma_{r,m}\eqdef \frac{3}{2}\left(\frac{1}{r}-\frac{1}{2}\right)+\frac{m}{2}.
\end{equation}
Then for any $\wE > 2 \sigma_{r,m}$ we obtain that
\begin{equation}
 \|\semiG(t) U_0\|_{\CHd^m_\wN}
\lesssim
(1+t)^{-\sigma_{r,m}}  
\|U_0\|_{\CHd^m_{\wN+\wE}  \cap Z_r^{\wN+\wE} },
 \label{thm.ls.1.soft}
\end{equation}
which will hold uniformly for any $t\geq 0$.
\end{theorem}

\subsubsection{A time-frequency Lyapunov inequality}\label{sec.sub.tfli}

In this subsection, we shall construct the desired time-frequency
Lyapunov functional as motioned before. The proof will be carried
out along the similar line as in \cite{DS-VPB}, but additional
efforts need to be made to take care of the weak dissipation of the
electromagnetic field.

{\bf Estimate on the micro dissipation}

The first step for the construction of the time-frequency Lyapunov
functional is to obtain the micro dissipation on the basis of the
coercivity property \eqref{coercLANDAU} of $\FL$. 

Thus \eqref{VPLlinear}-\eqref{poissonLINEAR} also reads with $q_1 = [+1, -1]$
\begin{equation}
\notag
    \left\{\begin{array}{l}
  \dis     \pa_t f+\vel\cdot\na_x f - 2E\cdot \vel \FM^{1/2}q_1 + \FL f = \sourceG,\\
-\Delta \phi  =a_+-a_-.
    \end{array}\right.
\end{equation}
Recall $E = -\nabla_x \phi$.
Notice that the continuity equation is also satisfied:
\begin{equation}\label{continEQ}
\partial_t \rho
+
\nabla _{x}\cdot \Jcont=0,
\end{equation}
where 
$$
\rho =\int \sqrt{\mu }[f_{+}-f_{-}]dv = a_+-a_-,
\quad
\Jcont=\int v\sqrt{\mu }[f_{+}-f_{-}]dv  =  \langle [\vel, - \vel ] \FM^{1/2},\{\FI-\FP\} f\rangle.
$$ 
This follows by subtracting 
\eqref{m0}$_+$ from \eqref{m0}$_-$ and using \eqref{compatibilityC}.

Taking the Fourier transform in $x$  gives
\begin{equation}\label{ls-1f}
    \left\{\begin{array}{l}
  \dis     \pa_t \hat{f}+i\vel\cdot k \hat{f} - 2 i \hat{\phi} k\cdot \vel \FM^{1/2}q_1+\FL \hat{f} = \hat{\sourceG},\\
    \dis     \pa_t \left( \widehat{a_+-a_-} \right)+i k \cdot \widehat{\Jcont}  = 0,\\
\dis |k|^2 \hat{\phi} =\widehat{a_+-a_-}.
    \end{array}\right.
\end{equation}
Then equation \eqref{ls-1f}$_1$  implies
\begin{equation}
\notag
    \frac{1}{2}\pa_t \nsm \hat{f} \nsm_{2}^2+\rmre \int_{\R^3}(\FL \hat{f}\mid \hat{f})d\vel
    -\rmre ( i \hat{\phi}  \mid k \cdot \widehat{\Jcont}  )=\rmre  {\int_{\R^3}(\hat{\sourceG}\mid \hat{f})d\vel}.
\end{equation}
We now observe from \eqref{ls-1f}$_2$-\eqref{ls-1f}$_3$ that
\begin{equation}
\notag
   \frac{1}{2}\pa_t  | \widehat{a_+-a_-} |^2+\rmre ( i k \cdot \widehat{\Jcont}  \mid |k|^2 \hat{\phi} )=0.
\end{equation}
Observe that $| \widehat{a_+-a_-} |^2 =  |k|^4 | \hat{\phi}  |^2$.
Since $\rmre ( i k \cdot \widehat{\Jcont}  \mid  \hat{\phi} )=\rmre ( i \hat{\phi}  \mid k \cdot \widehat{\Jcont}  )$, taking a summation of these two equalities gives
\begin{equation}
\notag
    \frac{1}{2}\pa_t \left(\nsm \hat{f} \nsm_{2}^2+ 2|k|^2 | \hat{\phi}  |^2 \right)+\rmre \int_{\R^3}(\FL \hat{f}\mid \hat{f})d\vel
    =
    \rmre  {\int_{\R^3}(\hat{\sourceG}\mid \hat{f})d\vel}.
\end{equation}
>From \eqref{coercLANDAU},  one has
\begin{equation}\label{diss-micr}
\pa_t \left(\nsm \hat{f} \nsm_{2}^2+2|k|^2 | \hat{\phi}  |^2\right)
+
\la |\{\FI-\FP\}\hat{f}|^2_\sigma
  \le
    \rmre  {\int_{\R^3}(\hat{\sourceG}\mid \hat{f})d\vel},
\end{equation}
for a properly small constant $0<\la\le \la_0$.

Here, we remark that equation \eqref{diss-micr} is the main estimate
for the construction of the time-frequency Lyapunov functional
$\CE(t,k)$. However, notice that for this time, the macroscopic part
$\FP\hat{f}$ and $\hat{E}$ are not included in the
dissipation rate of \eqref{diss-micr}. 

Next, based on the
macroscopic balance laws and high-order moment equation obtained in
Section \ref{sec.me}, we shall introduce some interactive functional
to capture the rest of the dissipation rate related to $\FP\hat{f}$ and $\hat{E}$.

{\bf Weighted energy estimates}\\ 
In this section our first goal is to prove the following instantaneous Lyapunov inequality
with a velocity weight $\wN \in \R$:
\begin{multline}\label{macroWeightINEQ}
\frac{d}{dt}\nsm w_{\wN}\{\FI-\FP\}\hat{\solU}(t,k)\nsm^2_{2} +
\la \nsm \{\FI-\FP\}\hat{\solU}(t,k)\nsm^2_{\sigma,\wN}
\\
\lesssim  |k|^2 \nsm \hat{\solU}\nsm^2_{2,-1}
+
  \nsm\{\FI-\FP\}\hat{\solU}\nsm_{L^2(B_{C})}^2
 +
 |k|^2 |\hat{\phi}|^2
+
\left| \langle w_{2\wN} \{\FI-\FP\}\hat{\sourceG},\{\FI-\FP\} \hat{\solU}\rangle \right|.
\end{multline}
To this end, we now split the solution $\hat{\solU}$ to equation \eqref{ls-1f}$_1$ into $\solU=\FP \solU + \{\FI-\FP\}\solU$  and then apply $\{\FI-\FP\}$ to the resulting equation:
\begin{multline}\notag
\pa_t \{\FI-\FP\}\hat{\solU} + \rmi \vel \cdot k \{\FI-\FP\}\hat{\solU} 
- 
\{\FI-\FP\} \left( 2 i \hat{\phi} k\cdot \vel \FM^{1/2}q_1 \right)
+
\FL\{\FI-\FP\} \hat{\solU} 
 \\
 =\{\FI-\FP\}\hat{\sourceG}
-\{\FI-\FP\}(\rmi \vel \cdot k \FP \hat{\solU} )
 +\FP (\rmi \vel \cdot k \{\FI-\FP\}\hat{\solU}).
\end{multline}
Multiply the last equation by $w_{2\wN} \{\FI-\FP\}\overline{\hat{\solU}}$ and integrate in $\vel$ to obtain
\begin{equation}
\label{app.vw.p05.int1.new}
\frac{1}{2}\frac{d}{dt}\nsm w_{\wN}\{\FI-\FP\}\hat{\solU}(t,k)\nsm^2_{2} +
 \rmre \langle w_{2\wN}\FL\{\FI-\FP\} \hat{\solU},\{\FI-\FP\} \hat{\solU}\rangle
 =
 \Ga_1 + \Ga_2 + \Ga_3,
\end{equation}
where $ \Ga_1 = \rmre \langle w_{2\wN} \{\FI-\FP\}\hat{\sourceG},\{\FI-\FP\} \hat{\solU}\rangle$ and
\begin{multline*}
 \Ga_2 =
-\rmre\left\langle
\{\FI-\FP\}(\rmi \vel \cdot k\FP \hat{\solU} ),
w_{2\wN}  \{\FI-\FP\}\hat{\solU}
\right\rangle
\\
 +\rmre \left\langle
 \FP (\rmi \vel \cdot k \{\FI-\FP\}\hat{\solU}),
w_{2\wN}  \{\FI-\FP\}\hat{\solU}  \right\rangle.
\end{multline*}
We furthermore define 
$$
\Ga_3 
= 
\rmre\left\langle
\{\FI-\FP\} \left( 2 i \hat{\phi} k\cdot \vel \FM^{1/2}q_1 \right),
w_{2\wN}  \{\FI-\FP\}\hat{\solU}
\right\rangle
$$
We will estimate each of the
three terms in \eqref{app.vw.p05.int1.new}.

As a result of the rapid decay in the coefficients of \eqref{form.p} we obtain
$$
\left| \Ga_2 \right| \le
\eta \nsm \{\FI-\FP\}\hat{\solU}\nsm^2_{\sigma, \wN}
+
C_\eta |k|^2 \left( \nsm w_{-\wE}\{\FI-\FP\}\hat{\solU}\nsm^2_{2}+\nsm \FP \hat{\solU}\nsm^2_{2} \right), 
$$
which holds for any small $\eta >0$ and any large $\wE>0$. 
Similarly 
$$
\left| \Ga_3 \right| \le
\eta \nsm w_{-\wE}\{\FI-\FP\}\hat{\solU}\nsm^2_{2}
+
C_\eta |k|^2 |\hat{\phi}|^2
$$
For the linear estimate, we invoke Lemma 9 in \cite{SG2} 
to achieve the coercive bound
$$
\rmre  \langle w_{2\wN}\FL\{\FI-\FP\} \hat{\solU},\{\FI-\FP\} \hat{\solU}\rangle
 \ge \la \nsm \{\FI-\FP\}\hat{\solU}\nsm^2_{\sigma, \wN}
 -
 C  \nsm\{\FI-\FP\}\hat{\solU}\nsm_{L^2(B_{C})}^2.
$$
We plug these last few estimates into \eqref{app.vw.p05.int1.new} to obtain \eqref{macroWeightINEQ}.

We furthermore remark, following the same procedure as above, that we get
\begin{equation}\label{macroWeightINEQ.noM}
\frac{1}{2}\frac{d}{dt}\nsm w_{\wN}\hat{\solU}(t,k)\nsm^2_{2} 
+
\la \nsm \hat{\solU}(t,k)\nsm^2_{\sigma, \wN}
\lesssim
   \nsm\hat{\solU}\nsm_{L^2(B_{C})}^2
   +
   |k|^2 |\hat{\phi}|^2
+
\left| \langle w_{2\wN} \hat{\sourceG}, \hat{\solU}\rangle \right|.
\end{equation}
In other words, if we multiply \eqref{ls-1f}$_1$ by $w_{2\wN}\overline{\hat{\solU}(t,k)}$, integrate in $\vel$ and use the same estimates as in the last case it follows that we obtain \eqref{macroWeightINEQ.noM}.

{\bf Estimate on the macro dissipation}

Let us apply those computations in Section \ref{sec.me}.
Taking the mean value of every
two equations with $\pm$ sign for \eqref{m0}, \eqref{m1}, \eqref{m2}
and noticing $\FP \sourceG=0$, one has
\begin{equation}\label{macro.1}
    \left\{
    \begin{array}{l}
      \dis \pa_t\left(\frac{a_++a_-}{2}\right)+\na_x\cdot b=0,\\
      \dis \pa_t b_i+\pa_i\left(\frac{a_++a_-}{2}+2c\right)+\frac{1}{2}\sum_{j=1}^3\pa_j\highG_{ij}(\{\FI-\FP\}f\cdot [1,1])=0,\\
      \dis \pa_t c+ \frac{1}{3}\na_x\cdot b +\frac{5}{6}\sum_{i=1}^3\partial_i \highB_i(\{\FI-\FP\}f \cdot [1,1])=0,
    \end{array}\right.
\end{equation}
for $1\leq i\leq 3$,

where moment functions $\highG(\cdot)$ and
$\highB(\cdot)$ are defined in \eqref{def.gala},  {and we used the following facts
\begin{eqnarray*}
&&\langle \FM^{1/2},g_\pm \rangle =\langle ([1,0]+[0,1])\FM^{1/2}, g\rangle=0,\\
&&\langle \vel_i\FM^{1/2},g_+ + g_- \rangle =\langle {[\vel_i,\vel_i]}\FM^{1/2}, g\rangle=0,\ \ 1\leq i\leq 3,\\
&&\langle \frac{1}{6}(|\vel|^2-3)\FM^{1/2},g_+ + g_- \rangle =\langle [|\vel|^2,|\vel|^2]\FM^{1/2}, g\rangle=0,
\end{eqnarray*}
due to $\FP g=0$ and likewise for $\FL \solU =[\FL_+  \solU, \FL_-  \solU ]$ due to $\FP \FL  \solU=0$.} Similarly, it
follows from \eqref{m2ii}, \eqref{m2ij} and \eqref{m3} that
\begin{equation}\label{macro.2}
    \left\{
    \begin{array}{l}
      \dis \pa_t \left[\frac{1}{2}\highG_{ij}(\{\FI-\FP\} \solU \cdot [1,1]) +2c\de_{ij}\right]\\
      \dis \hspace{1.5cm}+\pa_i b_j+\pa_j b_i
      =
      \frac{1}{2}\highG_{ij}((l_++l_-)+(\sourceG_++\sourceG_-)),\\
      \dis \frac{1}{2}\pa_t \highB_i(\{\FI-\FP\}  \solU \cdot [1,1])+\pa_i c=\frac{1}{2}\highB_i((l_++l_-)+(\sourceG_++\sourceG_-)),
    \end{array}\right.
\end{equation}
for $ 1\leq i,j\leq 3$, where $l_\pm$ is still defined in
\eqref{def.l}, and $\de_{ij}$ denotes as usual the Kronecker delta.

\begin{lemma}
There is a time-frequency functional  {$\CE_{\rm int}^{(1)}(t,k)$} defined by
\begin{eqnarray}
\notag
  \CE_{\rm int}^{(1)}(t,k) &=& \frac{1}{1+|k|^2}\sum_{i=1}^3\frac{1}{2} (i k_i \hat{c}\mid \highB_i(\{\FI-\FP\}\hat{\solU}\cdot [1,1]))\\
  \notag
  &&+\frac{\kappa_1}{1+|k|^2}\sum_{ {i,j}=1}^3(ik_i \hat{b}_j+ik_j\hat{b}_i\mid \frac{1}{2} 
  \highG_{ij}(\{\FI-\FP\}\hat{\solU}\cdot [1,1])+2 {\hat{c}}\de_{ij})
  \\
  && +\frac{\kappa_2}{1+|k|^2}\sum_{i=1}^3\left(ik_i\frac{\hat{a}_++\hat{a}_-}{2}\mid \hat{b}_i\right),
  \label{def.int1}
\end{eqnarray}
with two properly chosen constants $0<\kappa_2\ll\kappa_1\ll 1$ such that
\begin{multline}
\dis \pa_t \rmre \CE_{\rm int}^{(1)}(t,k)+\frac{\la |k|^2}{1+|k|^2} \left(|\widehat{a_++a_-}|^2+|\hat{b}|^2+|\hat{c}|^2\right)\\
\dis 
\lesssim
{\nsm \{\FI-\FP\}\hat{\solU}\nsm_{2,-m}^2+\nsm \hat{\sourceG}\nsm_{2,-m }^2},
\label{diss-macro+}
\end{multline}
holds for any $t\geq 0$ and $k\in \R^3$ and $m\ge 0$.
\end{lemma}

\begin{proof}
The proof can be found in Lemma 4.1 of \cite{D-Hypo}.  Although Lemma 4.1 in  \cite{D-Hypo} studies the pure Boltzmann equation without force terms our case is directly similar.  It follows from  \eqref{macro.1} and \eqref{macro.2}, which are otherwise not used following.
\end{proof}

In order to further obtain the dissipation rate related to
$\hat{a}_\pm$ from the formula
$$
|\hat{a}_+|^2+|\hat{a}_-|^2=\frac{|\widehat{a_++a_-}|^2}{2}+\frac{|\widehat{a_+-a_-}|^2}{2},
$$
we need to consider the dissipation of $\widehat{a_+-a_-}$. For
that, taking difference of two equations with $\pm$ sign for
\eqref{m0}, \eqref{m1} and also noticing  {$\FP \sourceG=0$ which implies $\langle \FM^{1/2}, g_\pm \rangle =0$}, one has
\begin{eqnarray}
&&\pa_t (a_+-a_-)+\na_x\cdot \Jcont=0,\label{m0-}\\
&&\pa_t \Jcont + \na_x (a_+-a_-)-4E+ \na_x\cdot \highG (\{\FI-\FP\}f\cdot q_1)\label{m1-}\\
\notag
&&\qquad\qquad\qquad\qquad\qquad\qquad\qquad= {\langle [\vel,-\vel] \FM^{1/2}, g - \FL \{\FI-\FP\} f \rangle}.
\end{eqnarray}
Note that here and hereafter
$
\left(\na_x\cdot \highG\right)_j (\cdot)
=
\partial_i \highG_{ij} (\cdot).
$
Together with
\begin{equation}\label{divE}
   - \Delta \phi =a_+-a_-,
\end{equation}
one has the following lemma.

\begin{lemma}\label{lem.a}
For any $t\geq 0$ and $k\in \R^3$,
it holds uniformly that
\begin{equation}
\dis \frac{\pa_t \rmre (\hat{\Jcont}\mid ik \widehat{({a}_+-{a}_-)})}{(1+|k|^2)} +\la |\widehat{a_+-a_-}|^2\\
\dis
\lesssim
{\nsm \{\FI-\FP\}\hat{\solU}\nsm_{2,-m}^2+\nsm \hat{\sourceG}\nsm_{2,-m }^2}.
\label{lem.a-1}
\end{equation}
This inequality will be true for any $m \ge 0$.
\end{lemma}

\begin{proof}
In fact, taking the
Fourier transform in $x$ for \eqref{m0-}, \eqref{m1-} and \eqref{divE} gives
\begin{equation}\label{lem.a-1.p1}
\left\{\begin{array}{l}
 \dis \pa_t \widehat{(a_+-a_-)}+ik\cdot \hat{\Jcont}=0,\\[3mm]
\dis \pa_t \hat{\Jcont}+ik \widehat{(a_+-a_-)}-4\widehat{E}+ik\cdot \highG(\{\FI-\FP\}\hat{\solU}\cdot q_1)\\[3mm]
\dis \qquad\qquad\qquad\qquad\qquad\qquad\qquad= {\langle [\vel,-\vel] \FM^{1/2}, \hat{\sourceG} -\FL \{\FI-\FP\} \hat{\solU} \rangle},\\[3mm]
\dis  |k|^2 \hat{\phi}=\widehat{a_+-a_-}.
\end{array}\right.
\end{equation}
On one hand, notice from \eqref{lem.a-1.p1}$_3$ (where  recall $\hat{E} = -ik\hat{\phi}$) that
\begin{equation}
\notag
    (ik \widehat{(a_+-a_-)}-4\widehat{E}\mid ik \widehat{(a_+-a_-)})=(|k|^2+4)|\widehat{a_+-a_-}|^2.
\end{equation}
On the other hand, it follows from \eqref{lem.a-1.p1}$_2$ that
\begin{eqnarray*}
&&  (ik \widehat{(a_+-a_-)}-4\widehat{E}\mid ik \widehat{(a_+-a_-)})\\
&& =(-\pa_t \hat{\Jcont} -ik\cdot \highG (\{\FI-\FP\}\hat{\solU}\cdot q_1) {+\langle [\vel,-\vel] \FM^{1/2}, \hat{g}-\FL \{\FI-\FP\} \hat{\solU} \rangle}\mid ik \widehat{(a_+-a_-)})\\
&&=-\pa_t (\hat{\Jcont}\mid ik \widehat{(a_+-a_-)})+(\hat{\Jcont}\mid ik \pa_t\widehat{(a_+-a_-)})\\
&&\ \ \ -\left(ik\cdot \highG(\{\FI-\FP\}\hat{\solU}\cdot q_1) {-\langle [\vel,-\vel] \FM^{1/2}, \hat{g}-\FL \{\FI-\FP\} \hat{\solU} \rangle}\mid ik \widehat{(a_+-a_-)}\right).
\end{eqnarray*}
Combining the above two equations
and using \eqref{lem.a-1.p1}$_1$, one has
\begin{multline*}
\pa_t (\hat{\Jcont}\mid ik \widehat{(a_+-a_-)})
+
(4+ |k|^2) ~ |\widehat{a_+-a_-}|^2
\\
=(\hat{\Jcont}\mid k ~k \cdot \hat{\Jcont})
\\
 -\left(ik\cdot \highG(\{\FI-\FP\}\hat{\solU}\cdot q_1) {-\langle [\vel,-\vel] \FM^{1/2}, \hat{g}-\FL \{\FI-\FP\} \hat{\solU} \rangle}\mid ik \widehat{(a_+-a_-)}\right).
\end{multline*}
It follows using Cauchy's inequality that
\begin{multline*}
\pa_t\rmre (\hat{\Jcont}\mid ik \widehat{(a_+-a_-)})+\la (1+|k|^2)|\widehat{a_+-a_-}|^2\\
\leq  |k\cdot \hat{\Jcont}|^2+C|k\cdot \highG(\{\FI-\FP\}\hat{\solU}\cdot q_1)|^2
 {+C|\langle [\vel,-\vel]\FM^{1/2},\hat{g}\rangle|^2}\\
 {+C|\langle [\vel,-\vel]\FM^{1/2},\FL \{\FI-\FP\}\hat{\solU}\rangle|^2}\\
\lesssim
{(1+|k|^2)}{\nsm \{\FI-\FP\}\hat{\solU}\nsm_{2,-m}^2+\nsm \hat{\sourceG}\nsm_{2,-m }^2}.
\end{multline*}
Notice further that from \eqref{divE}, the dissipation rate in
\eqref{lem.a-1} can be rewritten as
\begin{equation} \label{diss.f}
|\widehat{a_+-a_-}|^2= |k|^4 | \hat{\phi}|^2.
\end{equation}
Then \eqref{lem.a-1} holds by further dividing the previous inequality by $1+|k|^2$.
\end{proof}


{\bf Derivation of the  time-frequency Lyapunov inequality}

Now, we are in a position to prove

\begin{theorem}\label{thm.tfli}
Let $U=[f,\phi]$ be the solution to the Cauchy problem \eqref{VPLlinear} and \eqref{poissonLINEAR}
with $g=0$. Then for $\wN \in \R$ there is a time-frequency functional $\CE_\wN(t,k)$
such that
\begin{equation}\label{thm.tfli.1}
    \CE_\wN(t,k)\approx \| w_\wN \hat{\solU}\|_{2}^2+| k |^2 | \hat{\phi} |^2,
\end{equation}
such that for any $t\geq 0$ and $k\in \R^3$ we uniformly have
\begin{equation}\label{thm.tfli.2}
    \pa_t \CE_\wN (t,k)+\la \left( 1 \wedge |k|^2\right) \CE_{\wN-1}(t,k)\leq 0.
\end{equation}
Above we use the notation $1 \wedge |k|^2 \eqdef\min\{1,|k|^2\}$.
\end{theorem}

\begin{proof}
We initially consider the case when $\wN=0$.
Let
\begin{equation}\label{thm.tfli.p1}
    \CE_0 (t,k)\eqdef \nsm \hat{\solU} \nsm_{2}^2+2|k|^2 |\hat{\phi}|^2+\kappa_3\rmre (\CE_{\rm int}^{(1)}(t,k)+\CE_{\rm int}^{(2)}(t,k)),
\end{equation}
for a constant $\kappa_3>0$ to be determined later, where $\CE_{\rm int}^{(1)}(t,k)$ is given by \eqref{def.int1} and $\CE_{\rm int}^{(2)}(t,k)$ is denoted by 
\begin{equation}\label{def.int2}
  \CE_{\rm int}^{(2)}(t,k) \eqdef \frac{( {\hat{\Jcont}}\mid ik \widehat{({a}_+-{a}_-)})}{(1+|k|^2)}.
\end{equation}
One can then fix $\kappa_3>0$ small enough such that \eqref{thm.tfli.1} holds true for $\wN =0$. The rest is to check \eqref{thm.tfli.2}. In fact, the linear combination of \eqref{diss-micr}, \eqref{diss-macro+}, \eqref{lem.a-1} and \eqref{diss.f} according to the definition \eqref{thm.tfli.p1}
implies
\begin{equation*}
\dis\pa_t \CE_0(t,k)+\la \nsm\{\FI-\FP\}\hat{\solU}\nsm_{\sigma}^2
+\frac{\la |k|^2}{1+|k|^2}
\left(|\hat{a}_\pm|^2+|\hat{b}|^2+|\hat{c}|^2+ |k|^2 |\hat{\phi}|^2\right)
  \leq 0,
\end{equation*}
that is
\begin{equation}\label{big.step}
\dis\pa_t \CE_0 (t,k)+\la \nsm\{\FI-\FP\}\hat{\solU}\nsm_{\sigma}^2
+
\frac{\la |k|^2}{1+|k|^2}
\left(
\nsm\FP \hat{\solU}\nsm_{2,-1}^2 + |k|^2 |\hat{\phi}|^2 
\right) 
\leq 0,
\end{equation}
since one has
$$
|\hat{a}_\pm|^2+|\hat{b}|^2+|\hat{c}|^2\sim 
\nsm\FP \hat{\solU} \nsm_{2,-1}^2.
$$
We split into $|k| \le 1$ and $|k| \ge 1$; we use that 
when $|k|> 1$ then 
$
\frac{ |k|^2}{1+|k|^2} \ge \frac{1}{2},
$
and when $|k|\le 1$ it holds that
$
\frac{ |k|^2}{1+|k|^2} \ge \frac{|k|^2}{2}.
$
We also use the inequality 
$$
\CE_{0-1}(t,k) \lesssim \nsm\{\FI-\FP\}\hat{\solU}\nsm_{\sigma}^2 + \nsm\FP \hat{\solU}\nsm_{2,-1}^2 + |k|^2 |\hat{\phi}|^2.
$$  
Collecting these facts, then
\eqref{thm.tfli.2} follows in the case $\wN = 0$.

To do the weighted estimates (when $0 \ne \wN \in \R$), we use the energy splitting from \cite{sNonCutOp} 
 as follows.
With \eqref{thm.tfli.p1} we define 
\begin{equation}\notag
\begin{split}
    \CE_\wN^0(t,k)\eqdef &
    \ind_{|k|\le 1} \left(
    \CE(t,k)
    +\kappa_4  \nsm w_{\wN}\{\FI-\FP\}\hat{\solU}(t,k)\nsm_{2}^2 \right),
    \\
        \CE_\wN^1(t,k)\eqdef &
    \ind_{|k|> 1} \left(
    \CE(t,k)
    +\kappa_5  \nsm w_{\wN}\hat{\solU}(t,k)\nsm_{2}^2 \right),
    \end{split}
\end{equation}
where $ \kappa_4, \kappa_5>0$ will be determined just below.  Here $\ind_A$ is the usual indicator function of the set 
$A$.

Now when $|k| \le 1$, we choose $\kappa_4 >0$ sufficiently small so that from \eqref{thm.tfli.2} (with $\wN =0$) and \eqref{macroWeightINEQ}  (with $\sourceG =0$) and \eqref{big.step} one has  \eqref{thm.tfli.2} (when $0 \ne \wN \in \R$) when $|k| \le 1$ and $\partial_t \CE_\wN(t,k)$ is replaced by $\partial_t \CE_\wN^0(t,k)$.

Similarly when $|k| > 1$, we choose $\kappa_5 >0$ sufficiently small so that adding \eqref{thm.tfli.2} (with $\wN =0$) to \eqref{macroWeightINEQ.noM}  (with $\sourceG =0$) one has  \eqref{thm.tfli.2} (when $0 \ne \wN \in \R$) when $|k| > 1$ and $\partial_t \CE_\wN(t,k)$ is replaced by $\partial_t \CE_\wN^1(t,k)$.

Adding the two inequalities derived in the previous two paragraphs yields \eqref{thm.tfli.2} (when $0 \ne \wN \in \R$).
\end{proof}


\subsubsection{ Proof of time-decay of linear solutions}

The proof of time decay of linear solutions is now obtained using the interpolation argument as in Section 2.3 of \cite{sNonCutOp} and the energy inequalities in \eqref{thm.tfli.2}.   We omit it, since it is the same as in \cite{sNonCutOp}.

\section{Non-linear energy inequalities}\label{energy inequality}

In this section we prove a collection of non-linear energy inequalities for the local solution which was constructed in Section \ref{local}.  These energy inequalities will be used later on in Section \ref{proofmain} to prove that the solutions are in fact global in time.

Notice that in this section we use $\int g$ to mean $\int_{\R^3} dx ~g$ if $g=g(x)$ and alternatively it means 
$\int_{\R^3} dv \int_{\R^3} dx ~g$ when $g=g(x,v)$.

\begin{lemma}
\label{xenergy0} Let $f_{0}\in \testF$ and assume $f$ is the
solution constructed in Theorem \ref{localsolution} with $\mathcal{\tilde{E}}
_{2;2}(f)\leq M.$ Then 
\begin{eqnarray}
&&\frac{d}{dt}\left[ \int \sum_{\pm }\frac{(f_{\pm })^{2}}{2}+\int |\nabla \phi |^{2}\right] +\int
\langle Lf,f\rangle  \notag \\
&\lesssim & \sqrt{\mathcal{\tilde{E}}
_{2;2}(f)}\left( ||\{\FI-\FP\} f||_\si^2+||\na_x \phi||_2^2+||\na_x \FP f||_2^2 \right).\label{en0}
\end{eqnarray}
\end{lemma}
\begin{proof}
We use \eqref{vl} and also the continuity equation, e.g.  \eqref{continEQ}, to obtain
\begin{eqnarray*}
&&\frac{d}{dt}\left[ \int \sum_{\pm }\frac{(f_{\pm })^{2}}{2}+\int |\nabla \phi |^{2}\right] +\int
\langle Lf,f\rangle  \notag \\
&= &\int \sum_{\pm }f_{\pm }\left(  \pm\{E\cdot v\}f_\pm+\Ga_\pm(f,f)           \right).
\end{eqnarray*}
For $\int \sum_{\pm }f_{\pm } ( \pm\{E\cdot v\}f_\pm)$, we use the macro-micro decomposition to get
\begin{eqnarray*}
&&\int \sum_{\pm }\FP f_{\pm } ( \pm\{E\cdot v\}f_\pm)\\
&\lesssim & ||\FP f||_{ L^2_v L^6_x}||E||_{L^2_x}||vf||_{ L^2_v L^3_x}\\
&\lesssim & ||\na_x \FP f||_{L^2_x L^2_v}||E||_{L^2_x}\left(||vf||_{L^2_x L^2_v}+||~|v|~\na_x f||_{L^2_x L^2_v}\right)\\
&\lesssim & \sqrt{\mathcal{\tilde{E}}
_{2;2}(f)}\left( ||\na_x \phi||_2^2+||\na_x \FP f||_2^2 \right),
\end{eqnarray*}
and using \eqref{inftySob} and then \eqref{2Sob} we have
\begin{eqnarray*}
&&\int \sum_{\pm }\{\FI-\FP\} f_{\pm }  (\pm\{E\cdot v\}f_\pm)\\
&\lesssim & ||\langle v\rangle^{-1/2}\{\FI-\FP\} f||_{ L^2_v L^2_x}||E||_{L^\infty_x}||\langle v\rangle^{3/2}f||_{ L^2_v L^2_x}\\
&\lesssim & ||\{\FI-\FP\} f||_{\si}\left(||\na_x E||_{L^2_x}+||\na_x^2 E||_{L^2_x}\right)\sqrt{\mathcal{\tilde{E}}
_{2;2}(f)}\\
&\lesssim & ||\{\FI-\FP\} f||_{\si}\left(|| E||_{L^2_x}+||\na_x \FP f||_{2}\right)\sqrt{\mathcal{\tilde{E}}
_{2;2}(f)}\\
&\lesssim & \sqrt{\mathcal{\tilde{E}}
_{2;2}(f)}\left(||\{\FI-\FP\} f||_{\si}^2+ ||\na_x \phi||_2^2+||\na_x \FP f||_2^2 \right).
\end{eqnarray*}

For $\int \sum_{\pm }f_{\pm }\Ga_\pm(f,f)=\int \sum_{\pm }\left(\{\FI-\FP\}_{\pm}f \right)\Ga_\pm(f,f)$, we use the well-known expansion
\begin{equation}\label{decompGa}
\Ga(f,f)=\Ga(\FP f, \FP f)+\Ga(\{\FI-\FP\} f, \FP f)+\Ga(f, \{\FI-\FP\} f).
\end{equation}
Recall (\ref{form.p}), we have 
\begin{eqnarray*}
&&\int \sum_{\pm }\left(\{\FI-\FP\}_{\pm}f \right)\Ga_\pm(\FP f,\FP f)\\
&\lesssim & |||a|^2+|b|^2+|c|^2||_{ L^2_x}||\{\FI-\FP\}f||_\si\\
&\lesssim & |||a|+|b|+|c|||_{ L^6_x}|||a|+|b|+|c|||_{ L^3_x}||\{\FI-\FP\}f||_\si\\
&\lesssim & |||\na_x a|+|\na_x b|+|\na_x c|||_{ L^2_x}\sqrt{\mathcal{\tilde{E}}
_{2;2}(f)}||\{\FI-\FP\}f||_\si\\
&\lesssim & \sqrt{\mathcal{\tilde{E}}
_{2;2}(f)}\left(||\{\FI-\FP\} f||_{\si}^2+ ||\na_x \FP f||_2^2 \right).
\end{eqnarray*}

For $ \int \sum_{\pm }\left( \{\FI-\FP\}_{\pm}f \right) \Ga_{\pm}(\{\FI-\FP\} f, \FP f) $, from Proposition \ref{nonlinearTHM}, we have
\begin{eqnarray*}
&&\int \sum_{\pm }\left( \{\FI-\FP\}_{\pm}f \right) \Ga_\pm(\{\FI-\FP\} f,\FP f)\\
&\lesssim & ||\{\FI-\FP\} f||_\si || \{\FI-\FP\} f||_{L^3_x L^2_v} ||\FP f||_{L^6_x L^2_v}\\
&\lesssim & \sqrt{\mathcal{\tilde{E}}
_{2;2}(f)}\left(||\{\FI-\FP\} f||_{\si}^2+||\na_x \FP f||_2^2\right).
\end{eqnarray*}
Finally for  $ \int \sum_{\pm } \left( \{\FI-\FP\}_{\pm}f \right) \Ga_{\pm}( f, \{\FI-\FP\} f) $, from Proposition \ref{nonlinearTHM} we have
\begin{multline*}
\int \sum_{\pm } \left( \{\FI-\FP\}_{\pm}f \right)  \Ga_\pm( f,\{\FI-\FP\} f)
\\
\lesssim  ||\{\FI-\FP\} f||_\si || f||_{L^\infty_x L^2_v} ||\{\FI-\FP\} f||_\si
\lesssim  \sqrt{\mathcal{\tilde{E}}
_{2;2}(f)}||\{\FI-\FP\} f||_{\si}^2.
\end{multline*}
This completes the proof of (\ref{en0}).
\end{proof}

\begin{lemma}
\label{xenergy1} Let $f_{0}\in \testF$ and assume $f$ is the
solution constructed in Theorem \ref{localsolution} with $\mathcal{\tilde{E}}_{2;2}(f)\leq M.$ For any $K\geq 2$, and for $1\leq |\alpha |\leq K,$ we have
\begin{eqnarray}
&&\frac{d}{dt}\left[ \int \sum_{\pm }\frac{e^{\pm 2\phi }(\partial ^{\alpha
}f_{\pm })^{2}}{2}+\int |\nabla \partial ^{\alpha }\phi |^{2}\right] +\int
\langle L\partial ^{\alpha }f,\partial ^{\alpha }f\rangle  \notag \\
&\lesssim &\sum_{\pm }\int |\phi _{t}|(\partial ^{\alpha }f_{\pm })^{2}+%
\sqrt{\mathcal{\tilde{E}}%
_{K,K}(f)}\left(\sum_{1\leq |\alpha ^{\prime }|\leq K}||\partial ^{\alpha ^{\prime
}}f||_{\sigma }^{2}+||\na_x \phi||^2_2\right).  \label{pure2}
\end{eqnarray}%
For $|\alpha |=m\geq 3$ and $l\geq m$,
we have for any $\eta >0,$%
\begin{eqnarray}
&&\frac{d}{dt}\left[ \int \sum_{\pm }\frac{e^{\pm 2\phi }(\partial ^{\alpha
}f_{\pm })^{2}}{2}+\int |\nabla \partial ^{\alpha }\phi |^{2}\right] +\int
\langle L\partial ^{\alpha }f,\partial ^{\alpha }f\rangle  \notag \\
&\lesssim &\sum_{\pm }\int |\phi _{t}||\partial ^{\alpha }f_{\pm }|^{2}+%
\sqrt{\mathcal{\tilde{E}}%
_{2;2}(f)}\left(\int  |\partial ^{\alpha }f_{\pm }|_{\sigma }^{2}+||\na_x \phi||^2_2\right)+\eta \sum_{|\alpha
|=m}||\partial ^{\alpha ^{\prime }}f||_{\sigma }^{2}  \notag \\
&&+C_{m,\eta }[\mathcal{D}_{2;2}(f)\mathcal{\tilde{E}}_{m;l}(f)+\{1+\mathcal{\tilde{E}}%
_{m-1;l}(f)\}\mathcal{D}_{m-1;l}(f)].  \label{pure3}
\end{eqnarray}
\end{lemma}

\begin{proof}
By (\ref{vl}) and the continuity equation \eqref{continEQ}, we have
\begin{eqnarray}
&&\frac{d}{dt}\left[ \int \sum_{\pm }\frac{e^{\pm 2\phi }(\partial ^{\alpha
}f_{\pm }^{{}})^{2}}{2}+\int |\nabla \partial ^{\alpha }\phi |^{2}\right]
+\int \langle L\partial ^{\alpha }f,\partial ^{\alpha }f\rangle  \notag \\
&=&\sum_{\pm }\int e^{\pm 2\phi }\phi _{t}(\partial ^{\alpha }f_{\pm })^{2}
\label{x1} \\
&&+2\sum_{\pm }\int e^{\pm 2\phi }\nabla _{x}\partial ^{\alpha }\phi \cdot v%
\sqrt{\mu }\partial ^{\alpha }f_{\pm }(e^{\pm 2\phi }-1)  \label{x2} \\
&&+\sum_{\pm }\int (1-e^{\pm 2\phi })\partial ^{\alpha }f_{\pm }L_{\pm
}\partial ^{\alpha }f  \label{x3} \\
&&+\sum_{\pm }\int e^{\pm 2\phi }\partial ^{\alpha }f_{\pm }\partial
^{\alpha }\Gamma _{\pm }(f,f)  \label{x4} \\
&&+\sum_{\pm ,\alpha _{1}<\alpha }C_{\alpha }^{\alpha _{1}}\int e^{\pm 2\phi
}\partial ^{\alpha }f_{\pm }\partial ^{\alpha -\alpha _{1}}\nabla _{x}\phi
\cdot \nabla _{v}\partial ^{\alpha _{1}}f_{\pm }  \label{x5} \\
&&-\sum_{\pm ,\alpha _{1}<\alpha }C_{\alpha }^{\alpha _{1}}\int e^{\pm 2\phi
}\partial ^{\alpha }f_{\pm }\partial ^{\alpha -\alpha _{1}}\nabla _{x}\phi
\cdot v\partial ^{\alpha _{1}}f_{\pm }.  \label{x6}
\end{eqnarray}
We have $|1-e^{\pm \phi
}|\lesssim ||\phi ||_{\infty }\lesssim \sqrt{\mathcal{\tilde{E}}_{2;2}(f)}\lesssim 
\sqrt{M}$. We observe that 
\begin{eqnarray}
\int e^{\pm 2\phi }\phi _{t}(\partial ^{\alpha }f_{\pm })^{2}
&\lesssim& \int |\phi _{t}|(\partial ^{\alpha }f_{\pm })^{2},\label{x1bound}\\
\int e^{\pm 2\phi }\nabla _{x}\partial ^{\alpha }\phi \cdot v%
\sqrt{\mu }\partial ^{\alpha }f_{\pm }(e^{\pm 2\phi }-1) &\lesssim& 
\sqrt{\mathcal{\tilde{E}}_{2;2}(f)}\left(||\na_x \partial^{\alpha} \phi||^2_2+||\partial^{\alpha } f||_{\sigma }^2\right).\label{x2bound}
\end{eqnarray}
By Lemma 5 of \cite{G1}, we have 
\begin{equation}\label{x3bound}
(\ref{x3})\lesssim \sqrt{\mathcal{\tilde{E}}_{2;2}(f)}||\partial ^{\alpha }f||_{\sigma }^{2}.
\end{equation}
Now the rest of the proof for (\ref{pure3}) is the same as the proof of Lemma 14 in \cite{G0}. 

Now we focus on the proof of (\ref{pure2}). We first deal with (\ref{x5}) and (\ref{x6}). We have
\begin{eqnarray}
&&(\ref{x5}) + (\ref{x6})\notag\\
&\lesssim &\int |\partial ^{\alpha }f_{\pm }|_{\sigma }|\partial ^{\alpha
-\alpha _{1}}\nabla _{x}\phi ||\langle v\rangle ^{3/2}\partial ^{\alpha
_{1}}f_{\pm }|_{2}dx.  \label{unibound}
\end{eqnarray}

When $|\al|=1$ and $\al_1=0$, we place $\partial^\al\na_x\phi$ in $L^\infty_x$ and place the other two terms in (\ref{unibound}) in $L^2_x L^2_v$. Since $||\partial^\al\na_x\phi||_\infty\lesssim ||\na_x^3\phi||_2^2+||\na_x^4 \phi||_2^2\lesssim ||\na_x f||_\si^2+||\na_x^2 f||_\si^2$, we have the desired estimate in this case.

When $|\al|=2$ and $\al_1=0$, we place the three factors in the expression of (\ref{unibound})  in $L^2_x L^2_v$, $L^4_x $ and $L^4_x L^2_v$ respectively. This gives us the desired estimate in this case.

When $|\al|=2$ and $|\al_1| \ge 1$, we place the three factors in the expression of (\ref{unibound})  in $L^2_x L^2_v$, $L^\infty_x $ and $L^2_x L^2_v$ respectively. This gives us the desired estimate in this case.

When $|\al|\geq 3$ and $\al_1=0$, we place the three factors in the expression of (\ref{unibound})  in $L^2_x L^2_v$, $L^4_x $ and $L^4_x L^2_v$ respectively. This gives us the desired estimate in this case.

When $|\al|\geq 3$ and $K\geq |\al_1|\geq 1$, we place the three factors in the expression of (\ref{unibound}) in $L^2_x L^2_v$, $L^\infty_x $ and $L^2_x L^2_v$ respectively. This gives us the desired estimate in this case.

Now we turn to the estimate of (\ref{x4}). We use (\ref{decompGa}) again.   Recall (\ref{form.p}), we have 
\begin{eqnarray*}
&&\int \sum_{\pm } \left(\{\FI-\FP\}_{\pm}\pa^\al f \right)\pa^\al\Ga_\pm(\FP f,\FP f)\\
&\lesssim &
\sum_{\al_1 \le \al}
 || ~|\pa^{\al-\al_1}[a,b,c]|~|\pa^{\al_1}[a,b,c]|~||_{ L^2_x}
||\{\FI-\FP\}\pa^\al f||_{\si}.
\end{eqnarray*}
When $\al_1=\al$, we place $\pa^{\al-\al_1}[a,b,c]$ in $L^\infty_x$ and place $\pa^{\al_1}[a,b,c]$ in $L^2_x$. Since $$||[a,b,c]||_{L^\infty_x}\lesssim\sqrt{\mathcal{\tilde{E}}_{2;2}(f)},$$
 and 
 $$||\pa^{\al}[a,b,c]||_{L^2_x}\lesssim||\pa^\al f ||_\si,$$
 we have the desired estimate.
 
 When $1\leq |\al-\al_1|\leq K-1$, we place $\pa^{\al-\al_1}[a,b,c]$ in $L^3_x$ and place $\pa^{\al_1}[a,b,c]$ in $L^6_x$. Since $$||\pa^{\al-\al_1}[a,b,c]||_{L^3_x}\lesssim\sqrt{\mathcal{\tilde{E}}_{K;K}(f)},$$
 and 
 $$||\pa^{\al_1}[a,b,c]||_{L^6_x}\lesssim\sum_{|\al'|=|\al_1|+1}||\pa^{\al'}f ||_\si,$$
 we have the desired estimate.
 
 When $\al_1=0$, we place $\pa^{\al-\al_1}[a,b,c]$ in $L^2_x$ and place $\pa^{\al_1}[a,b,c]$ in $L^\infty_x$. This is the same as the first case. So we have the desired estimate.

For $ \int \sum_{\pm }\left(\pa^\al\{\FI-\FP\}_{\pm} f \right)\pa^\al \Ga_{\pm}(\{\FI-\FP\} f, \FP f) $, from Proposition \ref{nonlinearTHM}, we have
\begin{eqnarray*}
&&\int \sum_{\pm }\left( \pa^\al\{\FI-\FP\}_{\pm} f \right) \pa^\al \Ga_{\pm}(\{\FI-\FP\} f, \FP f)\\
&\lesssim & ||\pa^\al\{\FI-\FP\}  f||_{\si} || ~ |\pa^{\al-\al_1}\{\FI-\FP\} f|_{L^2_v} |\pa^{\al_1}\FP f|_{ L^2_v}~||_{L^2_x}.
\end{eqnarray*}

When $\al_1=\al$, we place $\pa^{\al-\al_1}\{\FI-\FP\} f$ in $L^\infty_x$ and place $\pa^{\al_1}\FP f$ in $L^2_x$. Since 
$$|||\{\FI-\FP\} f|_{L^2_v}||_{L^\infty_x}\lesssim\sqrt{\mathcal{\tilde{E}}_{2;2}(f)},$$
 and 
 $
 ||~|\pa^{\al}\FP f|_{L^2_v}~||_{L^2_x}\lesssim||\pa^\al f ||_\si,
 $
 we have the desired estimate.
 
 When $1\leq |\al-\al_1|\leq K-1$, we place $\pa^{\al-\al_1}\{\FI-\FP\} f$ in $L^3_x$ and place $\pa^{\al_1}\FP f$ in $L^6_x$. Since $$
 \| ~ |\pa^{\al-\al_1}\{\FI-\FP\} f|_{L^2_v}~ \|_{L^3_x}
 \lesssim\sqrt{\mathcal{\tilde{E}}_{K;K}(f)},
 $$
 and 
 $
 \| ~ |\pa^{\al_1}\FP f|_{L^2_v}~\|_{L^6_x}\lesssim\sum_{|\al'|=|\al_1|+1}||\pa^{\al'}f ||_\si,
 $
 we have the desired estimate.
 
 When $\al_1=0$, we place $\pa^{\al-\al_1}\{\FI-\FP\} f$ in $L^2_x$ and place $\pa^{\al_1}\FP f$ in $L^\infty_x$. Since
 $$
 \| ~ |\pa^{\al}\{\FI-\FP\} f|_{L^2_v}~\|_{L^2_x}
 \lesssim\sqrt{\mathcal{\tilde{E}}_{K;K}(f)},$$
 and 
$
\|~ |\FP f |_{L^2_v}~\|_{L^\infty_x}
\lesssim
||\na_x \FP f||_2+||\na_x^2 \FP f||_2,
$ 
 we have the desired estimate.

Finally for  $ \int \sum_{\pm }\left( \pa^\al\{\FI-\FP\}_{\pm}f \right)\pa^\al\Ga_{\pm}( f, \{\FI-\FP\} f) $, from Proposition \ref{nonlinearTHM} we have
\begin{eqnarray*}
&& \int \sum_{\pm }\pa^\al\{\FI-\FP\}_{\pm}f\pa^\al\Ga_{\pm}( f, \{\FI-\FP\} f)\\
&\lesssim & ||\pa^\al\{\FI-\FP\} f||_\si || | \pa^{\al-\al_1}f|_{ L^2_v} |\pa^{\al_1}\{\FI-\FP\} f|_\si||_{L^2_x}.
\end{eqnarray*}
When $\al_1=\al$, we place $ \pa^{\al-\al_1}f$ in $L^\infty_x$ and place $\pa^{\al_1}\{\FI-\FP\} f$ in $L^2_x$. Since
$$|||f|_{L^2_v}||_{L^\infty_x}\lesssim \sqrt{\mathcal{\tilde{E}}_{K;K}(f)},$$
and
$|||\pa^{\al}\{\FI-\FP\} f|_{\si}||_{L^2_x}\lesssim||\pa^\al\{\FI-\FP\} f||_\si ,$
 we have the desired estimate.
 
 When $1\leq |\al_1|\leq K-1$, we place $ \pa^{\al-\al_1}f$ in $L^4_x$ and place $\pa^{\al_1}\{\FI-\FP\} f$ in $L^4_x$. Since
$$|||\pa^{\al-\al_1}f|_{L^2_v}||_{L^4_x}\lesssim \sqrt{\mathcal{\tilde{E}}_{K;K}(f)},$$
and
$|||\pa^{\al_1}\{\FI-\FP\} f|_\si||_{L^4_x}\lesssim\sum_{1\leq |\al'|\leq K}||\pa^{\al'}\{\FI-\FP\} f||_\si ,$
 we have the desired estimate.

When $\al_1=0$, we place $ \pa^{\al-\al_1}f$ in $L^2_x$ and place $\pa^{\al_1}\{\FI-\FP\} f$ in $L^\infty_x$. Since
$$|||\pa^{\al}f|_{L^2_v}||_{L^2_x}\lesssim \sqrt{\mathcal{\tilde{E}}_{K;K}(f)},$$
and
$|||\{\FI-\FP\} f|_\si||_{L^\infty_x}\lesssim\sum_{1\leq |\al'|\leq K}||\pa^{\al'}\{\FI-\FP\} f||_\si ,$
 we have the desired estimate.
 
 This completes the proof of (\ref{pure2}).
\end{proof}

\begin{lemma}
\label{xenergy0I-P} Let $f_{0}\in \testF$ and assume $f$ is the
solution constructed in Theorem \ref{localsolution} with $\mathcal{\tilde{E}}
_{2;2}(f)\leq M.$ Then 
\begin{multline}
\frac{d}{dt}\left[ \int \sum_{\pm }\frac{(\{\FI-\FP\}f_{\pm })^{2}}{2}+\int |\nabla \phi |^{2}\right] +\int
\langle Lf,f\rangle
\\
\lesssim  \sqrt{\mathcal{\tilde{E}}
_{2;2}(f)}\left( ||\{\FI-\FP\} f||_\si^2+||\na_x\{\FI-\FP\} f||_\si^2+||\na_x \phi||_2^2+||\na_x \FP f||_2^2 \right)
\\
+
\eta ||\{\FI-\FP\} f||_\si^2+C_\eta \left(||\na_x \FP f||_2^2+||\na_x\{\FI-\FP\} f||_\si^2\right).
\label{enI-P0}
\end{multline}
Above, $\eta>0$ may be any small number.
\end{lemma}

\begin{proof}
By (\ref{vl}) and the continuity equation \eqref{continEQ}, we have
\begin{eqnarray*}
&&\frac{d}{dt}\left[ \int \sum_{\pm }\frac{(\{\FI-\FP\}f_{\pm })^{2}}{2}+\int |\nabla \phi |^{2}\right] +\int
\langle Lf,f\rangle  \notag \\
&= &\int \sum_{\pm }\{\FI-\FP\}f_{\pm }\{\FI-\FP\}\left(  \pm\{E\cdot v\}f_\pm+\Ga_\pm(f,f)-v\cdot\na_x f_\pm\mp E\cdot\na_vf_\pm           \right).
\end{eqnarray*}
For $\int \sum_{\pm }\left( \{\FI-\FP\}f_{\pm } \right) \{\FI-\FP\}(  \pm\{E\cdot v\}f_\pm)$, we have
\begin{eqnarray}
&&\int \sum_{\pm }\{\FI-\FP\}f_{\pm }\{\FI-\FP\}(  \pm\{E\cdot v\}f_\pm)\notag\\
&\lesssim &   \left(\int \sum_{\pm }|\langle v \rangle^{1/2}\{\FI-\FP\}(vf_{\pm })|^2\right)^{1/2}\notag\\
&&\cdot \left(\int \sum_{\pm }|\langle v \rangle^{-1/2}\{\FI-\FP\}f_{\pm }|^2+||E||^2_{L_x^\infty}\right)\notag\\
&\lesssim &  \sqrt{\mathcal{\tilde{E}}
_{2;2}(f)}\left( ||\{\FI-\FP\} f||_\si^2+||\na_x \phi||_2^2+||\na_x \FP f||_2^2\right).\label{I-P1}
\end{eqnarray}
In the last line above we have used \eqref{inftySob} and \eqref{2Sob} as in
$$||E||^2_{L_x^\infty}\lesssim ||\na_x E||^2_{L^2_x} +||\na_x^2 E||^2_{L^2_x}\lesssim || E||^2_{L^2_x} +||\na_x^2 E||^2_{L^2_x}\lesssim ||\na_x \phi||_2^2+||\na_x \FP f||_2^2.$$
For $\int \sum_{\pm }\{\FI-\FP\}f_{\pm }\{\FI-\FP\}( \Ga_\pm(f,f))$, the estimate is in the proof of Lemma \ref{xenergy0}.

For $\int \sum_{\pm }\{\FI-\FP\}f_{\pm }\{\FI-\FP\}( v\cdot\na_x f_\pm)$, we have
\begin{eqnarray}
&&\int \sum_{\pm }\{\FI-\FP\}f_{\pm }\{\FI-\FP\}( v\cdot\na_x f_\pm)\notag\\
&= &  \int \sum_{\pm }\{\FI-\FP\}f_{\pm }\{\FI-\FP\}( v\cdot\na_x \FP f_\pm)\notag\\
&&+\int \sum_{\pm }\{\FI-\FP\}f_{\pm }\{\FI-\FP\}( v\cdot\na_x \{\FI-\FP\}f_\pm)\notag\\
&= &  \int \sum_{\pm }\{\FI-\FP\}f_{\pm }\{\FI-\FP\}( v\cdot\na_x \FP f_\pm)\notag\\
&&-\int \sum_{\pm }\{\FI-\FP\}f_{\pm }\FP( v\cdot\na_x \{\FI-\FP\}f_\pm).\label{I-P3}
\end{eqnarray}
We give bounds for the two terms above. For the first term in \eqref{I-P3}, we have
\begin{eqnarray}
&&\int \sum_{\pm }\{\FI-\FP\}f_{\pm }\{\FI-\FP\}( v\cdot\na_x \FP f_\pm)\notag\\
&\lesssim &   \frac{1}{2}\eta \int \sum_{\pm }|\{\FI-\FP\}f_{\pm }\langle v\rangle^{-1/2}|^2\notag\\
&&+ \frac{1}{2}C_\eta \int \sum_{\pm }|\langle v\rangle^{1/2}\{\FI-\FP\}( v\cdot\na_x \FP f_\pm)|^2\notag\\
&\lesssim &   \frac{1}{2}\eta ||\{\FI-\FP\} f||_\si^2+ \frac{1}{2}C_\eta ||\na_x \FP f||_2^2.\label{I-P3-1}
\end{eqnarray}
For the second term in \eqref{I-P3}, we have
\begin{eqnarray}
&&\int \sum_{\pm }\{\FI-\FP\}f_{\pm }\FP( v\cdot\na_x \{\FI-\FP\}f_\pm)\notag\\
&\lesssim & \frac{1}{2} \eta \int \sum_{\pm }|\{\FI-\FP\}f_{\pm }\langle v\rangle^{-1/2}|^2\notag\\
&&+\frac{1}{2}C_\eta \int \sum_{\pm }|\langle v\rangle^{1/2}\FP( v\cdot\na_x \{\FI-\FP\}f_\pm)|^2\notag\\
&\lesssim &  \frac{1}{2}\eta ||\{\FI-\FP\} f||_\si^2+\frac{1}{2}C_\eta ||\na_x \{\FI-\FP\} f||_\si^2.\label{I-P3-2}
\end{eqnarray}
We finally turn to $\int \sum_{\pm }\{\FI-\FP\}f_{\pm }\{\FI-\FP\}\left( \mp E\cdot\na_vf_\pm           \right).$  We have
\begin{eqnarray}
&&\int \sum_{\pm }\{\FI-\FP\}f_{\pm }\{\FI-\FP\}\left( \mp E\cdot\na_vf_\pm           \right)\notag\\
&\lesssim &  \left(\int \sum_{\pm }|\langle v \rangle^{1/2}\left( \mp E\cdot\na_vf_\pm           \right)|^2\right)^{1/2}\notag\\
&&\cdot\left(\int \sum_{\pm }|\langle v \rangle^{-1/2}\{\FI-\FP\}f_{\pm }|^2+||E||^2_{L_x^\infty}\right)^{1/2}
\notag\\
&\lesssim & \sqrt{\mathcal{\tilde{E}}
_{2;2}(f)}\left( ||\{\FI-\FP\} f||_\si^2+||\na_x \phi||_2^2+||\na_x \FP f||_2^2\right)\label{I-P4}.
\end{eqnarray}
Combining \eqref{I-P1}, \eqref{I-P3-1}, \eqref{I-P3-2}, \eqref{I-P4} and the estimate of  $\int \sum_{\pm }\{\FI-\FP\}f_{\pm }\{\FI-\FP\}( \Ga_\pm(f,f))$ in the proof of Lemma \ref{xenergy0} gives \eqref{enI-P0}. This completes of proof. 
\end{proof}

Lemma \ref{xenergy0}, \ref{xenergy1} and \ref{xenergy0I-P} deal with pure energy  inequalities without weight. Now we turn to energy inequalities with a weight. 

\begin{lemma}
\label{xenergyK} Let $f_{0}\in \testF$ and assume $f$ is the solution constructed in Theorem \ref{localsolution} with 
$\mathcal{\tilde{E}}_{2;2}(f)\leq M$.  For any $K\geq 2$, and for $1\leq |\alpha |\leq K$, $w(\al,\be)=\langle v\rangle^{2(K-|\al|-|\be|)}$, we have%
\begin{eqnarray}
&&\frac{d}{dt} \int \sum_{\pm }\frac{e^{\pm 2\phi }w^2(\partial ^{\alpha
}f_{\pm })^{2}}{2} +||\partial^\al f||^2_{\si,w}\notag \\
&\lesssim &\sum_{\pm }\int \{||\phi _{t}||_{\infty}+||\na_x \phi ||_{\infty}\}w^2(\partial ^{\alpha }f_{\pm })^{2}+%
\sqrt{\mathcal{\tilde{E}}_{K,K}(f)}\mathcal{\tilde{\tilde{D}}}_{K,K}(f)\notag\\
&&+\eta||\partial^\al f||_{\si,w}^2+C_\eta||\partial^\al f||_\si^2+\eta ||\na_x \phi||^2_2,  \label{xenergyKbound}
\end{eqnarray}
where $\eta>0$ can be arbitrary small.
\end{lemma}

\begin{proof}
We are going to use the computations in \eqref{local1} through \eqref{local11}. The difference here is that we have $f^{n+1}=f^n=f$, so we don't have \eqref{local10} and \eqref{local11}. Moreover since we don't have a derivative with respect to the momentum variable $v$, we don't have \eqref{local2}. For \eqref{local1}, \eqref{local5} and \eqref{local7}, we use the estimates \eqref{ax}, \eqref{dfield} and \eqref{ineqlocal7}. Now we turn to \eqref{local3} and \eqref{local4}.  From the proof of Lemma 8 in \cite{G0}, we have 
\begin{multline}\label{field0}
-\int w^{2}\partial ^{\alpha }f_{\pm }\partial ^{\alpha -\alpha
_{1}}\nabla _{x}\phi \cdot \nabla _{v}\partial ^{\alpha _{1}}f_{\pm
} 
 \\
\lesssim \int |w\langle v\rangle ^{-1/2}\partial ^{\alpha }f_{\pm
}\partial ^{\alpha -\alpha _{1}}\nabla _{x}\phi \cdot \langle v\rangle
^{-2[|\alpha |-|\alpha _{1}|-1]}w(\alpha _{1},0 )\langle v\rangle
^{-3/2}\nabla _{v}\partial ^{\alpha _{1}}f_{\pm }|  
 \\
\lesssim 
\int |\partial ^{\alpha }f_{\pm }|_{\sigma ,w}|\partial
^{\alpha -\alpha _{1}}\nabla _{x}\phi ||\partial ^{\alpha
_{1}}f_{\pm }|_{\sigma ,\frac{w(\alpha _{1},0 )}{\langle v\rangle
^{2[|\alpha |-|\alpha _{1}|-1]}}}dx,  
\end{multline}
and 
\begin{eqnarray*}
& &|\int w^{2}\partial ^{\alpha }f_{\pm }\partial ^{\alpha
-\alpha _{1}}\nabla _{x}\phi \cdot v\partial ^{\alpha _{1}}f_{\pm
}]|\\
&\leq &\int |w\langle v\rangle ^{-1/2}\partial ^{\alpha }f_{\pm
}\partial ^{\alpha -\alpha _{1}}\nabla _{x}\phi \cdot \frac{w(\alpha
_{1},0 )}{\langle v\rangle ^{2[|\alpha |-|\alpha _{1}|-1]}}\langle
v\rangle ^{-\frac{3}{2}}\partial^{\alpha _{1}}f_{\pm }| \\
&\leq &C_{m}\int |\partial ^{\alpha }f_{\pm }|_{\sigma ,w(\alpha
,0 )}|\partial ^{\alpha -\alpha _{1}}\nabla _{x}\phi ||\partial^{\alpha _{1}}f_{\pm }|_{\sigma ,\frac{w(\alpha _{1},\beta )}{\langle
v\rangle ^{2[|\alpha |-|\alpha _{1}|-1]}}}.
\end{eqnarray*}
When $|\al_1|\geq 1$, we just use the estimates given in the proof of Lemma 8 in \cite{G0}. When $\al_1=0$, we split $\na_v f_{\pm}$ into $\na_v \FP f_{\pm}+\na_v\{\FI-\FP\} f_{\pm}$ and split $f_{\pm}$ into $ \FP f_{\pm}+\{\FI-\FP\} f_{\pm}$. For the terms involving $\na_v\{\FI-\FP\} f_{\pm}$ and $\{\FI-\FP\} f_{\pm}$, we use the same technique as in the proof of Lemma 8 in \cite{G0}. All these estimates don't involve the term $||\FP f||_{\si,w}$, so they are boumded by $\sqrt{\mathcal{\tilde{E}}_{K,K}(f)}\mathcal{\tilde{\tilde{D}}}_{K,K}(f)$. The remaining terms to be estimated are
$$
\int |w\langle v\rangle ^{-1/2}\partial ^{\alpha }f_{\pm
}\partial ^{\alpha -\alpha _{1}}\nabla _{x}\phi \cdot \langle v\rangle
^{-2[|\alpha |-|\alpha _{1}|-1]}w(\alpha _{1},0 )\langle v\rangle
^{-3/2}\nabla _{v}\partial ^{\alpha _{1}}\FP f_{\pm }|
$$
and 
$\int |w\langle v\rangle ^{-1/2}\partial ^{\alpha }f_{\pm
}\partial ^{\alpha -\alpha _{1}}\nabla _{x}\phi \cdot \frac{w(\alpha
_{1},0 )}{\langle v\rangle ^{2[|\alpha |-|\alpha _{1}|-1]}}\langle
v\rangle ^{-\frac{3}{2}}\partial^{\alpha _{1}}\FP f_{\pm }| $.

We place $w\langle v\rangle ^{-1/2}\partial ^{\alpha }f_{\pm}$  and  $\partial ^{\alpha -\alpha _{1}}\nabla _{x}\phi $ in $L^2_vL^2_x$ and $L^3_x$ respectively. We then place $$\langle v\rangle
^{-2[|\alpha |-|\alpha _{1}|-1]}w(\alpha _{1},0 )\langle v\rangle
^{-3/2}\nabla _{v}\partial ^{\alpha _{1}}\FP f_{\pm }$$
 and 
 $$\frac{w(\alpha
_{1},0 )}{\langle v\rangle ^{2[|\alpha |-|\alpha _{1}|-1]}}\langle
v\rangle ^{-\frac{3}{2}}\partial^{\alpha _{1}}\FP f_{\pm }$$ in $L^2_vL^6_x$. With \eqref{6Sob}, we have the estimate:
\begin{equation}
\eqref{local3}+\eqref{local4}\lesssim \sqrt{\mathcal{\tilde{E}}_{K,K}(f)}\mathcal{\tilde{\tilde{D}}}_{K,K}(f).\label{fieldnonlinear2}
\end{equation}
Now we turn to \eqref{local6}. By Lemma 8 in \cite{SG1}, we have
\begin{equation}
\eqref{local6}\lesssim \sqrt{\mathcal{\tilde{E}}_{2,2}(f)}\mathcal{\tilde{\tilde{D}}}_{K,K}(f).\label{smalll2}
\end{equation}
Next we turn to \eqref{local8}. Again by Lemma 8 in \cite{SG1}, we have
\begin{equation}
\eqref{local8}\lesssim \eta ||\partial^\al f||^2_{\si, w}+C_\eta ||\partial^\al f||_\si^2.\label{kbound2}
\end{equation}
Finally we turn to \eqref{local9}. By Proposition \ref{nonlinearTHM}, we have
\begin{equation}
\eqref{local9}\lesssim\int_{\R^3}\sum_{\alpha_1 \le \alpha}
|\partial^{\alpha _1}f|_{2,w(0,0)^{-1}}
|\partial ^{\alpha -\alpha _1}f|_{\sigma ,w}
|\partial ^\alpha f|_{\sigma ,w}dx.
\end{equation}

When $|\al_1|\leq |\al|-2$, we place $|\partial^{\alpha _1}f|_{2,w(0,0)^{-1}}$, $|\partial ^{\alpha -\alpha _1}f|_{\sigma ,w}$ and $|\partial ^\alpha f|_{\sigma ,w}$ in spaces $L^\infty_x$, $L^2_x$ and $L^2_x$ respectively.

When $|\al_1|= |\al|-1$, we place $|\partial^{\alpha _1}f|_{2,w(0,0)^{-1}}$, $|\partial ^{\alpha -\alpha _1}f|_{\sigma ,w}$ and $|\partial ^\alpha f|_{\sigma ,w}$ in spaces $L^4_x$, $L^4_x$ and $L^2_x$ respectively.

When $|\al_1|= |\al|$, we split $\partial ^{\alpha -\alpha _1}f$ into $\partial ^{\alpha -\alpha _1}\FP f+\partial ^{\alpha -\alpha _1}\{\FI-\FP\}f$. For $\partial ^{\alpha -\alpha _1}\{\FI-\FP\}f$, we place $|\partial^{\alpha _1}f|_{2,w(0,0)^{-1}}$, $|\partial ^{\alpha -\alpha _1}\{\FI-\FP\}f|_{\sigma ,w}$ and $|\partial ^\alpha f|_{\sigma ,w}$ in spaces $L^2_x$, $L^\infty_x$ and $L^2_x$ respectively.
For $\partial ^{\alpha -\alpha _1}\FP f$, we realize that
$$|\partial^{\alpha _1}f|_{2,w(0,0)^{-1}}\lesssim |\partial^{\alpha _1}f|_{\si,w}$$
and 
$|\partial ^{\alpha -\alpha _1}\FP f|_{\sigma ,w}\lesssim |\partial ^{\alpha -\alpha _1}\FP f|_{2}.$
So we still place $|\partial^{\alpha _1}f|_{2,w(0,0)^{-1}}$, $|\partial ^{\alpha -\alpha _1}\FP f|_{\sigma ,w}$ and $|\partial ^\alpha f|_{\sigma ,w}$ in spaces $L^2_x$, $L^\infty_x$ and $L^2_x$ respectively. These estimates give
\begin{equation}
\eqref{local9}\lesssim \sqrt{\mathcal{\tilde{E}}_{K,K}(f)}\mathcal{\tilde{\tilde{D}}}_{K,K}(f).\label{gammabound}
\end{equation}
Combing the estimates given above completes the proof of this lemma.
\end{proof}

\begin{lemma}
\label{I-PenergyK}Assume $f_{0}\in \testF$ and assume $f$ is the
solution constructed in Theorem \ref{localsolution} with $\mathcal{\tilde{E}}%
_{2;2}(f)\leq M.$ For any $K\geq 2$, $w(\al,\be)=\langle v\rangle^{2(K-|\al|-|\be|)}$, $|\alpha |+ |\beta| \leq K$ with $|\al|\leq K-1$ we have%
\begin{multline}\label{I-PK}
\frac{d}{dt} \int \sum_{\pm }\frac{e^{\pm 2\phi }w^2(\partial_{\beta} ^{\alpha
}\{\FI-\FP\}f_{\pm })^{2}}{2}+||\partial_{\beta}^\al\{\FI-\FP\}f ||_{\si, w(\al, \beta)}^2 
\\
\lesssim \{||\phi _{t}||_{\infty}+||\na_x \phi ||_{\infty}\} \sum_{\pm }\int w^2(\partial_\beta ^{\alpha }\{\FI-\FP\}f_{\pm })^{2}
\\
+\mathcal{\tilde{E}}_{K,K}(f)\mathcal{\tilde{D}}_{K,K}(f)+\eta\mathcal{\tilde{D}}_{K,K}(f)+\mathcal{\tilde{D}}_{|\al|+|\beta|-1,K}(f)
\\
+
\eta||\partial^\al_\be\{\FI-\FP\}f_{\pm}||^2_{\si,w}+C_\eta||\na_x^{|\al|+1}f_\pm||^2_\si+C_{\eta}||\partial^\al\na_x\phi||^2_2,  
\end{multline}
where $\eta>0$ can be arbitrarily small.
\end{lemma}

\begin{proof}
By \eqref{vl}, we have using the macro-micro decomposition that
\begin{eqnarray}
&&\partial _{t}\{\FI-\FP\}f_{\pm }+v\cdot \nabla _{x}\{\FI-\FP\}f_{\pm }\pm E\cdot \nabla _{v}\{\FI-\FP\}f_{\pm }\notag\\
&&+\{\FI-\FP\}(\mp2\{E\cdot v\}\sqrt{\mu})\notag\\
&=&-
L_{\pm }f 
\pm \{E\cdot v\}\{\FI-\FP\}f_{\pm }+\Gamma _{\pm
}(f,f)\notag\\
&&\mp E\cdot \FP(vf_{\pm }) \pm \{E\cdot v\}\FP f_{\pm } -v\cdot \nabla _{x}\FP f_{\pm }+\FP(v\cdot \nabla _{x}f_{\pm })\notag\\
&&\mp E\cdot \nabla _{v}\FP f_{\pm }\pm E\cdot \FP(\nabla _{v}f_{\pm }).\label{vlI-P}
\end{eqnarray}
So we have

\begin{eqnarray}
&&\frac{d}{dt}\left\{ \int \frac{e^{\pm 2\phi }w^{2}(\partial
_{\beta }^{\alpha }\{\FI-\FP\}f_{\pm })^{2}}{2}\right\} -\int \langle
w^{2}\partial _{\beta }^{\alpha }A\{\FI-\FP\}f,\partial _{\beta }^{\alpha
}\{\FI-\FP\}f_{\pm }\rangle  \label{local1'} \\
&=&-\int e^{\pm 2\phi }w^{2}\delta _{\beta }^{\mathbf{e}%
_{i}}\partial _{\beta -\mathbf{e}_{i}}^{\alpha +\mathbf{e}_{i}}\{\FI-\FP\}f_{\pm }\partial _{\beta }^{\alpha }\{\FI-\FP\}f_{\pm }  \label{local2'} \\
&&\pm \sum_{\alpha _{1}<\alpha }C_{\alpha }^{\alpha _{1}}\int e^{\pm
2\phi }w^{2}\partial _{\beta }^{\alpha }\{\FI-\FP\}f_{\pm }\partial
^{\alpha -\alpha _{1}}\nabla _{x}\phi \cdot \nabla _{v}\partial _{\beta
}^{\alpha _{1}}\{\FI-\FP\}f_{\pm }  \label{local3'} \\
&&\mp \sum_{\alpha _{1}<\alpha }C_{\alpha }^{\alpha _{1}}\int e^{\pm
2\phi }w^{2}\partial _{\beta }^{\alpha }\{\FI-\FP\}f_{\pm }\partial
^{\alpha -\alpha _{1}}\nabla _{x}\phi\cdot \partial _{\beta }[v\partial
^{\alpha _{1}}\{\FI-\FP\}f_{\pm }]  \label{local4'} \\
&&\mp \int [\frac{2(l-|\alpha |-|\beta |)}{1+|v|^{2}}\nabla _{x}\phi
\cdot v-\phi _{t}]e^{\pm 2\phi }w^{2}(\partial _{\beta
}^{\alpha }\{\FI-\FP\}f_{\pm })^{2}  \label{local5'} \\
&&+\mathcal{C} \label{local6''}
\end{eqnarray}
where $\mathcal{C}$ in  \eqref{local6''} is given by 
\begin{eqnarray}
\mathcal{C}&=&\int w^{2}(e^{\pm 2\phi }-1)\partial _{\beta }^{\alpha }\{\FI-\FP\}f_{\pm }\partial _{\beta }^{\alpha }A\{\FI-\FP\}f_{\pm }  \label{local6'} \\
&&+2\int w^2e^{\pm 2\phi}\partial^\al_\be\{\FI-\FP\}f_\pm\partial^\al_\be\{\FI-\FP\}(\mp\{E\cdot v\}\sqrt{\mu})\label{local7'}\\
&&+\int w^{2}e^{\pm 2\phi }\partial _{\beta }^{\alpha }K_{\pm
}\{\FI-\FP\}f_{\pm }\partial _{\beta }^{\alpha }\{\FI-\FP\}f_{\pm }  \label{local8'} \\
&&+\int w^{2}e^{\pm 2\phi }\partial _{\beta }^{\alpha }\Gamma _{\pm
}(f,f)\partial _{\beta }^{\alpha }\{\FI-\FP\}f_{\pm }  \label{local9'} \\
&&\mp \sum_{\alpha _{1}<\alpha }C_{\alpha }^{\alpha _{1}}\int\{ e^{\pm
2\phi }w^{2}\partial _{\beta }^{\alpha }\{\FI-\FP\}f_{\pm }\partial
^{\alpha -\alpha _{1}}\nabla _{x}\phi\notag\\
&&\cdot \partial _{\beta }[v\partial
^{\alpha _{1}}\FP f_{\pm }-\FP(v\partial
^{\alpha _{1}}f_{\pm })]\}  \label{local4''} \\
&&\pm \sum_{\alpha _{1}<\alpha }C_{\alpha }^{\alpha _{1}}\int \{e^{\pm
2\phi }w^{2}\partial _{\beta }^{\alpha }\{\FI-\FP\}f_{\pm }\partial
^{\alpha -\alpha _{1}}\nabla _{x}\phi \notag\\
&&\cdot [\nabla _{v}\partial _{\beta
}^{\alpha _{1}}\FP f_{\pm }-\partial _{\beta
}^{\alpha _{1}}\FP(\na_v f_{\pm })] \} \label{local3''} \\
&&+\int e^{\pm
2\phi }w^{2}\partial _{\beta }^{\alpha }\{\FI-\FP\}f_{\pm }\partial _{\beta }^{\alpha } [\FP(v\cdot\na_x f_{\pm})
-v\cdot\na_x\FP f_{\pm}]. 
   \label{local2''}
\end{eqnarray}

Above 
\begin{equation*}
\delta _{\beta }^{\mathbf{e}_{i}}=1\text{ \ if \ \ \ }\mathbf{e}_{i}\leq
\beta ;\text{ \ }\ \text{or \ \ }\delta _{\beta }^{\mathbf{e}_{i}}=0,\text{
\ otherwise}. 
\end{equation*}
Now for \eqref{local1'} through \eqref{local9'}, we can use the estimates given in the proof of Lemma \ref{mixnenergy}. Note that we have
$$ \mathcal{D}_{K,K}(\{\FI-\FP\}f) \leq \mathcal{\tilde{D}}_{K,K}(f). $$
For \eqref{local4''} and \eqref{local3''}, we have two cases.

 When $|\al_1|\geq 1$, we have 
$$||\langle v\rangle^{-2}w\partial _{\beta }^{\alpha }\{\FI-\FP\}f_{\pm }||_{L^2_vL^2_x}\leq ||\partial _{\beta }^{\alpha }\{\FI-\FP\}f_{\pm }||_{\si,w},$$  
$$||\partial
^{\alpha -\alpha _{1}}\nabla _{x}\phi ||_{L^\infty_x}\leq \sqrt{\mathcal{\tilde{E}}_{K,K}(f)},$$
and 
\begin{eqnarray*}
&&||w \langle v\rangle^{2}  \partial _{\beta }[v\partial
^{\alpha _{1}}\FP f_{\pm }-\FP(v\partial
^{\alpha _{1}}f_{\pm })]||_{L^2_vL^2_x}\\
&&+||w \langle v\rangle^{2}  \partial _{\beta }[\nabla _{v}\partial _{\beta
}^{\alpha _{1}}\FP f_{\pm }-\partial _{\beta
}^{\alpha _{1}}\FP(\na_v f_{\pm })]||_{L^2_vL^2_x} \\
&\leq &\sqrt{\mathcal{\tilde{D}}_{K,K}(f)}.        
\end{eqnarray*}

When $\al_1=0$, we have
$$||\partial
^{\alpha -\alpha _{1}}\nabla _{x}\phi ||_{L^3_x}\leq \sqrt{\mathcal{\tilde{E}}_{K,K}(f)},$$
and 
\begin{eqnarray*}
&&||w \langle v\rangle^{2}  \partial _{\beta }[v\partial
^{\alpha _{1}}\FP f_{\pm }-\FP(v\partial
^{\alpha _{1}}f_{\pm })]||_{L^2_vL^6_x}\\
&&+||w \langle v\rangle^{2}  \partial _{\beta }[\nabla _{v}\partial _{\beta
}^{\alpha _{1}}\FP f_{\pm }-\partial _{\beta
}^{\alpha _{1}}\FP(\na_v f_{\pm })]||_{L^2_vL^6_x} \\
&\leq &||\na_x \FP f||_{\si, w}\\
&\leq& \sqrt{\mathcal{\tilde{D}}_{K,K}(f)},        
\end{eqnarray*}
where we have used \eqref{6Sob}. These give the estimate of \eqref{local4''} and \eqref{local3''}.

We finally turn to \eqref{local2''}. We have 
$$||\langle v\rangle^{-2}w\partial _{\beta }^{\alpha }\{\FI-\FP\}f_{\pm }||_{L^2_vL^2_x}\leq ||\partial _{\beta }^{\alpha }\{\FI-\FP\}f_{\pm }||_{\si,w},$$ and 
$$||w \langle v\rangle^{2}\partial _{\beta }^{\alpha } [\FP(v\cdot\na_x f_{\pm})-v\cdot\na_x\FP f_{\pm}]||_{L^2_vL^2_x}\leq ||\na_x^{|\al|+1} f_{\pm }||_{\si}.$$
The above bounds give the desired estimate of \eqref{local2''}. Combining the estimates of \eqref{local1'} through \eqref{local2''} completes the proof of this lemma.
\end{proof}

The next three lemmas allow us to include more terms into our dissipation in our energy inequality. So in our final energy inequality, we can absorb many terms in the bounds  of the lemmas  already given in this section. Recall that 
\begin{equation}
\{\partial _{t}+v\cdot \nabla _{x}\}f_{\pm }\mp 2\{E\cdot v\}\sqrt{\mu }%
+L_{\pm }f = N_{\pm }(f),  
 \label{N}
\end{equation}
where
$$
N_{\pm }(f) \equiv \mp E\cdot \nabla _{v}f_{\pm }\pm \{E\cdot v\}f_{\pm }+\Gamma _{\pm
}(f,f).  
$$
The following lemma enables us to include the term $||\na_x \phi||^2_2$ in our dissipation.

\begin{lemma}
\label{dissE}Let $f_{0}\in \testF$ and assume $f$ is the
solution constructed in Theorem \ref{localsolution} with $\mathcal{\tilde{E}}%
_{2;2}(f)\leq M.$ Then we have
\begin{eqnarray}
&&\frac{d}{dt} \int -\langle [v,-v]\mu^{1/2},\{\FI-\FP\}f\rangle\cdot E dx+\lambda ||E||^2_2 \notag \\
&\lesssim &||\{\FI-\FP\}f||_{\si}^{2}+||\na_x f||^2_{\si}+||\langle [v,-v]\mu^{1/2},[N_+(f), N_-(f)]\rangle||_{L^2_x}^2,  \label{Ebound}
\end{eqnarray}
for some positive $\lambda$.
\end{lemma}
\begin{proof}
By \eqref{m1}, we have
\begin{eqnarray}
&& 
\pa_t [b_i+  \langle \vel_i \FM^{1/2},\{\FI_\pm-\FP_\pm\} f\rangle
]+\pa_i (a_\pm+2c)\mp 2E_i
\notag
\\
&&
\qquad 
+\na_x\cdot \langle \vel\vel_i
\FM^{1/2},\{\FI_\pm-\FP_\pm\} f\rangle
=
\langle\vel_i  \FM^{1/2},N_{\pm }(f) {-\FL_\pm f}\rangle.
\label{m1N}
\end{eqnarray}
We take the difference between the above equation involving $f_+$ and the above equation involving $f_-$. We use the vector form for $v=[v_1, v_2, v_3]$. Thus
\begin{eqnarray}
&&\partial_t\langle [v,-v]\mu^{1/2},\{\FI-\FP\}f\rangle+\na_x (a_+-a_-)-4E\notag\\
&&+\na_x\cdot\highG(\{\FI-\FP\}f\cdot[1,-1])
=\langle [v,-v]\mu^{1/2},N(f)\rangle.\label{eqnE}
\end{eqnarray}
Here $N(f) = [N_+(f),N_-(f)]$.  
Now we take the dot product of \eqref{eqnE} with $E$ and integrate with respect to $x$:
\begin{eqnarray}
&&\frac{d}{dt} \int -\langle [v,-v]\mu^{1/2},\{\FI-\FP\}f\rangle\cdot E dx+ 4 ||E||^2_2 \notag \\
&=&\int \na_x(a_+-a_-)\cdot E dx +\int \na_x\cdot\highG(\{\FI-\FP\}f\cdot[1,-1])\cdot E dx\notag\\
&&-\int \langle [v,-v]\mu^{1/2},N(f)\rangle\cdot E dx-\int \langle [v,-v]\mu^{1/2},\{\FI-\FP\}f\rangle\cdot \partial_t E dx.
\label{Eb1}
\end{eqnarray}
Notice that
\begin{eqnarray}
&&|\int \na_x (a_+-a_-)\cdot E dx +\int \na_x\cdot\highG(\{\FI-\FP\}f\cdot[1,-1])\cdot E dx\notag\\
&&-\int \langle [v,-v]\mu^{1/2},N(f)\rangle\cdot E dx|\notag\\
&\le& \eta ||E||^2_2+  C_\eta  ||\na_x f||^2_{\si} + C_\eta ||\langle [v,-v]\mu^{1/2},N(f)\rangle||_{L^2_x}^2.\label{Eb2}
\end{eqnarray}
For the last term in \eqref{Eb1}, we use the continuity equation \eqref{continEQ}. Observe that
\begin{eqnarray}
|\int \langle [v,-v]\mu^{1/2},\{\FI-\FP\}f\rangle\cdot \partial_t E dx|
\lesssim ||\{\FI-\FP\}f||_{\si}^2.\label{Eb3}
\end{eqnarray}
combining \eqref{Eb1}, \eqref{Eb2} and \eqref{Eb3} gives \eqref{Ebound}. This completes the proof.
\end{proof}

Next we want to include the term $||\na_x\FP f||_2^2$ in our dissipation.
Now we redo \eqref{macro.1} with $g_{\pm}=N_{\pm}$ and get
\begin{equation}\label{macro.1'}
    \left\{
    \begin{array}{l}
       \dis \pa_t\left(\frac{a_++a_-}{2}\right)+\na_x\cdot b=\frac{1}{2}\langle \mu^{1/2},N_++N_-\rangle,\\
       \dis \pa_t b_i+\pa_i\left(\frac{a_++a_-}{2}+2c\right)+\frac{1}{2}\sum_{j=1}^3\pa_j\highG_{ij}(\{\FI-\FP\}f\cdot [1,1])\\
       \dis \qquad \qquad \qquad \qquad \qquad \qquad \qquad \qquad =\frac{1}{2}\langle v_i\mu^{1/2},N_++N_-\rangle,\\
       \dis \pa_t c+ \frac{1}{3}\na_x\cdot b +\frac{5}{6}\sum_{i=1}^3\partial_i \highB_i(\{\FI-\FP\}f \cdot [1,1])\\
       \dis \qquad \qquad \qquad \qquad \qquad \qquad \qquad \qquad  =\frac{1}{12}\langle (|v|^2-3)\mu^{1/2},N_++N_-\rangle,
    \end{array}\right.
\end{equation}
for $1\leq i\leq 3$.

We then redo \eqref{macro.2} with $g_{\pm}=N_{\pm}$ and get
\begin{equation}\label{macro.2'}
    \left\{
    \begin{array}{l}
      \dis \pa_t \left[\frac{1}{2}\highG_{ij}(\{\FI-\FP\} \solU \cdot [1,1]) +2c\de_{ij}\right]+\pa_i b_j+\pa_j b_i\\
      \dis \hspace{1.5cm}
      =
      \frac{1}{2}\highG_{ij}((l_++l_-)+(N_++N_-))+\frac{1}{2}\de_{ij}\langle \mu^{1/2},N_++N_-\rangle,\\
      \dis \frac{1}{2}\pa_t \highB_i(\{\FI-\FP\}  \solU \cdot [1,1])+\pa_i c=\frac{1}{2}\highB_i((l_++l_-)+(N_++N_-)),
    \end{array}\right.
\end{equation}
for $ 1\leq i,j\leq 3$.

We also redo the computation for $a_+-a_-$ and get
\begin{eqnarray}
&&\pa_t (a_+-a_-)+\na_x\cdot \Jcont=\langle \mu^{1/2},N_+-N_-\rangle,\label{m0-'}\\
&&\pa_t \Jcont + \na_x (a_+-a_-)-4E+ \na_x\cdot \highG (\{\FI-\FP\}f\cdot q_1)\notag\\
\notag
&&\qquad\qquad\qquad\qquad\qquad\qquad\qquad= {\langle [\vel,-\vel] \FM^{1/2}, N - \FL \{\FI-\FP\} f \rangle}.
\end{eqnarray}
Following the proof of \eqref{diss-macro+} and \eqref{lem.a-1} we have the following lemma:
\begin{lemma}\label{naPlemma}
 Let $f_{0}\in \testF$ and assume $f$ is the
solution constructed in Theorem \ref{localsolution} with $\mathcal{\tilde{E}}%
_{2;2}(f)\leq M.$ There is a time-frequency functional  {$\CE_{\rm int}^{(1)}(t,k)$} defined by
\begin{eqnarray}
\notag
  \CE_{\rm int}^{(1)}(t,k) &=& \frac{1}{1+|k|^2}\sum_{i=1}^3\frac{1}{2} (i k_i \hat{c}\mid \highB_i(\{\FI-\FP\}\hat{\solU}\cdot [1,1]))\\
  \notag
  &&+\frac{\kappa_1}{1+|k|^2}\sum_{ {i,j}=1}^3(ik_i \hat{b}_j+ik_j\hat{b}_i\mid \frac{1}{2} 
  \highG_{ij}(\{\FI-\FP\}\hat{\solU}\cdot [1,1])+2 {\hat{c}}\de_{ij})
  \\
  && +\frac{\kappa_2}{1+|k|^2}\sum_{i=1}^3\left(ik_i\frac{\hat{a}_++\hat{a}_-}{2}\mid \hat{b}_i\right),
  \label{def.int1'}
\end{eqnarray}
with two properly chosen constants $0<\kappa_2\ll\kappa_1\ll 1$ such that
\begin{eqnarray}
&& \pa_t \rmre \CE_{\rm int}^{(1)}(t,k)+\frac{\pa_t \rmre (\hat{\Jcont}\mid ik \widehat{({a}_+-{a}_-)})}{(1+|k|^2)}+\frac{\la |k|^2}{1+|k|^2} \left(|\widehat{a_+}|^2+|\widehat{a_-}|^2+|\hat{b}|^2+|\hat{c}|^2\right)\notag\\
 &\lesssim&
\nsm \{\FI-\FP\}\hat{\solU}\nsm_{\si}^2+\sum_{1\leq i,j,l\leq3}\frac{1}{1+|k|^2}[\nsm\langle \hat{N},\mu^{1/2}\rangle\nsm^2  \notag\\
&&+\nsm\langle \hat{N},v_i\mu^{1/2}\rangle\nsm^2+\nsm\langle \hat{N},v_iv_j\mu^{1/2}\rangle\nsm^2\notag+\nsm\langle \hat{N},v_iv_jv_l\mu^{1/2}\rangle\nsm^2],\label{diss-macro+'}
\end{eqnarray}
holds for any $t\geq 0$ and $k\in \R^3$.
\end{lemma}
The next lemma deals with the term 
$$
\nsm\langle \hat{N},\mu^{1/2}\rangle\nsm^2 +\nsm\langle \hat{N},v_i\mu^{1/2}\rangle\nsm^2+\nsm\langle \hat{N},v_iv_j\mu^{1/2}\rangle\nsm^2\notag+\nsm\langle \hat{N},v_iv_jv_l\mu^{1/2}\rangle\nsm^2.
$$

\begin{lemma}\label{Nlemma}
Let $f_{0}\in \testF$ and assume $f$ is the solution constructed in Theorem \ref{localsolution} with $\mathcal{\tilde{E}}%
_{2;2}(f)\leq M.$ Let $N$ be defined by \eqref{N}.  Let $\tilde{\mu}$ denote any of the functions $\mu^{1/2}, v_i\mu^{1/2}, v_iv_j\mu^{1/2}, v_iv_jv_l\mu^{1/2}$. For $\al$ with $|\al|\leq 2$, we have
\begin{eqnarray}
||\langle \partial^\al{N},\tilde{\mu}\rangle||^2_{L^2_x}\lesssim \tilde{\mathcal{E}}_{2,2}(f)(||\na_x f||^2_\si+||\na_x^2 f||^2_\si).\label{Nbound}
\end{eqnarray}
\end{lemma}

\begin{proof}
We first deal with $\mp E\cdot \nabla _{v}f_{\pm }\pm \{E\cdot v\}f_{\pm }$. We have 
$$
\partial^\al(\mp E\cdot \nabla _{v}f_{\pm }\pm \{E\cdot v\}f_{\pm })=\sum_{\al_1\leq\al}C_{\al}^{\al_1}\partial^{\al_1}(\mp E)\cdot \partial^{\al-\al_1}(\nabla _{v}f_{\pm }-vf_{\pm }).
$$

When $| \al-\al_1 |=0$, we place $\partial^{\al_1}(\mp E)$ and $\partial^{ \al-\al_1 }(\nabla _{v}f_{\pm }-vf_{\pm })\tilde{\mu}$ in the spaces $ L^2_x$ and $L^1_vL^\infty_x$ respectively. Notice that we use \eqref{inftySob} to handle $\partial^{\al-\al_1}(\nabla _{v}f_{\pm }-vf_{\pm })\mu^{1/4}$ in this case.

When $| \al-\al_1 |\geq 1$, we place $\partial^{\al_1}(\mp E)$ and $\partial^{\al-\al_1}(\nabla _{v}f_{\pm }-vf_{\pm })\tilde{\mu}$ in the spaces $ L^\infty_x$ and $L^1_vL^2_x$ respectively. Notice that we use \eqref{inftySob} to handle $\partial^{\al_1}(\mp E)$ in this case. 

Now we deal with $\Gamma _{\pm}(f,f)$. By Lemma 7 of \cite{G1}, we have
\begin{eqnarray}
|\langle\Gamma(\partial^{\al_1}f, \partial^{\al-\al_1}f),\tilde{\mu}\rangle|
&\lesssim&|\partial^{\al_1}f|_2|\partial^{\al-\al_1}f|_\si.\label{expreGa}
\end{eqnarray}

When $| \al-\al_1 |=0$, we place $|\partial^{\al_1}f|_2$ and $|\partial^{\al-\al_1}f|_\si$ in the spaces $L^2_x$ and $L^\infty_x$ respectively and use \eqref{inftySob} to handle $|\partial^{\al-\al_1}f|_\si$.

When $| \al-\al_1 |=1$, we place $|\partial^{\al_1}f|_2$ and $|\partial^{\al-\al_1}f|_\si$ in the spaces $L^4_x$ and $L^4_x$ respectively and use \eqref{4Sob} to handle both terms.

When $| \al-\al_1 |=2$, we place $|\partial^{\al_1}f|_2$ and $|\partial^{\al-\al_1}f|_\si$ in the spaces $L^\infty_x$ and $L^2_x$ respectively and use \eqref{inftySob} to handle $|\partial^{\al_1}f|_2$.

This completes the proof of the lemma.
\end{proof}

Now we are ready to give our energy inequalities.

\begin{proposition}\label{E33}
 Let $f_{0}\in \testF$ and assume $f$ is the
solution constructed in Theorem \ref{localsolution} with $\mathcal{\tilde{E}}%
_{3;3}(f)\leq M.$ Let $ w(\al,\be)=\langle v\rangle^{2(3-|\al|-|\be|)}$. Then there exist continuous energy functionals  
$\zeta(t)$ and $\zeta^h(t)$ such that 
$\zeta\simnew\tilde{\mathcal{E}}_{3,3}$
 and 
\begin{eqnarray*}
\zeta^h &\simnew&\sum_{|\alpha |+|\beta |\leq 3}\sum_{\pm}
||\partial _{\beta }^{\alpha }\{\FI_{\pm}-\FP_{\pm}\}f(t)||_{2,w(\alpha ,\beta )}^{2}
\\
&&+||E(t)||_2^2
+
\sum_{1\leq |\alpha |\leq 3}\sum_{\pm}||\partial^{\alpha }\FP_{\pm}f(t)||_{2}^{2}.
\end{eqnarray*}
Here $\zeta$ is defined in \eqref{defzeta} and $\zeta^h$ is defined in \eqref{defzetah}.  These functional satisfy
\begin{eqnarray}
&&\frac{d}{dt}\zeta+\lambda\tilde{\mathcal{D}}_{3,3}\lesssim  (||\phi_t||_{\infty}+||\na_x \phi||_\infty)\zeta,\label{energyine1}\\
&&\frac{d}{dt}\zeta^h+\lambda\tilde{\mathcal{D}}_{3,3}\lesssim  (||\phi_t||_{\infty}+||\na_x \phi||_\infty)\zeta^h+||\na_x \FP f||^2_2,\label{energyine2}
\end{eqnarray}
where $\lambda$ is some positive number.
\end{proposition}

\begin{proof}
Let 
$$
\zeta_1=\int \sum_{\pm}\frac{f^2_{\pm}}{2}+\int |\na_x \phi|^2+\sum_{1\leq |\al|\leq 3}\left(\int \sum_{\pm}e^{\pm2\phi}\frac{\partial^\al f^2_{\pm}}{2}+\int |\na_x \partial^\al \phi|^2\right).$$
By \eqref{coercLANDAU}, Lemma \ref{xenergy0} and Lemma \ref{xenergy1}, since $\tilde{\mathcal{E}}_{3,3}(t)$ is sufficiently small, we have
\begin{eqnarray}
&&\frac{d}{dt}\zeta_1+\lambda^{(1)}_1\sum_{0\leq |\al|\leq 3}||\partial^\al \{\FI-\FP\}f||_\si^2\notag\\
&\lesssim&  ||\phi_t||_\infty \zeta_1+\sqrt{\tilde{\mathcal{E}}_{3;3}}(||E||^2_2+\sum_{1\leq |\al|\leq 3}||\partial^\al \FP f||_\si^2),\label{zeta1}
\end{eqnarray}
for some positive $\lambda^{(1)}_1$.
Let 
\begin{eqnarray*}
\zeta_2&=&\zeta_1+\sum_{1\leq |\al|\leq 3}\eta^{(2)}_{\al}\int \sum_{\pm }\frac{e^{\pm 2\phi }w^2(\partial ^{\alpha
}f_{\pm })^{2}}{2}\\
&&+\sum_{|\al|+|\be|\leq 3, |\al|\leq 2}\eta^{(2)}_{\al,\be}\int \sum_{\pm }\frac{e^{\pm 2\phi }w^2(\partial_{\beta} ^{\alpha
}\{\FI-\FP\}f_{\pm })^{2}}{2}
\end{eqnarray*}
with $\eta^{(2)}_{\al},\eta^{(2)}_{\al,\be}$ suitably small. By Lemma \ref{xenergyK}, Lemma \ref{I-PenergyK} 
and \eqref{zeta1}, we have
\begin{multline}
\frac{d}{dt}\zeta_2+\lambda^{(2)}_1\sum_{0\leq|\al|\leq 3}||\partial^\al \{\FI-\FP\}f  ||^2_\si
\\
+\lambda^{(2)}_2(\eta)\sum_{1\leq |\al|\leq 3}||\partial^\al f||_{\si,w}^2+\lambda^{(2)}_3(\eta)\sum_{1\leq |\al|+|\be|\leq 3,|\al|\leq2}||\partial^\al_\be \{\FI-\FP\}f||^2_{\si,w}
\\
\lesssim  (||\phi_t||_{\infty}+||\na_x \phi||_\infty)\zeta_2
+
\sqrt{\tilde{\mathcal{E}}_{3;3}}\left(||E||^2_2 + \sum_{1\leq|\al|\leq3}||\partial^\al\FP f||^2_2\right)\\
+
\eta\left(||\na_x \phi||^2_2+\sum_{1\leq|\al|\leq3}||\partial^\al\FP f||^2_2\right),
\label{zeta2}
\end{multline}
 where $\lambda^{(2)}_i,1\leq i\leq 3$ are positive numbers and $\eta>0$ can be arbitrarily small. Note that $\lambda^{(2)}_2$, $\lambda^{(2)}_3$ depend on $\eta$. We also make the observation that $\zeta_2\simnew \tilde{\mathcal{E}}_{3;3}$.

Now let
\begin{multline} \label{defzeta}
\zeta = \zeta_2+\eta_0^{(3)} \int -\langle [v,-v]\mu^{1/2},\{\FI-\FP\}f\rangle\cdot E dx
\\
+\sum_{1\leq l\leq 3}\eta_l^{(3)}\int |k|^{2(l-1)} ((1+|k|^2)\rmre \CE_{\rm int}^{(1)}(t,k)+\rmre (\hat{\Jcont}\mid ik \widehat{({a}_+-{a}_-)}) dk,  
\end{multline}
for suitable small positive numbers $\eta_l^{(3)},0\leq l\leq 3$.
Then $\zeta$ satisfies \eqref{energyine1} by \eqref{zeta2}, Lemma \ref{dissE}, Lemma \ref{naPlemma}, Lemma \ref{Nlemma} and the assumption $\mathcal{\tilde{E}}%
_{3;3}(f)\leq M$ for $M$ sufficiently small.

Now for \eqref{energyine2}, we let
\begin{multline}
\zeta^h_1=\int \sum_{\pm}\frac{ \left(\{\FI-\FP\}f_{\pm} \right)^2 }{2}+\int |\na_x \phi|^2
\\
+\sum_{1\leq |\al|\leq 3}C_\al \left(\int \sum_{\pm}e^{\pm2\phi}\frac{\left(\partial^\al f_{\pm}\right)^2}{2}+\int |\na_x \partial^\al \phi|^2\right),
\end{multline}
for $C_\al>0, 1\leq |\al|\leq 3$ suitably large. Then by \eqref{coercLANDAU}, Lemma \ref{xenergy0I-P} and Lemma \ref{xenergy1}, we have
\begin{eqnarray}
&&\frac{d}{dt}\zeta^h_1+\lambda^{(1)}_1\sum_{0\leq |\al|\leq 3}||\partial^\al \{\FI-\FP\}f||_\si^2\notag\\
&\lesssim&  ||\phi_t||_\infty \zeta^h_1+\sqrt{\tilde{\mathcal{E}}_{3;3}}(||E||^2_2+\sum_{1\leq |\al|\leq 3}||\partial^\al \FP f||_\si^2)+C||\na_x\FP f||^2_2,\label{zetah1}
\end{eqnarray}
for some positive ${\lambda^h}^{(1)}_1$. We then proceed exactly as in the previous case.  We define 
\begin{multline}
\zeta^h=\zeta^h_1
+\sum_{1\leq |\al|\leq 3}\eta^{(2)}_{\al}\int \sum_{\pm }\frac{e^{\pm 2\phi }w^2(\partial ^{\alpha
}f_{\pm })^{2}}{2}
\\
+\sum_{|\al|+|\be|\leq 3, |\al|\leq 2}\eta^{(2)}_{\al,\be}\int \sum_{\pm }\frac{e^{\pm 2\phi }w^2(\partial_{\beta} ^{\alpha
}\{\FI-\FP\}f_{\pm })^{2}}{2}
\\
+\eta_0^{(3)} \int -\langle [v,-v]\mu^{1/2},\{\FI-\FP\}f\rangle\cdot E dx
\\
+\sum_{1\leq j\leq 3}\eta_j^{(3)}\int (1+|k|^2)|k|^{2(j-1)}\CE_{\rm int}^{(1)}(t,k)dk. \label{defzetah}
\end{multline}
As in the previous construction, we again have \eqref{energyine2}. This completes the proof.
\end{proof}

\begin{proposition}
 Let $f_{0}\in \testF$ and assume $f$ is the
solution constructed in Theorem \ref{localsolution} with $\mathcal{\tilde{E}}%
_{3;l}(f)\leq M$ for some $l\geq 3$. Let $ w(\al,\be)=\langle v\rangle^{2(l-|\al|-|\be|)}$. Then there exist continuous energy functionals 
$\zeta_{3;l}(t)\simnew\tilde{\mathcal{E}}_{3,l}(t)$
 and 
\begin{eqnarray*}
\zeta^h_{3;l}(t)&\simnew&\sum_{|\alpha |+|\beta |\leq 3}\sum_{\pm
}||\partial _{\beta }^{\alpha }\{\FI_{\pm}-\FP_{\pm}\}f(t)||_{2,w(\alpha ,\beta )}^{2}\\
&&+||E(t)||_2^2+\sum_{1\leq |\alpha |\leq 3}\sum_{\pm
}||\partial^{\alpha }\FP_{\pm}f(t)||_{2}^{2}
\end{eqnarray*}
  with
\begin{eqnarray}
&&\frac{d}{dt}\zeta_{3;l}+\lambda\tilde{\mathcal{D}}_{3,l}\lesssim  (||\phi_t||_{\infty}+||\na_x \phi||_\infty)\zeta_{3;l},\label{energyine1l}\\
&&\frac{d}{dt}\zeta^h_{3;l}+\lambda\tilde{\mathcal{D}}_{3,l}\lesssim  (||\phi_t||_{\infty}+||\na_x \phi||_\infty)\zeta_{3;l}^h+||\na_x \FP f||^2_2,\label{energyine2l}
\end{eqnarray}
where $\lambda$ is some positive number.
\end{proposition}
\begin{proof}
With $ w(\al,\be)=\langle v\rangle^{2(l-|\al|-|\be|)}$, following the proof of \eqref{xenergyKbound} we obtain
\begin{eqnarray}
&&\frac{d}{dt} \int \sum_{\pm }\frac{e^{\pm 2\phi }w^2(\partial ^{\alpha
}f_{\pm })^{2}}{2} +||\partial^\al f||^2_{\si,w}\notag \\
&\lesssim &\sum_{\pm }\int \{||\phi _{t}||_{\infty}+||\na_x \phi ||_{\infty}\}w^2(\partial ^{\alpha }f_{\pm })^{2}+%
\sqrt{\mathcal{\tilde{E}}_{3,l}(f)}\mathcal{\tilde{\tilde{D}}}_{3,l}(f)\notag\\
&&+\eta||\partial^\al f||_{\si,w}^2+C_\eta||\partial^\al f||_\si^2+\eta ||\na_x \phi||^2_2,  \label{xenergyKboundl}
\end{eqnarray}
for $1\leq |\al|\leq 3$.
Also \eqref{I-PK} becomes
\begin{eqnarray}
&&\frac{d}{dt} \int \sum_{\pm }\frac{e^{\pm 2\phi }w^2(\partial_{\beta} ^{\alpha
}\{\FI-\FP\}f_{\pm })^{2}}{2}+||\partial_{\beta}^\al\{\FI-\FP\}f_{\pm }||_{\si, w(\al, \beta)} \notag \\
&\lesssim &\sum_{\pm }\int \{||\phi _{t}||_{\infty}+||\na_x \phi ||_{\infty}\}w^2(\partial_\beta ^{\alpha }\{\FI-\FP\}f_{\pm })^{2}\notag\\
&&+\mathcal{\tilde{E}}_{3;l}(f)\mathcal{\tilde{D}}_{3;l}(f)+\eta\mathcal{\tilde{D}}_{3,l}(f)+\mathcal{\tilde{D}}_{|\al|+|\beta|-1,l}(f)\notag\\
&&+\eta||\partial^\al_\be\{\FI-\FP\}f_{\pm}||^2_{\si,w}+C_\eta||\na_x^{|\al|+1}f_\pm||^2_\si,  \label{I-PKl}
\end{eqnarray}
for $|\al|+|\be|\leq 3$ and $|\al|\leq 2$.
Note that $C \| \partial^\alpha \nabla_x \phi \|_2^2\leq\eta||E||^2_2+C_\eta||\na^{|\al|+1}_x f||_\sigma^2$ by interpolation. 
Then the proof of Proposition \ref{E33} can be applied here. 
Under the assumption $\mathcal{\tilde{E}}%
_{3;l}(f)\leq M$, we have \eqref{energyine1l} and \eqref{energyine2l}.
\end{proof}

\section{Proof of our main theorem}\label{proofmain}
This section is devoted to the proof of Theorem \ref{maintheorem}. 
We start out with some basic estimates.  
The next lemma deals with the nonlinear term $N$.

\begin{lemma}\label{EandGa}
Let $N=[N_+, N_-]$ be defined as in \eqref{N}. Then we have
\begin{equation}
||N||_{\CH^1}+||N||_{Z_1}\lesssim \tilde{\mathcal{E}}_{3;3}(f).  \label{Nbootstrap}
\end{equation}
\end{lemma}

\begin{proof}
We first deal with the terms $E\cdot\na_v f$ and $E\cdot vf$. We have
\begin{eqnarray}
&&||E\cdot\na_v f||_{\CH^1}+||E\cdot vf||_{\CH^1}\notag\\
&\lesssim& (||E||_\infty +||\na_x E||_\infty)(||\na_v f||_{2}+||\na_x\na_v f||_{2})\notag\\
&&+(||E||_\infty +||\na_x E||_\infty)(||v f||_{2}+||v \na_x f||_{2})\notag\\
&\lesssim& \tilde{\mathcal{E}}_{3;3}(f),\label{Ebootstrap1}
\end{eqnarray}
and 
\begin{eqnarray}
&&||E\cdot\na_v f||_{Z_1}+||E\cdot vf||_{Z_1}\notag\\
&\lesssim& (||E||_2 +||\na_x E||_2)(||\na_v f||_{2}+||\na_x\na_v f||_{2})\notag\\
&&+(||E||_2 +||\na_x E||_2)(||v f||_{2}+||v \na_x f||_{2})\notag\\
&\lesssim& \tilde{\mathcal{E}}_{3;3}(f).\label{Ebootstrap2}
\end{eqnarray}
Now we turn to the term $\Ga(f,f)$. Now, using \eqref{gamma}, it suffices to estimate
\begin{multline*}
\tilde{\Gamma} [g_{1} , g_2]
\eqdef 
\partial _i \left( \{\Phi ^{ij}*[\mu^{1/2} g_1]\}\partial _jg_2 \right)
-
\{\Phi ^{ij}*[v_i\mu^{1/2} g_1]\} \partial _j g_2
\\
-
\partial _i \left(\{\Phi ^{ij}*[\mu^{1/2}\partial _j g_1]\} g_2 \right)
+
\{\Phi ^{ij}*[v_i\mu^{1/2}\partial _j g_1]\}  g_2
\end{multline*}
This expansion can be found in Lemma 6 of \cite{G1}.
Above and below we implicitly sum over indices $i$ and $j$ when they are repeated. Since $\Phi^{ij}\lesssim \frac{1}{|v|}$, we have
$$\Phi^{ij}*[\mu^{\delta}g]\lesssim ((\Phi^{ij})^2*\mu^{\delta})^{1/2}|\mu^{\delta/2}g|_2\lesssim |\mu^{\delta/2}g|_{L^2_v}.$$
Thus we have
\begin{equation*}
|\tilde{\Gamma}(f,f) |_{L^2_v}
\lesssim  
|f|_{L^2_v}|\na_v f|_{L^2_v}+|f|_{L^2_v}|\na_v^2 f|_{L^2_v}+|\na_v f|_{L^2_v}|\na_v f|_{L^2_v}.
\end{equation*}
Hence
\begin{eqnarray}
&&||\tilde{\Gamma}(f,f)||_2\notag\\
&\lesssim & ||f||_{L^4_xL^2_v}||\na_v f||_{L^4_xL^2_v}+||f||_{L^4_xL^2_v}||\na_v^2 f||_{L^4_xL^2_v}+||\na_v f||_{L^4_xL^2_v}||\na_v f||_{L^4_xL^2_v}\notag\\
&\lesssim & \tilde{\mathcal{E}}_{3;3}(f),\label{Gabootstrap1}
\end{eqnarray}
where we have used \eqref{4Sob} in the last line above.

Similarly, we have
\begin{eqnarray*}
&&|\na_x \tilde{\Gamma} (f,f)|_{L^2_v}
\\
&\lesssim & |\na_x f|_{L^2_v}|\na_v f|_{L^2_v}+|\na_x f|_{L^2_v}|\na_v^2 f|_{L^2_v}+|\na_x \na_v f|_{L^2_v}|\na_v f|_{L^2_v}\\
&&+|f|_{L^2_v}|\na_x\na_v f|_{L^2_v}+|f|_{L^2_v}|\na_x\na_v^2 f|_{L^2_v}+|\na_v f|_{L^2_v}|\na_x\na_v f|_{L^2_v},
\end{eqnarray*}
and
\begin{eqnarray}
&&||\na_x \tilde{\Gamma} (f,f)||_2\notag\\
&\lesssim & ||\na_x f||_{L^4_xL^2_v}||\na_v f||_{L^4_xL^2_v}+||\na_x f||_{L^4_xL^2_v}||\na_v^2 f||_{L^4_xL^2_v}\notag\\
&&+||\na_x\na_v f||_{L^4_xL^2_v}||\na_v f||_{L^4_xL^2_v}
+||f||_{L^4_xL^2_v}||\na_x\na_v f||_{L^4_xL^2_v}\notag\\
&&+||f||_{L^\infty_xL^2_v}||\na_x\na_v^2 f||_{L^2_xL^2_v}+||\na_v f||_{L^4_xL^2_v}||\na_x\na_v f||_{L^4_xL^2_v}\notag\\
&\lesssim & \tilde{\mathcal{E}}_{3;3}(f),\label{Gabootstrap2}
\end{eqnarray}
where we have used \eqref{inftySob} and \eqref{4Sob} in the last line above.

Finally we estimate $|| \tilde{\Gamma} (f,f)||_{Z_1}$ as follows:  
\begin{multline}
||\tilde{\Ga}(f,f)||_{Z_1}
\lesssim  || f||_{L^2_vL^2_x}||\na_v f||_{L^2_vL^2_x}+|| f||_{L^2_vL^2_x}||\na_v^2 f||_{L^2_vL^2_x}
\\
+||\na_v f||_{L^2_vL^2_x}||\na_v f||_{L^2_vL^2_x}
\lesssim  
\tilde{\mathcal{E}}_{3;3}(f).
\label{Gabootstrap3}
\end{multline}
This completes the proof of Lemma \ref{EandGa}.
\end{proof}

The following two lemmas grant time decay of the energy, locally in time.

\begin{lemma}
\label{decay}Assume that for $0\leq t\leq T,$ $\sup_{0\leq t\leq T}\mathcal{\tilde{E}}_{3;3+\frac{3}{4p'}}(f)(t)\leq M$ sufficiently small for some fixed $p'\in (0,1/2)$, and 
\begin{equation}
\text{ }\int_{0}^{T}(||\phi _{t}(s)||_{\infty }+\| \nabla_x \phi \|_\infty)ds\leq 1.
\label{integrability}
\end{equation}%
 We have for $0\leq t\leq T$ that
\begin{eqnarray}
\mathcal{\tilde{E}}_{3;3}(f(t))&\lesssim &C_{p'}(1+t)^{-\frac{3}{2}}\left(\mathcal{\tilde{E}}_{3;3+\frac{3}{4p'}}(f_{0})+||f_0||_{Z_1}^2\right).
\label{polydecay1'} \end{eqnarray}
\end{lemma}
\begin{proof}
The proof of this lemma benefits from developments in \cite[Theorem 1.2]{sNonCutOp}. Recall the definition of $\zeta$ 
from \eqref{defzeta}.  Now for $1/2>p'>0$, we define $\zeta_{low}$ and $\zeta_{high}$ as follows:
\begin{eqnarray}
\zeta_{low}&=&\int_{\langle v\rangle <t^{p'}} \sum_{\pm}\frac{f^2_{\pm}}{2}+\int |\na_x \phi|^2+
\sum_{1\leq |\al|\leq 3}\int_{\langle v\rangle <t^{p'}} \sum_{\pm}e^{\pm2\phi}\frac{(\partial^{\alpha}f_{\pm })^{2}}{2}
\notag
\\
&&+\sum_{|\al|+|\be|\leq 3, |\al|\leq 2}\eta^{(2)}_{\al,\be}\int_{\langle v\rangle <t^{p'}} \sum_{\pm }\frac{e^{\pm 2\phi }w^2(\partial_{\beta} ^{\alpha
}\{\FI-\FP\}f_{\pm })^{2}}{2}\notag\\
&&+\sum_{1\leq |\al|\leq 3}\eta^{(2)}_{\al}\int_{\langle v\rangle <t^{p'}} \sum_{\pm }\frac{e^{\pm 2\phi }w^2(\partial ^{\alpha
}f_{\pm })^{2}}{2}, \label{defzetalow}
\end{eqnarray}
\begin{eqnarray}
\zeta_{high}&=&\int_{\langle v\rangle \geq t^{p'}} \sum_{\pm}\frac{f^2_{\pm}}{2}
+
\sum_{1\leq |\al|\leq 3}\int_{\langle v\rangle \geq t^{p'}} \sum_{\pm}e^{\pm2\phi}\frac{(\partial ^{\alpha}f_{\pm })^{2}}{2}\notag\\
&&+\sum_{|\al|+|\be|\leq 3, |\al|\leq 2}\eta^{(2)}_{\al,\be}\int_{\langle v\rangle \geq t^{p'}} \sum_{\pm }\frac{e^{\pm 2\phi }w^2(\partial_{\beta} ^{\alpha
}\{\FI-\FP\}f_{\pm })^{2}}{2}\notag\\
&&+\sum_{1\leq |\al|\leq 3}\eta^{(2)}_{\al}\int_{\langle v\rangle \geq t^{p'}} \sum_{\pm }\frac{e^{\pm 2\phi }w^2(\partial ^{\alpha
}f_{\pm })^{2}}{2}. \label{defzetahigh}
\end{eqnarray}
Then we have $\zeta\simnew \zeta_{low}+\zeta_{high}$.  From \eqref{defzetalow}, we have
\begin{equation}
\frac{\zeta_{low}(t)}{ t^{p^\prime} }\lesssim \frac{\|\FP \solU\|_{\CL^2}^2}{t^{p^\prime}}+  \tilde{\mathcal{D}}_{3;3}(f)(t),
\label{soft.up.d}
\end{equation}
for $t>1$. Conbining this and \eqref{energyine1} gives
\begin{eqnarray}
\frac{d}{dt}\zeta(t)
+
\lambda p t^{p-1}\zeta(t)
&\lesssim&
(||\phi_t||_\infty+||\na_x \phi||_\infty)\zeta+
t^{p-1}\|\FP \solU\|_{2}^2\notag\\
&&
+
 t^{p-1} \zeta_{high}(t)+\lambda p t^{p-1}||\na_x \phi 1_{\{t<1\}}||^2_2,\label{energyine3}
\end{eqnarray}
where we defined $\lambda= \frac{\la^\prime}{ p } $ with $p= -p^\prime+1>0$. Use the factor 
$$e^{-\lambda t^p+\int_0^t (||\phi_t||_\infty+||\na_x\phi||_\infty)ds}$$
to obtain
\begin{eqnarray}
\zeta(t)&\lesssim& e^{-\lambda t^p+\int_0^t (||\phi_t||_\infty+||\na_x\phi||_\infty)ds}\zeta(0)\notag\\
&&+\int_0^t e^{-\lambda t^p+\lambda s^p+\int_s^t (||\phi_t||_\infty+||\na_x\phi||_\infty)ds_1}\notag\\
&&\cdot\left(s^{p-1}\|\FP \solU\|_{2}^2+s^{p-1} \zeta_{high}(s)+\lambda p s^{p-1}||\na_x \phi 1_{\{s<1\}}(s)||^2_2\right)ds.\label{energyine4}
\end{eqnarray}
Because of \eqref{integrability}, \eqref{energyine4} becomes
\begin{eqnarray}
\zeta(t)&\lesssim& e^{-\lambda t^p}\zeta(0)\notag\\
&&+\int_0^t e^{-\lambda t^p+\lambda s^p}\left(s^{p-1}\|\FP \solU\|_{2}^2+s^{p-1} \zeta_{high}(s)+\lambda p s^{p-1}||\na_x \phi 1_{\{s<1\}}(s)||^2_2\right)ds.\label{energyine4'}
\end{eqnarray}
From \eqref{defzetahigh}, we have
\begin{equation}
\label{decay.high}
\zeta_{high}(t)
\lesssim
(1+t)^{-\frac{3}{2}} \tilde{\mathcal{E}}_{3,3+\frac{3}{4p'}}(f_0),
\end{equation}
where we have used \eqref{integrability} and \eqref{energyine1l}.

Next we deal with the term $||\FP f||_2^2$ in \eqref{energyine4'}.
We have
\begin{equation}  \label{Boltz.rep}
   \solU(t)=I_0(t)+I_1(t),
\end{equation}
with 
$$
  I_0(t) =  \semiG(t)\solU_0, \quad
  I_1(t) =\int_0^t\semiG(t-s)N(f)(s)ds.
$$
Using \eqref{thm.ls.1.soft} with $m=0$, $r=1$ and $\wN=-\wB\ge 0$ with $b$ sufficiently large, we obtain
\begin{equation}  \notag
    \|\langle v\rangle^{-\wB}I_0(t)\|_{2}
    \lesssim
    (1+t)^{-\frac{\dim}{4}} \|\langle v\rangle^{-\wB+\wE} \solU_0\|_{{L^2_{x,v}}\cap Z_1}\lesssim
    (1+t)^{-\frac{\dim}{4}} \| \solU_0\|_{{L^2_{x,v}}\cap Z_1},
\end{equation}
and
\begin{multline*}
    \|\langle v\rangle^{-\wB} I_1(t)\|_{2}\leq \int_0^t\|\langle v\rangle^{-\wB}\semiG(t-s)N(f)(s)\|_{2}ds
    \\
\lesssim
    \int_0^t(1+t-s)^{-\frac{3}{4}}\|\langle v\rangle^{-\wB+\wE}N(\solU)(s)\|_{{L^2_{x,v}}\cap Z_1}ds\\
  \lesssim
    \int_0^t(1+t-s)^{-\frac{3}{4}}\|N(\solU)(s)\|_{{L^2_{x,v}}\cap Z_1}ds  .
\end{multline*}
  Define
\begin{equation}\label{def.einfty}
    \zeta^\infty(t) \eqdef \sup_{0\leq s\leq t}(1+s)^{\frac{\dim}{2}}\zeta(s).
\end{equation}
For $I_1(t)$, from \eqref{Nbootstrap} and the definition \eqref{def.einfty} of $\zeta^\infty(t)$, it holds that
\begin{multline*}
    \|\langle v\rangle^{-\wB} I_1(t)\|_{2}\lesssim
    \int_0^t(1+t-s)^{-\frac{3}{4}}\zeta(s)ds
    \\
    \lesssim \zeta^\infty(t)\int_0^t(1+t-s)^{-\frac{3}{4}}(1+s)^{-\frac{3}{2}}ds\\
    \lesssim \zeta^\infty(t)(1+t)^{-\frac{3}{4}}.
\end{multline*}
Here we used the decay estimates for the time integrals as in  \cite[Proposition 4.5]{strainSOFT}.

Collecting the estimates on $I_1(t)$ and $I_2(t)$ above, with \eqref{form.p}, implies 
\begin{multline}
\|\FP \solU(t)\|_{2}^2\lesssim\|\langle v\rangle^{-\wB}\solU(t)\|_{2}^2\lesssim
\|\langle v\rangle^{-\wB} I_0(t)\|_{2}^2
+
\|\langle v\rangle^{-\wB} I_1(t)\|_{2}^2
\\
\lesssim
(1+t)^{-\frac{3}{2}} \|\solU_0\|_{L_{x,v}^2\cap Z_1}^2
+
 (1+t)^{-\frac{3}{2}}[\zeta^\infty(t)]^2.\label{thm.ns.p2}
\end{multline}
For the term $\lambda p ||\na_x \phi 1_{\{t<1\}}||^2_2$, we have
\begin{equation}
\lambda p ||\na_x \phi 1_{\{t<1\}}||^2_2\lesssim
 \lambda p(1+t)^{-\frac{3}{2}}\zeta^\infty(t).\label{thm.ns.p2'}
\end{equation}
Now we plug \eqref{thm.ns.p2}, \eqref{thm.ns.p2'} and
\eqref{decay.high} 
into 
\eqref{energyine4'} to obtain \eqref{polydecay1'}  since $M$ is sufficiently small and $\lambda$ can be chosen to be small. Note that we used Lemma \ref{BasicDecayCor}.   \end{proof}

\begin{lemma}
\label{decay'}Assume that for $0\leq t\leq T,$ $||f_0||_{Z_1}+\sup_{0\leq t\leq T}\mathcal{\tilde{E}}_{3;\frac{5}{2}+\frac{5}{4p'}}(f)(t)\leq M$ sufficiently small for some fixed $p'\in (0,1/2)$, and 
\begin{equation}
\text{ }\int_{0}^{T}(||\phi _{t}(s)||_{\infty }+\| \nabla_x \phi \|_\infty)ds\leq 1.
\label{integrability'}
\end{equation}%
 We have 
\begin{eqnarray}
&  &||\partial _{t}\phi (t)||_{\infty }+||\nabla _{x}\phi (t)||_{\infty
}\notag\\
 &\lesssim &C_{p'}(1+t)^{-\frac{5}{4}+\frac{p'}{2}}\left(\sqrt{\mathcal{\tilde{E}}_{3;\frac{5}{2}+\frac{5}{4p'}}(f_{0})}+||f_0||_{Z_1}\right).
\label{polydecay2'} 
\end{eqnarray}
\end{lemma}

\begin{proof}
The proof of this lemma is similar to that of Lemma \ref{decay}. Recall the definition of $\zeta^h$ from \eqref{defzetah}.  Now for $1/2>p'>0$, we define $\zeta^h_{low}$ and $\zeta^h_{high}$ as follows:
\begin{eqnarray}
\zeta^h_{low}&=&\int_{\langle v\rangle <t^{p'}} \sum_{\pm}\frac{\{\FI-\FP\}f^2_{\pm}}{2}+\int |\na_x \phi|^2+\sum_{1\leq |\al|\leq 3}C_\al\int_{\langle v\rangle <t^{p'}} \sum_{\pm}e^{\pm2\phi}\frac{\left(\partial^\al f_{\pm}\right)^2}{2}
\notag\\
&&+\sum_{|\al|+|\be|\leq 3, |\al|\leq 2}\eta^{(2)}_{\al,\be}\int_{\langle v\rangle <t^{p'}} \sum_{\pm }\frac{e^{\pm 2\phi }w^2(\partial_{\beta} ^{\alpha
}\{\FI-\FP\}f_{\pm })^{2}}{2}\notag\\
&&+\sum_{1\leq |\al|\leq 3}\eta^{(2)}_{\al}\int_{\langle v\rangle <t^{p'}} \sum_{\pm }\frac{e^{\pm 2\phi }w^2(\partial ^{\alpha
}f_{\pm })^{2}}{2}, \label{defzetahlow}
\end{eqnarray}
\begin{eqnarray}
\zeta^h_{high}&=&\int_{\langle v\rangle \geq t^{p'}} \sum_{\pm}\frac{\{\FI-\FP\}f^2_{\pm}}{2}
+
\sum_{1\leq |\al|\leq 3}C_\al\int_{\langle v\rangle \geq t^{p'}} \sum_{\pm}e^{\pm2\phi}\frac{\left(\partial^\al f_{\pm}\right)^2}{2}\notag\\
&&
+
\sum_{|\al|+|\be|\leq 3, |\al|\leq 2}\eta^{(2)}_{\al,\be}\int_{\langle v\rangle \geq t^{p'}} 
\sum_{\pm }\frac{e^{\pm 2\phi }w^2(\partial_{\beta} ^{\alpha}\{\FI-\FP\}f_{\pm })^{2}}{2}\notag\\
&&+\sum_{1\leq |\al|\leq 3}\eta^{(2)}_{\al}\int_{\langle v\rangle \geq t^{p'}} \sum_{\pm }\frac{e^{\pm 2\phi }w^2(\partial ^{\alpha
}f_{\pm })^{2}}{2}. \label{defzetahhigh}
\end{eqnarray}
Then we have $\zeta^h\simnew \zeta^h_{low}+\zeta^h_{high}$.  From \eqref{defzetahlow}, we have
\begin{equation}
\frac{\zeta^h_{low}(t)}{ t^{p^\prime} }\lesssim  \tilde{\mathcal{D}}_{3;3}(f)(t),
\label{soft.up.d'}
\end{equation}
for $t>1$. 
Conbining this and \eqref{energyine2} gives
\begin{eqnarray}
\frac{d}{dt}\zeta^h(t)
+
\lambda p t^{p-1}\zeta^h(t)
&\lesssim&
(||\phi_t||_\infty+||\na_x \phi||_\infty)\zeta^h+
\|\na_x\FP \solU\|_{2}^2\notag\\
&&
+
 t^{p-1} \zeta^h_{high}(t)+\lambda p t^{p-1}||\na_x \phi 1_{\{t<1\}}||^2_2.\label{energyine3'}
\end{eqnarray}
Following the exact procedure used to obtain \eqref{energyine4'}, we achieve
\begin{eqnarray}
\zeta^h(t)&\lesssim& e^{-\lambda t^p}\zeta^h(0)\notag\\
&&+\int_0^t e^{-\lambda t^p+\lambda s^p}\left(\|\na_x\FP \solU\|_{2}^2+s^{p-1} \zeta^h_{high}(s)+\lambda p s^{p-1}||\na_x \phi 1_{\{s<1\}}(s)||^2_2\right)ds.\label{energyine5}
\end{eqnarray}
>From \eqref{defzetahhigh}, \eqref{integrability} and \eqref{energyine1l}, we have
\begin{equation}
\label{decay.high'}
\zeta^h_{high}(t)
\lesssim
(1+t)^{-\frac{5}{2}+p'} \tilde{\mathcal{E}}_{3,\frac{5}{2}+\frac{5}{4p'}}(f_0).
\end{equation}
Next we deal with the term $||\na_x\FP f||_2^2$ in \eqref{energyine5}.
It follows from Theorem \ref{thm.ls} and \eqref{form.p} for $\wB \ge 0$ that
\begin{multline}
\notag
\|\na_x\FP \solU(t)\|_{2}^2
\lesssim 
\|\langle v\rangle^{-\wB}\na_x\FP \solU(t)\|_{2}^2
\lesssim 
(1+t)^{-\frac{5}{2}} \|\langle v\rangle^{-\wB+\wE}\solU_0\|_{\CHd^{1}\cap Z_1}^2\\
 + \int_0^t (1+t-s)^{-\frac{5}{2}} \|\langle v\rangle^{-\wB+\wE} N(\solU)(s)\|^2_{\CHd^{1}\cap Z_1}ds.
\end{multline}
Here $\wE$ is defined as in Theorem \ref{thm.ls}.  We can take $b$ sufficiently large to make $-b+j<0$. Since $3+\frac{3}{4p'}<\frac{5}{2}+\frac{5}{4p'}$, we can apply Lemma \ref{decay} together with Lemma \ref{EandGa} and Proposition 4.5 in \cite{strainSOFT} to get
 \begin{eqnarray}
&&\|\na_x\FP \solU(t)\|_{2}^2\notag\\
&\lesssim &
(1+t)^{-\frac{5}{2}} \|\solU_0\|_{\CHd^{1}\cap Z_1}^2
 + \int_0^t (1+t-s)^{-\frac{5}{2}} \| N(\solU)(s)\|^2_{\CHd^{1}\cap Z_1}ds\notag\\
 &\lesssim &
(1+t)^{-\frac{5}{2}} \|\solU_0\|_{\CHd^{1}\cap Z_1}^2
 + \int_0^t (1+t-s)^{-\frac{5}{2}}\zeta(s)ds\notag\\
  &\lesssim& (1+t)^{-\frac{5}{2}}\left( \|\solU_0\|_{\CHd^{1}\cap Z_1}^2+ (\mathcal{\tilde{E}}_{3;3+\frac{3}{4p'}}(f_{0})+||f_0||_{Z_1}^2)^2\right)\notag\\
&\lesssim& (1+t)^{-\frac{5}{2}}\left( \|\solU_0\|_{\CHd^{1}\cap Z_1}^2+ (\mathcal{\tilde{E}}_{3;3+\frac{3}{4p'}}(f_{0})+||f_0||_{Z_1}^2)\right)\notag\\
&\lesssim& (1+t)^{-\frac{5}{2}}\left (\mathcal{\tilde{E}}_{3;\frac{5}{2}+\frac{5}{4p'}}(f_{0})+||f_0||_{Z_1}^2\right),\label{napf}
\end{eqnarray}
where we have used the smallness of $M$ in the last line above.
  Define
\begin{equation}\label{def.ehinfty}
    \zeta^{h,\infty}(t) \eqdef \sup_{0\leq s\leq t}(1+s)^{\frac{5}{2}-p'}\zeta^h(s).
\end{equation}
For the term $\lambda p ||\na_x \phi 1_{\{t<1\}}||^2_2$, we have
\begin{equation}
\lambda p ||\na_x \phi 1_{\{t<1\}}||^2_2\lesssim
 \lambda p(1+t)^{-\frac{5}{2}+p'}\zeta^{h,\infty}(t).\label{thm.ns.p2h'}
\end{equation}
Now we plug \eqref{napf}, \eqref{thm.ns.p2h'} and
\eqref{decay.high'} 
into 
\eqref{energyine5} and to obtain
\begin{equation}
\zeta^{h,\infty} \lesssim C_{p'}\left(\sqrt{\mathcal{\tilde{E}}_{3;\frac{5}{2}+\frac{5}{4p'}}(f_{0})}+||f_0||_{Z_1}\right)^2.
\label{polydecay2''} 
\end{equation}
Note that we have used Lemma \ref{BasicDecay} and Lemma \ref{BasicDecayCor} in our integral. 

Notice that $||\phi_t||_\infty+||\na_x\phi||_\infty\lesssim \sqrt{\zeta^h(t)}$. So \eqref{polydecay2'} follows from \eqref{polydecay2''}. 
\end{proof}

Here is a lemma which deals with the the mixed $v$ and $x$ derivative estimates.
\begin{lemma}
\label{mixenergy}Let $f_{0}\in \testF$ and assume $f$ is the
solution constructed in Theorem \ref{localsolution}. Assume for $M>0$
sufficiently small,  that 
$
\mathcal{E}_{3;\frac{5}{2}+\frac{5}{4p'}}(f)\leq M.
$

(1) For $l\geq \frac{5}{2}+\frac{5}{4p'}$, we have%
\begin{eqnarray}
&&\tilde{\mathcal{E}}_{3;l}(f)+\int_{0}^{t}\mathcal{D}_{3;l}(f)ds  \label{mix3}
\\
&\lesssim &\{\tilde{\mathcal{E}}_{3;l}(f_{0})+\int_{0}^{t}\sum_{|\alpha |\leq
3}||\partial ^{\alpha }f||_{\sigma }^{2} +\int_0^t ||\na_x \phi||^2_2 \notag\\
&& +\int_{0}^{t}[\mathcal{D}%
_{3;l}(f)+||\nabla \phi ||_{\infty }+||\partial _{t}\phi ||_{\infty }]%
\tilde{\mathcal{E}}_{3;l}(f)ds\}.  \notag
\end{eqnarray}%
\ (2) If $m\geq 4$ and $l\geq \frac{5}{2}+\frac{5}{4p'}$, we have 
\begin{eqnarray}
&&\tilde{\mathcal{E}}_{m;l}(f)+\int_{0}^{t}\mathcal{D}_{m;l}(f)ds  \label{mix4}
\\
&\leq &C_{l,m}\{\tilde{\mathcal{E}}_{m;l}(f_{0})+\int_{0}^{t}[\mathcal{D}%
_{m-1;l}(f)+||\nabla _{x}\phi ||_{\infty }+||\partial _{t}\phi ||_{\infty
}]\tilde{\mathcal{E}}_{m;l}(f)\notag\\
&&+\int_{0}^{t}\sum_{|\alpha |=m}||\partial ^{\alpha
}f||_{\sigma }^{2}+\int_0^t ||\na_x\phi||^2_2  \notag \\
&&+\int_{0}^{t}[\tilde{\mathcal{E}}_{m-1;l}(f)+1]\mathcal{D}_{m-1;l}(f)\}. 
\notag
\end{eqnarray}
\end{lemma}

\begin{proof}
We use mostly the estimates given in the proof of Lemma \ref{mixnenergy}. Actually \eqref{emn} implies \eqref{mix4} already. For \eqref{mix3}, the only difference is in the estimate of \eqref{local9}. By \eqref{langle}, we have
\begin{eqnarray*}
\eqref{local9}\lesssim 
\sum_{\alpha_1 \le \alpha}
\sum_{\bar{\beta}\le\beta_1 \le \beta}
|\langle v\rangle^{-\wN}\partial _{\bar{\beta}}^{\alpha _1}f|_{2}
|\partial _{\beta -\beta _1}^{\alpha -\alpha _1}f|_{\sigma ,w(\al,\be)}
|\partial _\beta^\alpha f|_{\sigma ,w(\al,\be)}. 
\end{eqnarray*}
When $|\al_1|+|\be_1|\leq 1$, we place $|\langle v\rangle^{-\wN}\partial _{\bar{\beta}}^{\alpha _1}f|_{2}$ in $L^\infty_x$. We place both $|\partial _{\beta -\beta _1}^{\alpha -\alpha _1}f|_{\sigma ,w(\al,\be)}$ and $|\partial _\beta^\alpha f|_{\sigma ,w(\al,\be)}$ in $L^2_x$.

When $|\al_1|+|\be_1|\geq 2$, we place  $|\partial _{\beta -\beta _1}^{\alpha -\alpha _1}f|_{\sigma ,w(\al,\be)}$ in $L^\infty_x$. We place  $|\langle v\rangle^{-\wN}\partial _{\bar{\beta}}^{\alpha _1}f|_{2}$ and $|\partial _\beta^\alpha f|_{\sigma ,w(\al,\be)}$ in $L^2_x$.

In both cases, we have
\begin{equation}
\eqref{local9}\lesssim \sqrt{\tilde{\mathcal{E}}_{3;l}(f)}\mathcal{D}_{3;l}(f)\lesssim \eta \mathcal{D}_{3;l}(f)+C_{\eta}\tilde{\mathcal{E}}_{3;l}(f)\mathcal{D}_{3;l}(f).\label{newlocal9}
\end{equation}
Combining \eqref{newlocal9} and estimates given in the proof of Lemma \ref{mixnenergy} gives \eqref{mix3}.  \end{proof}

Now we are ready to prove Theorem \ref{maintheorem}.
\begin{proof}[Proof of Theorem \ref{maintheorem}] The strategy we use is from the proof of Theorem 1 in \cite{G0}.

\textit{Step 1. Global Small }$\mathcal{E}_{3;\frac{5}{2}+\frac{5}{4p'}}$\textit{\ Solutions. }

We denote 
\begin{eqnarray}
T_{\ast }&=&\sup_{t\geq 0}\left\{ \tilde{\mathcal{E}}_{3;\frac{5}{2}+\frac{5}{4p'}}(f)(t)+\int_{0}^{t}%
\mathcal{D}_{3;\frac{5}{2}+\frac{5}{4p'}}(f)(s)ds\leq M\right.\notag\\
&&\left.\text{ and}\int_{0}^{t}(||\nabla _{x}\phi
(s)||_{\infty }+||\phi_t||_\infty)ds\leq \sqrt{M}\right\} .  \label{t*}
\end{eqnarray}%
Clearly $T_{\ast }>0$ if $\mathcal{E}_{3;\frac{5}{2}+\frac{5}{4p'}}(f_{0})$ is sufficiently small
from Theorem \ref{localsolution}. Our goal is to show $T_{\ast }=\infty $ if
we further choose $\mathcal{E}_{3;\frac{5}{2}+\frac{5}{4p'}}(f_{0})$ small.

For $l\geq \frac{5}{2}+\frac{5}{4p'}$, we take a large constant $C\times \eqref{energyine1}+\eqref{mix3}$ to get
\begin{eqnarray}
&&\tilde{\mathcal{E}}_{3;l}(f)+\int_{0}^{t}\mathcal{D}_{3;l}(f)ds  \notag\\
&\lesssim &\{\tilde{\mathcal{E}}_{3;l}(f_{0})+\int_{0}^{t}||\FP f||_2^2
 +\int_{0}^{t}[\mathcal{D}%
_{3;l}(f)+||\nabla \phi ||_{\infty }+||\partial _{t}\phi ||_{\infty }]%
\tilde{\mathcal{E}}_{3;l}(f)ds\}.  \label{mix3'}
\end{eqnarray}
By Lemma \ref{decay}, we have
\begin{eqnarray}
\int_{0}^{t}||\FP f||_2^2&\lesssim& \int_0^t  (1+t)^{-\frac{3}{2}}\left(\mathcal{\tilde{E}}_{3;3+\frac{3}{4p'}}(f_{0})+||f_0||_{Z_1}^2\right)dt\notag\\
&\lesssim& \mathcal{\tilde{E}}_{3;3+\frac{3}{4p'}}(f_{0})+||f_0||_{Z_1}^2,
\end{eqnarray}
for $t\leq T_*$. Since $\int_{0}^{T_{\ast }}[||\partial _{t}\phi (s)||_{\infty }+||\nabla
_{x}\phi (s)||_{\infty }+\mathcal{D}_{3;\frac{5}{2}+\frac{5}{4p'}}(f)]ds\leq 1,$ The Gronwall's inequality such as Lemma 4 in \cite{G0} implies 
\begin{equation}
\tilde{\mathcal{E}}_{3;\frac{5}{2}+\frac{5}{4p'}}(f)+\int_{0}^{t}\mathcal{D}_{3;\frac{5}{2}+\frac{5}{4p'}}(f)ds\lesssim 
(\tilde{\mathcal{E}}_{3;\frac{5}{2}+\frac{5}{4p'}}(f_{0})+||f_0||_{Z_1}^2).  \label{mix3energy}
\end{equation}
>From Lemma \ref{decay'}, we obtain%
\begin{eqnarray*}
&&\int_{0}^{t}\{||\nabla _{x}\phi (s)||_{\infty }+||\partial _{t}\phi
(s)||_{\infty }\}ds\\
&\lesssim& \int_{0}^{t}C_{p'}(1+t)^{-\frac{5}{4}+\frac{p'}{2}}\left(\sqrt{\mathcal{\tilde{E}}_{3;\frac{5}{2}+\frac{5}{4p'}}(f_{0})}+||f_0||_{Z_1}\right)ds\\
&\lesssim& \sqrt{\mathcal{\tilde{E}}_{3;\frac{5}{2}+\frac{5}{4p'}}(f_{0})}+||f_0||_{Z_1}.
\end{eqnarray*}%
Upon choosing the initial condition $\sqrt{\mathcal{\tilde{E}}_{3;\frac{5}{2}+\frac{5}{4p'}}(f_{0})}+||f_0||_{Z_1}$ further
small, we deduce that for $0\leq t\leq T_{\ast },$ 
\begin{equation*}
\tilde{\mathcal{E}}_{3;\frac{5}{2}+\frac{5}{4p'}}(f(t))+\int_{0}^{t}\mathcal{D}_{3;\frac{5}{2}+\frac{5}{4p'}}(f)ds\leq \frac{M}{2%
}<M
\end{equation*}
and
\begin{equation*} 
\int_{0}^{T_{\ast }}\{||\nabla _{x}\phi (s)||_{\infty
}+||\partial _{t}\phi (s)||_{\infty }\}ds\leq \frac{M}{2}<M.
\end{equation*}%
This implies that $T_{\ast }=\infty $ and the solution is global. Then \eqref{polydecay1} and \eqref{polydecay2} follow from Lemma \ref{decay} and Lemma \ref{decay'}.

\textit{Step 2. Higher Moments and Higher Regularity. }
We shall prove \eqref{highnorm} by an induction of the total derivatives $|\al|+|\be|=m$. By \eqref{mix3'}, we know \eqref{highnorm} is valid when $m=3,l\geq \frac{5}{2}+\frac{5}{4p'}$.

Assume \eqref{highnorm} holds for $ m-1$. 
Recall \eqref{pure3}. For $|\alpha |=m,$ we have for any $\eta >0,$%
\begin{eqnarray*}
&&\frac{d}{dt}\left[ \int \sum_{\pm }\frac{e^{\pm 2\phi }(\partial ^{\alpha
}f_{\pm })^{2}}{2}+\int |\nabla \partial ^{\alpha }\phi |^{2}\right] +\int
\langle L\partial ^{\alpha }f,\partial ^{\alpha }f\rangle  \notag \\
&\lesssim &\sum_{\pm }\int |\phi _{t}||\partial ^{\alpha }f_{\pm }|^{2}+%
\sqrt{\mathcal{\tilde{E}}%
_{2;2}(f)}\left(\int  |\partial ^{\alpha }f_{\pm }|_{\sigma }^{2}+||\na_x \phi||^2_2\right)+\eta \sum_{|\alpha
|=m}||\partial ^{\alpha ^{\prime }}f||_{\sigma }^{2}  \notag \\
&&+C_{m,\eta }[\mathcal{D}_{2;2}(f)\mathcal{\tilde{E}}_{m;l}(f)+\{1+\mathcal{\tilde{E}}%
_{m-1;l}(f)\}\mathcal{D}_{m-1;l}(f)]. 
\end{eqnarray*}
We note that from Lemma 7 of \cite{G1}, we have
$$ \sum_{|\al|\leq m-1} ||\partial^\al N_{\pm}||^2_2\lesssim C_m  \{1+\mathcal{\tilde{E}}%
_{m-1;l}(f)\}\mathcal{D}_{m-1;l}(f).    $$
>From \eqref{def.int1'}, we have
\begin{eqnarray*}
&&\int \CE_{\rm int}^{(1)}(t,k)(1+|k|^2)|k|^{2(m-1)}dk\\
 &\lesssim &\sum_{|\alpha |=m}||\partial ^{\alpha
}f(t)||^{2}+\sum_{|\alpha |=m-1}||\partial ^{\alpha }f(t)||^{2} \\
&\lesssim &\sum_{|\alpha |=m}||\partial ^{\alpha }f(t)||^{2}+\tilde{\mathcal{E}}%
_{m-1;l}(f(t)) \\
&\lesssim &\sum_{|\alpha |=m}||\partial ^{\alpha }f(t)||^{2}+P_{m-1,l}(%
\tilde{\mathcal{E}}_{m-1;l}(f_{0}))
\end{eqnarray*}%
by the induction hypothesis.

We take the time derivative of
 $$\kappa_{m,l}\int \CE_{\rm int}^{(1)}(t,k)(1+|k|^2)|k|^{2(m-1)}dk+\sum_{|\al|=m}\left[ \int \sum_{\pm }\frac{e^{\pm 2\phi }(\partial ^{\alpha
}f_{\pm })^{2}}{2}+\int |\nabla \partial ^{\alpha }\phi |^{2}\right]$$ 
for some small multiple $\kappa_{m,l}$ and apply Lemma \ref{naPlemma} together with the inductive hypothesis to get
\begin{eqnarray*}
&&\sum_{|\alpha |=m}||\partial ^{\alpha }f||_{2}^{2}+\sum_{|\alpha
|=m}\int_{0}^{t}||\partial ^{\alpha }f||_{\sigma }^{2} \\
&\lesssim &\tilde{\mathcal{E}}_{m;l}(f_{0})+P_{m-1,l}(\tilde{\mathcal{E}}_{m;l}(f_{0}))+(%
\sqrt{M}+\eta )\int_{0}^{t}\sum_{|\alpha |=m}||\partial ^{\alpha }f||_{\sigma }^{2}\\
&&+\int_{0}^{t}\mathcal{D}_{2;2}(f)\tilde{\mathcal{E}}_{m;l}(f)
+C_{l,m,\eta }\int_{0}^{t}\{1+\tilde{\mathcal{E}}_{m-1;l}(f)\}\mathcal{D}%
_{m-1;l}(f) \\
&\lesssim &\tilde{\mathcal{E}}_{m;l}(f_{0})+(\sqrt{M}+\eta )\sum_{|\alpha
|=m}||\partial ^{\alpha}f||_{\sigma }^{2}+\int_{0}^{t}\mathcal{D}%
_{2;2}(f)\tilde{\mathcal{E}}_{m;l}(f) \\
&&+C_{l,m,\eta }[1+P_{m-1,l}(\tilde{\mathcal{E}}_{m;l}(f_{0}))]P_{m-1,l}(\tilde{\mathcal{E}}_{m;l}(f_{0}))
\end{eqnarray*}%
Here $P_{m-1,l}$ is a continuous, increasing function from the induction
hypothesis. For $M,\eta $ sufficiently small 
\begin{multline}
\sum_{|\alpha |=m}\left\{ |\partial ^{\alpha }f|^{2}+\int |\nabla \partial
^{\alpha }\phi |^{2}\right\} +\sum_{|\alpha |=m}\int_{0}^{t}||\partial
^{\alpha }f||_{\sigma }^{2}  
 \\
\lesssim 
C_{l,m}[1+P_{m-1,l}(\tilde{\mathcal{E}}_{m;l}(f_{0}))][\tilde{\mathcal{E}}_{m;l}(f_{0})+P_{m-1,l}(\tilde{\mathcal{E}}_{m;l}(f_{0}))]   
+\int_{0}^{t}\mathcal{D}_{2;2}(f)\tilde{\mathcal{E}}_{m;l}(f),
\label{pure3energy}
\end{multline}
where we have used $\tilde{\mathcal{E}}_{m-1;l}(f_{0})\leq \tilde{\mathcal{E}}_{m;l}(f_{0}),P_{m-1,l}(\tilde{\mathcal{E}}_{m-1;l}(f_{0}))\leq P_{m-1,l}(%
\tilde{\mathcal{E}}_{m;l}(f_{0})).$
Multiplying a large constant $C\times (\ref{pure3energy})+(\ref{mix4})$ to
absorb $\int_{0}^{t}\sum_{|\alpha |=m}||\partial ^{\alpha }f||_{\sigma }^{2}$
in (\ref{mix4})$,$ we obtain:%
\begin{eqnarray*}
\tilde{\mathcal{E}}_{m;l}(f)+\int_{0}^{t}\mathcal{D}_{m;l}(f) &\lesssim
&C_{l,m}[1+P_{m-1,l}(\tilde{\mathcal{E}}_{m;l}(f_{0}))][\tilde{\mathcal{E}}_{m;l}(f_{0})+P_{m-1,l}(\tilde{\mathcal{E}}_{m;l}(f_{0}))] \\
&&+C_{l,m}\int_{0}^{t}\left(\mathcal{D}_{m-1;l}(f)+||\na_x\phi||_\infty+||\partial_t\phi||_\infty\right)\tilde{\mathcal{E}}_{m;l}(f).
\end{eqnarray*}%
We use  Lemma 4 in \cite{G0} with
$$\int_{0}^{t}\left(\mathcal{D}_{m-1;l}(f)+||\na_x\phi||_\infty+||\partial_t\phi||_\infty\right)ds\lesssim P_{m-1,l}(\mathcal{E}_{m-1;l,q}(f_{0}))$$
 to get%
\begin{eqnarray*}
&&\tilde{\mathcal{E}}_{m;l}(f)+\int_{0}^{t}\mathcal{D}_{m;l}(f) \\
&\lesssim &C_{l,m}e^{C_{l,m}P_{m-1}(\tilde{\mathcal{E}}_{m;l}(f_{0}))}[1+P_{m-1,l}(%
\tilde{\mathcal{E}}_{m;l}(f_{0}))][\tilde{\mathcal{E}}_{m;l}(f_{0})+P_{m-1,l}(\tilde{\mathcal{E}}_{m;l}(f_{0}))] \\
&\equiv &P_{m,l}(\tilde{\mathcal{E}}_{m;l}(f_{0})).
\end{eqnarray*}%
This concludes the theorem for $f_{0}\in \testF.$ For a general
datum $f_{0}\in \mathcal{E}_{m;l}$ we can use a sequence of smooth
approximation $f_{0}^{k\text{ }}$ and take a limit$.$
\end{proof}

\appendix
\section{Time decay of certain integrals}\label{timeDecayInt}

In this appendix we give some basic time decay estimates of certain integrals.  
The first  lemma deals with the polynomial decay rate.

\begin{lemma}
\label{BasicDecay}
Suppose that $0<p \le 1$, $\lambda > 0$ and  $\mu \ge 0$.  Then
$$
\int_0^t  ds ~  e^{-\lambda \left(t^p- s^p  \right) }
(1+ s)^{-\mu}
\le
C(\lambda,\mu, p) (1+t)^{1-p}
(1+ t)^{-\mu},
$$
where $C(\lambda,\mu, p) >0$ only depends upon $\lambda$, $\mu$ and $p$.
\end{lemma}

\begin{remark}\label{sharpR}  Note further that (for $t\ge 1$) we have the following lower bound
\begin{multline}  \notag
\int_0^t  ds ~  e^{-\lambda \left(t^p- s^p  \right) }
(1+ s)^{-\mu}
\ge
\int_{t/2}^t  ds ~  e^{-\lambda \left(t^p- s^p  \right) }
(1+ s)^{-\mu}
\\
\ge
4^{1-p}(1+ t)^{-\mu}(1+t)^{1-p}
\int_{t/2}^t  ds ~  s^{p-1}e^{-\lambda \left(t^p- s^p  \right) }
\\
=
\frac{4^{1-p}}{\lambda p}
(1+ t)^{-\mu}(1+t)^{1-p}
\left(1 -  e^{-\lambda t^p\left(1 - 2^{-p}  \right) } \right).
\end{multline}  
We thus observe that the upper bound in Lemma \ref{BasicDecay} is asymptotically sharp.
\end{remark}

\noindent {\it Proof of Lemma \ref{BasicDecay}}. 
Without loss of generality it is safe to assume that $t \ge 1$.
 We split the integral as
\begin{gather*}
\int_0^t ds ~  
e^{-\lambda \left(t^p- s^p  \right) }
(1+ s)^{-\mu}
=
\int_0^{t/2}  
+
\int_{t/2}^t.
\end{gather*}
For the first integral we have the crude estimate of
\begin{gather*}
\int_0^{t/2}  ds
~  
e^{-\lambda \left(t^p- s^p  \right) }
(1+ s)^{-\mu}
\le
 t e^{-\lambda t^p \left(1 - 2^{-p}  \right) },
\end{gather*}
which will decay faster than any polynomial.

Now for the second integral we proceed as in Remark \ref{sharpR}:
\begin{multline}  \notag
\int_{t/2}^t  ds ~  e^{-\lambda \left(t^p- s^p  \right) }
(1+ s)^{-\mu}
\le
(1+ t)^{-\mu}(1+t)^{1-p}
\int_{t/2}^t  ds ~  s^{p-1}e^{-\lambda \left(t^p- s^p  \right) }
\\
=
\frac{1}{\lambda p}
(1+ t)^{-\mu}(1+t)^{1-p}
\left(1 -  e^{-\lambda t^p\left(1 - 2^{-p}  \right) } \right)
\\
\le
\frac{1}{\lambda p}
(1+ t)^{-\mu}(1+t)^{1-p}.
\end{multline}  
This completes the proof.
\qed

\begin{lemma}
\label{BasicDecayCor}
Suppose that $0<p \le 1$, $\lambda > 0$ and  $\mu \ge 0$.  
Then
$$
\int_0^t  ds ~  s^{p-1 }e^{-\lambda \left(t^p- s^p  \right) }
(1+ s)^{-\mu}
\le
C(\lambda,\mu, p) 
(1+ t)^{-\mu}
$$
where $C(\lambda,\mu, p) >0$ only depends upon $\lambda$, $\mu$ and $p$.
\end{lemma}

The proof of Lemma \ref{BasicDecayCor} follows exactly as in the proof of Lemma \ref{BasicDecay}.

\end{document}